\documentclass{mcom-l}
\usepackage[utf8]{inputenc}
\usepackage[english]{babel}
\usepackage{amsmath,amssymb,amsthm}
\usepackage{mathrsfs}
\usepackage{mathtools}
\usepackage{enumitem}
\usepackage{bm}
\usepackage{float}
\usepackage{xcolor}
\usepackage{subcaption}
\usepackage{hyperref}
\usepackage[nameinlink,capitalise,noabbrev]{cleveref}
\usepackage{todonotes}
\usepackage{siunitx}
\usepackage[export]{adjustbox}
\usepackage{pgfplots} \pgfplotsset{compat=1.18}
\usepgfplotslibrary{groupplots}
\definecolor{seaborngreen}{rgb}{0.3333333333333333, 0.6588235294117647, 0.40784313725490196}  
\definecolor{seaborncyan}{rgb}{0.39215686274509803, 0.7098039215686275, 0.803921568627451}  
\definecolor{seabornblue}{rgb}{0.2980392156862745, 0.4470588235294118, 0.6901960784313725}  
\definecolor{seabornpurple}{rgb}{0.5058823529411764, 0.4470588235294118, 0.6980392156862745}  
\definecolor{seabornred}{rgb}{0.7686274509803922, 0.3058823529411765, 0.3215686274509804}  
\definecolor{seabornorange}{rgb}{0.958, 0.476, 0.206}  
\definecolor{seabornsand}{rgb}{0.8, 0.7254901960784313, 0.4549019607843137}  
\pgfplotsset{colormap={paraview}{rgb255=(71,71,219) rgb255=(55,55,191) rgb255=(40,40,163) rgb255=(24,24,135) rgb255=(9,9,108) rgb255=(0,25,108) rgb255=(0,81,144) rgb255=(0,138,180) rgb255=(0,194,216) rgb255=(0,250,252) rgb255=(0,230,204) rgb255=(0,202,149) rgb255=(0,174,93) rgb255=(0,147,37) rgb255=(18,137,0) rgb255=(75,165,0) rgb255=(124,190,0) rgb255=(180,218,0) rgb255=(237,246,0) rgb255=(255,232,0) rgb255=(255,197,0) rgb255=(255,162,0) rgb255=(255,128,0) rgb255=(251,95,0) rgb255=(219,73,0) rgb255=(186,52,0) rgb255=(154,31,0) rgb255=(121,9,0) rgb255=(121,9,9) rgb255=(147,26,26) rgb255=(173,43,43) rgb255=(198,60,60) rgb255=(224,77,77)}}
\newcommand{\cbar}[4]{%
    \begin{tikzpicture}[baseline=(current bounding box.center)]
        \begin{axis}[hide axis, scale only axis, width=1pt, height=#3, xmin=0, xmax=1, ymin=0, ymax=1,
            colormap name=paraview, colorbar, point meta min=#1, point meta max=#2,
            colorbar style={at={(1,0)}, anchor=south west, xshift=0pt, height=#3, width=4pt, ytick={#4}, tick align=outside, major tick length=2pt,
                tick label style={font=\scriptsize}, yticklabel style={/pgf/number format/fixed, /pgf/number format/precision=2}}]
            \addplot[draw=none] coordinates {(0,0)};
        \end{axis}
    \end{tikzpicture}}
\usepackage{pgfplotstable}

\allowdisplaybreaks
\hypersetup{
    colorlinks,
    linkcolor={red!50!black},
    citecolor={blue!50!black},
    urlcolor={blue!80!black}
}
\def\NN{\mathbb{N}}
\def\RR{\mathbb{R}}
\def\LL{\bm{L}}
\newcommand{\pvec}[1]{\vec{#1}\mkern3mu\vphantom{#1}}

\def\M{\bm{M}}

\def\I{\bm{I}}
\def\u{\bm{u}}
\def\w{\bm{w}}
\def\Du{\mathrm{D}\u}
\def\vtheta{\vartheta}
\def\Th{\mathcal{T}_h}
\def\Vh{\mathcal{V}_h}
\def\Qh{\mathcal{Q}_h}
\def\Xh{\mathcal{X}_h}
\def\Uh{\mathcal{U}_h}
\def\Itau{\mathcal{I}_\tau}
\def\bsig{\bm{\sigma}}

\def\pa{\partial}
\def\na{\nabla}

\def\dt{\partial_t}
\def\dd{\mathrm{d}}

\def\Du{\mathrm{D}\u}
\def\divM{\operatorname{div}}
\def\ddt{\frac{\mathrm{d}}{\mathrm{d}t}}
\def\dttau{\partial_t^\tau}
\def\cv{c_v}

\def\la{\langle}
\def\ra{\rangle}

\DeclarePairedDelimiter{\norm}{\|}{\|}

\usepackage{letltxmacro}

\def\softd{{\leavevmode\setbox1=\hbox{d}%
		\hbox to 1.05\wd1{d\kern-0.4ex{\char039}\hss}}}

\DeclareTextAccent{\vaccent}{OT1}{20}
\def\v{\bm{v}}

\newtheorem{theorem}{Theorem}[section]
\newtheorem{lemma}[theorem]{Lemma}

\newtheorem{corollary}[theorem]{Corollary}

\theoremstyle{definition}

\newtheorem{example}[theorem]{Example}

\theoremstyle{remark}
\newtheorem{remark}[theorem]{Remark}

\numberwithin{equation}{section}

\usepackage{etoolbox}
\AtBeginEnvironment{lemma}{\crefalias{theorem}{lemma}}
\AtBeginEnvironment{proposition}{\crefalias{theorem}{proposition}}
\AtBeginEnvironment{corollary}{\crefalias{theorem}{corollary}}
\AtBeginEnvironment{definition}{\crefalias{theorem}{definition}}
\AtBeginEnvironment{example}{\crefalias{theorem}{example}}
\AtBeginEnvironment{remark}{\crefalias{theorem}{remark}}

\begin{document}

\title[Quasi-incompressible, thermal Maxwell--Stefan equations]{Structure-preserving discretisations for the quasi-incompressible, thermal Maxwell--Stefan equations}

\author[B.~D.~Andrews]{Boris D.~Andrews}
\address{Mathematical Institute, University of Oxford, Oxford OX1 3LB, United Kingdom}
\email{boris.andrews@maths.ox.ac.uk}
\thanks{B.D.A.\ was supported by the European Union [ERC, GeoFEM, 101164551].}

\author[A.~Baier-Reinio]{Aaron Baier-Reinio}
\address{Department of Engineering Science, University of Oxford, Oxford OX1 3PJ, United Kingdom}
\email{aaron.baier-reinio@eng.ox.ac.uk}
\thanks{A.B.R.\ was supported by a Clarendon scholarship from the University of Oxford.}

\author[A.~Brunk]{Aaron Brunk}
\address{Institute of Mathematics, Johannes Gutenberg University, 55099 Mainz, Germany}
\email{abrunk@uni-mainz.de}
\thanks{A.B.\ was supported by the German Research Foundation (DFG) via TRR~146, subproject C3, project number 233630050 and via SPP~2256 within the project ``Variational quantitative phase-field modeling and simulation of powder bed fusion additive manufacturing'', project number 441153493.}

\author[P.~E.~Farrell]{Patrick E.~Farrell}
\address{Mathematical Institute, University of Oxford, Oxford OX1 3LB, United Kingdom}
\email{patrick.farrell@maths.ox.ac.uk}
\thanks{P.E.F.\ was supported by the Engineering and Physical Sciences Research Council [grant no.~EP/W026163/1], the Science and Technology Facilities Council [grant no.~UKRI/ST/B000495/1], the Donatio Universitatis Carolinae Chair ``Mathematical modelling of multicomponent systems'', the UKRI Digital Research Infrastructure Programme through the Science and Technology Facilities Council's Computational Science Centre for Research Communities (CoSeC), the National Science Foundation under grant no.~DMS-1929284 while in residence at the Institute for Computational and Experimental Research in Mathematics in Providence, USA, and the Erwin Schr\"odinger International Institute for Mathematics and Physics, Vienna. For the purpose of open access, the authors have applied a CC~BY public copyright licence to any author accepted manuscript arising from this submission.}

\author[M.~Luk\'a\vaccent{c}ov\'a-Medvi{\softd}ov\'a]{M\'aria Luk\'a\vaccent{c}ov\'a-Medvi\softd ov\'a}
\address{Institute of Mathematics, Johannes Gutenberg University, 55099 Mainz, Germany}
\email{lukacova@uni-mainz.de}
\thanks{M.L.-M.\ gratefully acknowledges the support of DFG Project 5258 3336 funded within Focused Programme SPP 2410 ``Hyperbolic Balance Laws: Complexity, Scales and Randomness'' and of the Mainz Institute of Multiscale Modeling. She is also grateful to the Gutenberg Research College for supporting her research.}



\date{\today}

\begin{abstract}
We study a thermal Maxwell--Stefan--Navier--Stokes system modelling miscible mixtures under a quasi-incompressibility constraint. The system consists of the conservation laws for mass and momentum, and the balance law for entropy, which is preferred over the balance law for energy because it facilitates the development of two structure-preserving discretisation methods.
The first method applies a semi-implicit Euler scheme in time, together with suitable finite elements in space, which guarantees generation of entropy and enforces the quasi-incompressibility constraint pointwise, but numerically dissipates total energy. The second method employs a continuous Petrov--Galerkin discretisation in time, which also generates entropy and satisfies the constraint pointwise, and in addition conserves the total energy and allows high-order time approximations. The structure-preserving properties and behaviour of both numerical methods are demonstrated by a series of numerical experiments.
\end{abstract}

\makeatletter
\let\MakeUppercase\relax
\makeatother
\maketitle

\section{Introduction}

The model we consider describes the transport of a mixture of $N > 1$ distinct chemical species in a single phase (liquid or gas) via convection and diffusion, represented by partial mass densities $\vec\rho = (\rho_1,\ldots,\rho_N)$.
It incorporates a momentum balance for the barycentric velocity $\u$, and the evolution of the specific internal energy $e$.
Over a bounded Lipschitz domain $\Omega \subset \RR^d$, $d \in \{2, 3\}$,
\begin{subequations}\label{eq:asys}
\begin{align}
    \dt\rho_i + \divM(\rho_i\u) &= - \divM\bm{J}_i, \qquad i=1,\ldots,N, \label{eq:asys1}
    \\
    \dt (\varrho\u) + \divM(\varrho\u\otimes\u) &= \divM\bsig, \label{eq:asys2} 
    \\
    \dt (\varrho e) + \divM(\varrho e\u) &= -\divM\bm{J}_e + \bsig:\Du. \label{eq:asys3}
\end{align}
\end{subequations}
The total mass density is $\varrho\coloneqq\sum_{i=1}^N\rho_i$.
The system is closed by specifying constitutive relations for the diffusion fluxes $\bm{J}_1,\ldots,\bm{J}_N$, energy flux $\bm{J}_e$, and Cauchy stress tensor $\bsig\in \RR_{\mathrm{sym}}^{d\times d}$. 
In the analysis in this work we consider $\Omega = (0, 1)^d$ with periodic boundary conditions (topologically equivalent to a torus) for the sake of exposition and clarity;
the numerical results in \Cref{Sec:Simul} use non-periodic domains and inhomogeneous boundary conditions.

\subsection{Quasi-incompressibility}

Taking the specific volume $V_i$ of each component $i$ to be constant, we assume the mixture develops no void spaces, i.e.~the following constraint holds:
\begin{equation}\label{1.incom}
    \sum_{i=1}^N V_i\rho_i = 1.
\end{equation}
To observe how \eqref{1.incom} imposes a restriction on the velocity field, we multiply \eqref{eq:asys1} by $V_i$ and sum over $i=1,\ldots, N$:
\begin{subequations}
\begin{align}
    \dt\bigg(\sum_{i=1}^N V_i\rho_i\bigg) 
        + \divM\bigg(\u\sum_{i=1}^N V_i\rho_i\bigg) 
        + \divM\bigg(\sum_{i=1}^N V_i\bm{J}_i\bigg)
        &=  0,  \\
    \implies \divM\u 
        &= -\sum_{i=1}^N V_i\divM\bm{J}_i.  \label{1.VJ}
\end{align}
\end{subequations}
This characterises $\divM\u$ entirely by the diffusion fluxes $\bm{J}_i$: a quasi-incompressibility constraint. This is enforced by a Lagrange multiplier $p$, the fluid pressure.

\subsection{State equations}

To model the remaining thermodynamics, we specify an internal energy density $\varrho e$, typically a function of $\vec\rho$ and the entropy density $\varrho s$, and a Helmholtz free energy density $\varrho F$, typically a function of $\vec\rho$ and the temperature $\vtheta$.
These are related by the classical Legendre relation
\begin{subequations}\label{1.thermo}
\begin{equation}\label{1.thermo_a}
    \varrho e = \varrho F + \vtheta \varrho s.
\end{equation}
The entropy density $\varrho s$ and temperature $\vtheta$ can be determined from $\varrho e$ and $\varrho F$ as
\begin{equation}\label{1.thermo_b}
    \varrho s = - \, \partial_\vtheta (\varrho F)(\vec\rho, \vtheta),  \qquad
    \vtheta = \partial_{\varrho s} (\varrho e)(\vec\rho, \varrho s).
\end{equation}
\end{subequations}
Taking the partial derivative of $\varrho e = \varrho F + \vtheta \varrho s$ \eqref{1.thermo_a} with respect to $\rho_i$ for a fixed $\varrho s$, we find the following thermodynamic relation:
\begin{subequations}\label{eq:rhof_rhoe}
\begin{align}
    \partial_{\rho_i}(\varrho e)(\vec\rho, \varrho s)
        &=  \big(\partial_{\rho_i}(\varrho F)(\vec\rho, \vtheta) + \partial_{\vtheta}(\varrho F)(\vec\rho, \vtheta)\partial_{\rho_i}\vtheta(\vec\rho, \varrho s)\big)  \\
        &\qquad\qquad\qquad\qquad\qquad\qquad\qquad\qquad+ \partial_{\rho_i}\vtheta(\vec\rho, \varrho s) \varrho s,  \notag \\
        &=  \partial_{\rho_i}(\varrho F)(\vec\rho, \vtheta),
\end{align}
\end{subequations}
where in the second equality we use the definition of the entropy density $\varrho s = - \, \partial_\vtheta (\varrho F)(\vec\rho, \vtheta)$ \eqref{1.thermo_b}.
The chemical potentials $\mu_1,\ldots,\mu_N$, typically defined via $\mu_i = \pa_{\rho_i}(\varrho F)(\vec\rho, \vtheta)$, incorporate the Lagrange multiplier $p$, i.e.~the fluid pressure,
\begin{equation}\label{1.chem}
    \mu_i
        = \partial_{\rho_i}(\varrho F)(\vec\rho, \vtheta) + V_i p
        = \partial_{\rho_i}(\varrho e)(\vec\rho, \varrho s) + V_i p,  \qquad
        i = 1,\ldots,N,
\end{equation}
where we see these definitions are equivalent from \eqref{eq:rhof_rhoe}. 

\begin{example}[Georgiadis--J\"{u}ngel mixing energy] \label{ex:georgiadis_I}
    As a concrete example of such a free energy, we may consider the mixing energy from \cite{GeJu24},
    \begin{equation}\label{eq:georgiadis}
      \varrho F = \sum_{i=1}^N \vtheta\rho_i\log(\rho_i/\varrho) - \cv\varrho\vtheta(\log\vtheta-1),
    \end{equation}
    for a constant $\cv>0$. This results in an entropy $\varrho s = \cv\varrho\log(\vtheta)-\sum_{i=1}^N \rho_i\log(\rho_i/\varrho)$, an internal energy $\varrho e = \cv\varrho\vtheta$, and chemical potentials $\mu_i = \vtheta\log(\rho_i/\varrho) - \cv\vtheta(\log\vtheta-1) + V_i p$.
\end{example}

Following Onsager theory \cite{onsager1931a,onsager1931b}, each flux is modelled as a linear combination of the thermodynamic driving forces\footnote{
    The cross-coupling coefficients $c_i$ are not to be confused with the heat capacity $c_v$ in the Georgiadis--J\"{u}ngel mixing energy \eqref{eq:georgiadis}.
}: the gradients of the (i)~scaled chemical potentials $\tfrac{\mu_i}{\vtheta}$, and (ii)~inverse temperature $\tfrac{1}{\vtheta}$,
\begin{subequations}\label{eq:1.flux}
\begin{align}
    \bm{J}_i &= -\sum_{j=1}^N M_{ij}(\vec\rho,\vtheta)\na\frac{\mu_j}{\vtheta} + c_i(\vec\rho,\vtheta)\nabla\frac{1}{\vtheta}, \qquad
    i=1,\ldots,N, \\
    \bm{J}_e &= -\sum_{j=1}^N c_j(\vec\rho,\vtheta)\na\frac{\mu_j}{\vtheta} +K(\vec\rho,\vtheta)\nabla\frac{1}{\vtheta}.
\end{align}
\end{subequations}
To ensure the second law of thermodynamics is satisfied, the Maxwell--Stefan matrix $\bm{M}=(M_{ij})_{i,j=1}^N$ is assumed symmetric and positive semi-definite, and the heat conductivity $K$ is assumed positive.
The total mass conservation equation $\dt\varrho + \divM(\varrho\u) = 0$ requires the total flux $\sum_{i=1}^N \bm{J}_i = \bm{0}$, meaning $\bm{M}$ and $\bm{c}=(c_i)_{i=1}^N$ must satisfy $\sum_{i=1}^N M_{ij} = \sum_{i=1}^N c_i = 0$.
We denote by $\LL$ the full, symmetric, positive semi-definite Onsager matrix,
\begin{equation}
    \LL(\vec\rho, \vtheta) \coloneqq \begin{pmatrix}
    \bm{M}(\vec\rho, \vtheta) & - \, \bm{c}(\vec\rho, \vtheta) \\
    - \, \bm{c}(\vec\rho, \vtheta)^\top & K(\vec\rho, \vtheta)
    \end{pmatrix}.
\end{equation}
The negative off-diagonal blocks reflect the opposite signs with which the mass and energy fluxes enter the entropy production.

The Cauchy stress $\bsig$ decomposes into an isotropic part and a dissipative (viscous) part $\mathbb{S}$. Using the Gibbs--Duhem relation \cite{Bothe2014,Giovangigli1999}, we write
\begin{subequations}\label{eq:stress}
\begin{align}
    \bsig  &=  \bigg( \varrho F - \sum_{i=1}^N \rho_i \mu_i \bigg)\I + \mathbb{S}(\vec\rho,\vtheta;\Du),  \\
    \mathbb{S}(\vec\rho,\vtheta;\Du)  &=  \eta(\vec\rho,\vtheta)\Du + \lambda(\vec\rho,\vtheta)\underset{= \, \mathrm{tr}(\Du)}{\underbrace{\divM\u}}\I,
\end{align}
\end{subequations}
where $\Du \coloneqq \tfrac{1}{2}(\nabla\u + \nabla\u^\top)$ is the symmetric gradient of $\u$. 
The semicolon in $\mathbb{S}$ indicates linearity.

\begin{remark}[Onsager matrix]
    The general formulation for the fluxes is only given implicitly by, for each $i$,
    \begin{equation}
        - \, \rho_i\Big(\nabla\frac{\mu_i}{\vtheta} - \sum_{j=1}^N \frac{\rho_j}{\varrho}\nabla\frac{\mu_j}{\vtheta}\Big)
            =  \sum_{j=1}^N f_{ij}\Big(\frac{\rho_j}{\varrho}\bm{J}_i - \frac{\rho_i}{\varrho}\bm{J}_j + \frac{d_i\rho_j - d_j\rho_i}{\varrho}\vtheta \nabla\frac{1}{\vtheta} \Big),
    \end{equation}
    where $f_{ij}$ denote the Maxwell--Stefan frictions, and $d_i$ denote the thermal diffusivities accounting for cross-effects related to Soret and Dufour effects~\cite{Bothe2014,Bothe2023}.
    In general, inverting for the fluxes requires the specification of a reference velocity and a pseudo-inverse~\cite{vanbrunt2021a}.
    However, in the case of equal Maxwell--Stefan frictions, i.e.~$f_{ij}\equiv 1$, it follows that
    \begin{equation}
        \bm{J}_i
            = - \, \rho_i \Big(\nabla\frac{\mu_i}{\vtheta}
            - \sum_{j=1}^N \frac{\rho_j}{\varrho} \nabla\frac{\mu_j}{\vtheta}\Big)
            - \Big(\sum_{j=1}^N \frac{d_i\rho_j - d_j\rho_i}{\varrho}\Big) \vtheta\nabla\frac{1}{\vtheta}.
    \end{equation}
    The majority of the entries of the Onsager matrix $\LL$ can be read directly from this identity as
    \begin{subequations}
    \begin{equation}
        \LL_{\rho_i\rho_j}
            = M_{ij}
            = \delta_{ij}\rho_i - \frac{\rho_i\rho_j}{\varrho},  \quad
        \LL_{\rho_i\vtheta}
            = \LL_{\vtheta\rho_i}
            = - \, c_i
            = \vtheta \sum_{j=1}^N \frac{d_i\rho_j - d_j\rho_i}{\varrho}.
    \end{equation}
    The final entry can be given as
    \begin{equation}
        \LL_{\vtheta\vtheta} = K = K_{\textrm{ht}} + \alpha \sum_{i=1}^N c_i^2
    \end{equation}
    \end{subequations}
    for an explicitly known $\alpha \geq 0$ that depends on temperature and densities, cf.~\cite{vanbrunt2021a,Bothe2014,Bothe2023}, where $K_{\textrm{ht}}>0$ denotes the pure heat conductivity in the absence of thermodiffusive cross effects ($\bm{c} = \bm{0}$).
\end{remark}

\begin{remark}[Pressure]
    The Lagrange multiplier $p$ solved for does not necessarily coincide with either the mechanical pressure or the thermodynamic pressure (see Rajagopal \cite{rajagopal2015} for a discussion of these distinctions).
    For the present free-energy model, the thermodynamic pressure is
    \begin{equation}
        p_{\mathrm{thermo}} \coloneqq \sum_{i=1}^N\rho_i\mu_i-\varrho F,
    \end{equation}
    so the stress relation \eqref{eq:stress} reads $\bsig=-p_{\mathrm{thermo}}\I+\mathbb{S}$.
    The mechanical pressure is the mean normal stress,
    \begin{equation}
        p_{\mathrm{mech}}
        = -\frac{1}{d}\operatorname{tr}\bsig
        = p_{\mathrm{thermo}}-\frac{1}{d}\operatorname{tr}\mathbb{S}.
    \end{equation}
    Thus, the thermodynamic and mechanical pressures coincide precisely when the viscous stress has zero trace.
    In particular, this need not hold, since $\divM\u$ is generally non-zero (cf.~\eqref{1.VJ}).
    On the other hand, the chemical-potential relation \eqref{1.chem} and quasi-incompressibility constraint \eqref{1.incom} give
    \begin{multline}
        p_{\mathrm{thermo}}
        = \sum_{i=1}^N\rho_i \mu_i - \varrho F
        = \sum_{i=1}^N\rho_i \partial_{\rho_i}(\varrho F) + \sum_{i=1}^N \rho_i V_i p - \varrho F  \\
        = p + \bigg(\sum_{i=1}^N \rho_i \partial_{\rho_i}(\varrho F) - \varrho F\bigg).
    \end{multline}
    Hence, the thermodynamic pressure $p_{\mathrm{thermo}}$ coincides with the Lagrange multiplier $p$ only if $\varrho F(\vec\rho,\vtheta)$ is homogeneous of degree one in $\vec\rho$, i.e.~$\sum_{i=1}^N \rho_i \partial_{\rho_i}(\varrho F) - \varrho F = 0$.
    In particular, this is true for the free energy \eqref{eq:georgiadis} proposed by Georgiadis \& J\"ungel \cite{GeJu24}.
\end{remark}

Combining the balance laws \eqref{eq:asys} with the chemical-potential relation \eqref{1.chem} and the constraint \eqref{1.incom} gives the following closed partial differential--algebraic system:
\begin{subequations}\label{eq:fullsys}
\begin{align}
    \dt\rho_i + \divM(\rho_i\u) &= \divM\!\left(\sum_{j=1}^N M_{ij}\nabla\frac{\mu_j}{\vtheta} - c_i\nabla\frac{1}{\vtheta}\right)\!, \label{eq:1mass}
    \\
    \mu_i &= \partial_{\rho_i}(\varrho F)(\vec\rho, \vtheta) + V_ip, \label{eq:2chem}
    \\
    \dt (\varrho\u) + \divM(\varrho\u\otimes\u) &= \divM(\mathbb{S}(\vec\rho,\vtheta;\Du)) + \nabla\!\left(\varrho F - \sum_{i=1}^N\rho_i\mu_i\right)\!,  \label{eq:3mom}
    \\
    \divM\u &= \divM\!\left(\sum_{i,j=1}^N V_i M_{ij}\nabla\frac{\mu_j}{\vtheta} - \sum_{i=1}^N V_i c_i\nabla\frac{1}{\vtheta}\right)\!, \label{eq:4divu}
    \\
    \dt (\varrho e) + \divM(\varrho e\u) &= \divM\!\left( \sum_{j=1}^N c_j\nabla\frac{\mu_j}{\vtheta} - K\nabla\frac{1}{\vtheta}\right)\! + \mathbb{S}(\vec\rho,\vtheta;\Du):\Du  \label{eq:5temp}  \\
    &\qquad\qquad\qquad\qquad\qquad\quad  + \!\left(\varrho F - \sum_{i=1}^N\rho_i\mu_i\right)\!\divM\u,  \notag
\end{align}
\end{subequations}
where we use the flux \eqref{eq:1.flux} and stress \eqref{eq:stress} identities.

\subsection{Main results}

The objective is to design a structure-preserving finite element method for system \eqref{eq:fullsys}. Specifically, two numerical schemes are developed. Both schemes satisfy the quasi-incompressibility constraint pointwise \eqref{1.incom}, conserve the partial mass of each species (and with it the total mass), and preserve the generation of entropy. The first, a semi-implicit Euler scheme (i.e.~a mixed explicit-implicit Euler method), dissipates total energy, while the second, a continuous Petrov--Galerkin (CPG) scheme, ensures total energy conservation and extends to arbitrary order. The first scheme is simpler and cheaper than the more complicated and expensive CPG scheme.

\subsection{State of the art}

The above system represents an extension of the quasi-incompressible isothermal mixture model presented by Druet \cite{Dru21}. For the isothermal case, existence results (i.e.\ for general $V_i>0$) were presented in \cite{BoDr21,Dru21}. We are not aware of any analytical treatment of the thermal system presented here. However, there are already some results for the fully compressible system. For such a system, existence in the stationary case is shown in \cite{Bulicek22}, while the non-stationary problem is considered in \cite{Mucha15,Bulicek17}.

The literature is much richer for the Maxwell--Stefan equations with vanishing barycentric velocity.
This helps to simplify the analysis, since it implies that the governing equations do not need to be
coupled to a momentum conservation law such as the Navier--Stokes equations.
Existence of generalised solutions was proven in \cite{Bot11,GeJu24,JuSt13} for the isothermal setting,
while thermal flows were analysed in \cite{HeJu21,HuSa18}.

In the numerical analysis literature on Maxwell--Stefan or related cross-diffusion flows, it is also frequently assumed that
the barycentric velocity vanishes.
Structure-preserving numerical methods for cross-diffusion equations with vanishing barycentric velocity are considered in
\cite{BPS22,CEM24,HLTW21,JuZu23}.
For numerical analysis literature on Maxwell--Stefan equations coupled to Navier--Stokes (or Stokes) flow,
we are only aware of the finite element methods in
\cite{AFMV25,Baier25,baier2025finite,Brunk2026,burman2003bunsen,longo2012finite},
and a series of finite volume schemes from Bell and collaborators
(see e.g.~\cite{donev2015low,srivastava2023staggered}).
However, apart from \cite{Brunk2026}, none of these works considers the preservation of energetic or entropic structures.
Hence, this paper can be viewed as a substantial extension of \cite{Brunk2026}, where we now consider the possibility of
exact energy preservation in the setting of thermal, quasi-incompressible flow.

\section{Notation and preliminaries}

Vectors in $\RR^N$ collecting variables for each species are denoted by an arrow, as in $\vec{v}$.
Vectors in $\RR^d$, e.g.~velocity vectors and multiples thereof, are represented with boldface $\bm{v}$.
The $L^2(\Omega)$ inner product is written as $\langle\cdot,\cdot\rangle$. We note that
\begin{equation}
    \mathbb{S}(\vec\rho, \vtheta; \Du) : \Du
        =  \eta(\vec\rho, \vtheta) \|\Du\|^2 + \lambda(\vec\rho, \vtheta)(\divM\u)^2,
\end{equation}
where $\|a\|$ denotes the standard Euclidean norm for vectors and the Frobenius norm for matrices.
Assuming $\eta(\vec\rho, \vtheta)$ is uniformly bounded below and $d\lambda(\vec\rho, \vtheta) +\eta(\vec\rho, \vtheta)\geq 0$, the viscous dissipation is coercive in $\bm{H}^1(\Omega)$, modulo the kernel of the symmetric gradient, by Korn's inequality.
It is assumed that integrable, finite-energy initial data at time $t = 0$ are given for the partial densities $\vec\rho$, momentum $\varrho\u$, and temperature $\vtheta$, satisfying the quasi-incompressibility constraint $\sum_{i=1}^N V_i\rho_i^0 = 1$ \eqref{1.incom}.

The following lemma bounds the change in internal energy density between two thermodynamic states and will be used in the analysis of our proposed semi-implicit Euler scheme \eqref{eq:schemeA}.

\begin{lemma}[Bound on change in energy density]\label{lem:rhoe_bound}
    Suppose the Helmholtz free energy $\varrho F$ is (i)~convex in $\vec\rho$ for fixed $\vtheta$, and (ii)~concave in $\vtheta$ for fixed $\vec\rho$.
    Then, for two distinct thermodynamic states indexed by $*_\pm$, we may bound
    \begin{equation}\label{eq:rhoe_bound}
        (\varrho e)_+ - (\varrho e)_-
            \le  \sum_{i=1}^N \partial_{\rho_i}(\varrho F)(\vec\rho_+, \vtheta_+) \big[\rho_{i+} - \rho_{i-}\big] + \vtheta_+ \big[(\varrho s)_+ - (\varrho s)_-\big].
    \end{equation}
\end{lemma}

\begin{proof}
    From the Legendre relation for $\varrho e$ and $\varrho F$ \eqref{1.thermo_a},
    \begin{equation}\label{eq:rhoe_bound_intermediate_a}
        (\varrho e)_+ - (\varrho e)_-
            = \big[(\varrho F)_+ - (\varrho F)_-\big] + \big[\vtheta_+ (\varrho s)_+ - \vtheta_- (\varrho s)_-\big].
    \end{equation}
    We bound $(\varrho F)_+ - (\varrho F)_-$ by passing through an intermediate state $(\vec\rho_-, \vtheta_+)$:
    \begin{subequations}\label{eq:rhoe_bound_intermediate_b}
    \begin{align}
        (\varrho F)_+ - (\varrho F)_-
            &= \varrho F(\vec\rho_+, \vtheta_+) - \varrho F(\vec\rho_-, \vtheta_-)  \\
            &= \big[\varrho F(\vec\rho_+, \vtheta_+) - \varrho F(\vec\rho_-, \vtheta_+)\big]  \\
            &\qquad\qquad\qquad\qquad\qquad+ \big[\varrho F(\vec\rho_-, \vtheta_+) - \varrho F(\vec\rho_-, \vtheta_-)\big].  \notag
    \end{align}
    \end{subequations}
    For the first bracket, we use the convexity of $\varrho F$ in $\vec\rho$:
    \begin{subequations}
    \begin{equation}
        \varrho F(\vec\rho_+, \vtheta_+) - \varrho F(\vec\rho_-, \vtheta_+)
            \le  \sum_{i=1}^N \partial_{\rho_i} (\varrho F)(\vec\rho_+, \vtheta_+)\big[\rho_{i+} - \rho_{i-}\big].
    \end{equation}
    For the second, we use the concavity of $\varrho F$ in $\vtheta$:
    \begin{multline}
        \varrho F(\vec\rho_-, \vtheta_+) - \varrho F(\vec\rho_-, \vtheta_-)
            \le  \partial_\vtheta (\varrho F)(\vec\rho_-, \vtheta_-)\big[\vtheta_+ - \vtheta_-\big]  \\
            =  - \, (\varrho s)_- \big[\vtheta_+ - \vtheta_-\big].
    \end{multline}
    \end{subequations}
    Substituting these bounds into \eqref{eq:rhoe_bound_intermediate_a} and \eqref{eq:rhoe_bound_intermediate_b},
    \begin{multline}
        (\varrho e)_+ - (\varrho e)_-
            \le  \sum_{i=1}^N \partial_{\rho_i}(\varrho F)(\vec\rho_+, \vtheta_+)\big[\rho_{i+} - \rho_{i-}\big]  \\
            - (\varrho s)_- \big[\vtheta_+ - \vtheta_-\big]
            + \big[\vtheta_+ (\varrho s)_+ - \vtheta_- (\varrho s)_-\big].
    \end{multline}
    Collapsing like terms gives the desired result.
\end{proof}

\section{Entropy balance and global structures}

For smooth solutions, the equation for internal energy can equivalently be replaced by the equation for total energy or entropy. The system is reformulated in terms of the entropy equation to facilitate numerical approximation; cf.~\cite{Hoehn26} for a similar approach to the thermal Cahn--Hilliard equation.

\begin{lemma}[Entropy balance equation]
    For smooth solutions of \eqref{eq:fullsys},
    \begin{equation}\label{eq:6entropy}
        \dt(\varrho s) + \divM(\varrho s\u)
            = - \divM\bm{J}_s + \mathcal{P}(\vec\rho, \vtheta, \vec\mu, \Du),
    \end{equation}
    where the entropy flux $\bm{J}_s$ is defined as
    \begin{equation}
        \bm{J}_s  \coloneqq  -\sum_{i=1}^N \frac{\mu_i}{\vtheta} \bm{J}_i + \frac{1}{\vtheta}\bm{J}_e,
        \label{eq:entropy_flux_def}
    \end{equation}
    and $\mathcal{P}$ denotes the entropy production rate
    \begin{equation}\label{eq:calP}
        \mathcal{P}(\vec\rho, \vtheta, \vec\mu, \Du)
            \coloneqq
            \frac{1}{\vtheta}\mathbb{S}(\vec\rho,\vtheta;\Du):\Du
            + \begin{pmatrix}
                \nabla\frac{\vec{\mu}}{\vtheta} \\ \\ \nabla\frac{1}{\vtheta}
            \end{pmatrix}^\top \!\!\!\LL(\vec\rho, \vtheta) \begin{pmatrix}
            \nabla\frac{\vec{\mu}}{\vtheta} \\  \\\nabla\frac{1}{\vtheta}
            \end{pmatrix}
            \ge  0.
    \end{equation}
\end{lemma}

\begin{proof}
    Denote $\tfrac{D}{Dt}\phi = \dt \phi + \u\cdot\nabla\phi$, for a general scalar field $\phi$.
    Recalling that $\varrho e = \varrho F + \vtheta \varrho s$ \eqref{1.thermo}, applying the convective derivative $\tfrac{D}{Dt}$ gives
    \begin{equation}
        \frac{D}{Dt}(\varrho e)
            =  \frac{D}{Dt}(\varrho F) +  \varrho s \frac{D}{Dt}\vtheta + \vtheta \frac{D}{Dt}(\varrho s).
    \end{equation}
    The differential $\tfrac{D}{Dt}(\varrho F)$ reduces to
    \begin{equation}
        \frac{D}{Dt}(\varrho F)
            =  \sum_{i=1}^N \partial_{\rho_i}(\varrho F)\frac{D}{Dt}\rho_i + \partial_\vtheta(\varrho F) \frac{D}{Dt}\vtheta
            =  \sum_{i=1}^N (\mu_i - V_i p)\frac{D}{Dt}\rho_i - \varrho s \frac{D}{Dt}\vtheta,
    \end{equation}
    using \eqref{1.chem} and \eqref{1.thermo} respectively, thus implying
    \begin{subequations}
    \begin{align}
        \frac{D}{Dt}(\varrho e)
            &= \sum_{i=1}^N (\mu_i - V_i p)\frac{D}{Dt}\rho_i + \vtheta\frac{D}{Dt}(\varrho s),
    \end{align}
    and expanding material derivatives we have
    \begin{align}
        \dt (\varrho e) \, + \, &\divM(\varrho e \u)  \\
            &= \sum_{i=1}^N (\mu_i - V_i p)(\dt \rho_i + \divM(\rho_i\u)) + \vtheta(\dt (\varrho s) + \divM(\varrho s \u))  \notag  \\
            &\qquad\qquad\qquad\qquad\qquad+ \!\left(\varrho e - \sum_{i=1}^N (\mu_i - V_i p)\rho_i - \vtheta \varrho s\right)\!\divM\u.  \notag
    \end{align}
    \end{subequations}
    Using \eqref{1.incom} and \eqref{1.thermo}, the final term simplifies to $(\varrho F - \sum_{i=1}^N \rho_i \mu_i + p)\divM\u$.
    Substituting in the evolution equations for $\rho_i$ \eqref{eq:asys1} and $\varrho e$ \eqref{eq:asys3},
    \begin{multline}
        - \divM\bm{J}_e + \bsig : \Du
            = -\sum_{i=1}^N (\mu_i - V_i p)\divM\bm{J}_i + \vtheta(\dt (\varrho s)  \\
            + \divM(\varrho s \u)) + \!\left(\varrho F - \sum_{i=1}^N \rho_i \mu_i + p\right)\!\divM\u.
            \label{eq:entropy_derivation_eq}
    \end{multline}
    From the definition of the Cauchy stress $\bsig$ \eqref{eq:stress} and the quasi-incompressibility condition \eqref{1.VJ},
    \begin{equation}
        -\divM\bm{J}_e + \mathbb{S}(\vec\rho, \vtheta; \Du) : \Du
            =  -\sum_{i=1}^N \mu_i \divM\bm{J}_i + \vtheta(\dt (\varrho s) + \divM(\varrho s \u)).
    \end{equation}
    Rearranging,
    \begin{equation}
        \sum_{i=1}^N \frac{\mu_i}{\vtheta} \divM\bm{J}_i - \frac{1}{\vtheta}\divM\bm{J}_e + \frac{1}{\vtheta}\mathbb{S}(\vec\rho, \vtheta; \Du) : \Du
            = \dt (\varrho s) + \divM(\varrho s \u).
    \end{equation}
    Writing $\phi\divM\bm{v} = \divM(\phi\bm{v}) - \nabla\phi \cdot \bm{v}$ for each term on the left-hand side and using the definitions for the mass and energy fluxes \eqref{eq:1.flux}, the claimed identity \eqref{eq:6entropy} holds.
\end{proof}

Given the above reformulation, we are able to show conservation of partial masses and total energy, and non-negativity of entropy production.

\begin{lemma}[Global structures]\label{lem:struc}
    For smooth solutions to \eqref{eq:fullsys},
    \begin{subequations}\label{eq:struc}
    \begin{align}
        \frac{\dd}{\dd t}\int_\Omega \rho_i \dd x  &= 0, \qquad i = 1, \ldots, N,  \label{eq:struc_mass}  \\
        \frac{\dd}{\dd t}\int_\Omega \left(\varrho e(\vec\rho,\vtheta) + \tfrac{\varrho}{2}\|\bm{u}\|^2\right)\dd x  &= 0,  \label{eq:struc_energy}  \\
        \frac{\dd}{\dd t}\int_\Omega \varrho s(\vec\rho,\vtheta)\dd x  &=  \int_\Omega \mathcal{P}(\vec\rho, \vtheta, \vec\mu, \Du)\,\dd x \ge 0. \label{eq:struc_entropy}
    \end{align}
    \end{subequations}
\end{lemma}

\begin{proof}
    The first and last identities hold by integration by parts on \eqref{eq:1mass} and \eqref{eq:6entropy} respectively.
    The remaining identity follows by summing \eqref{eq:1mass}, testing the resulting total-mass equation with $\tfrac{1}{2}\|\u\|^2$, testing \eqref{eq:3mom} with $\u$, and integrating \eqref{eq:5temp} over the domain.

\end{proof}

\section{Variational formulations}\label{sec.2}

To construct our variational formulations of \eqref{eq:fullsys}, we replace the energy balance equation \eqref{eq:5temp} with that of entropy balance \eqref{eq:6entropy}.
The energy equation could alternatively be used;
for discretisation, employing the inverse temperature may be advantageous \cite{Brunk25Enu,Brunk2025_noniso}.

For simplicity, we henceforth consider the case $c_i = 0$.
The extension to non-zero $c_i$ is straightforward.

\subsection{Formulation I}

Testing the momentum convection term $\divM(\varrho\u\otimes\u)$ against a vector $\bm{v}$ and applying integration by parts,
\begin{subequations}
\begin{align}
    \la\divM(\varrho\u\otimes\u), \bm{v}\ra
    &= \frac12\la\divM(\varrho\u\otimes\u), \bm{v}\ra
    - \frac12\la\varrho\u\otimes\u, \na\bm{v}\ra  \\
    &= \frac12\la\divM(\varrho\u)\u, \bm{v}\ra + \frac12\la(\varrho\u\cdot\na)\u, \bm{v}\ra
    - \frac12\la(\varrho\u\cdot\na)\bm{v}, \u\ra.
\end{align}
\end{subequations}
Accordingly, we introduce the convective trilinear form
\begin{equation}
    \mathcal{B}(\varrho\bm{u};\bm{v},\bm{w})
    \coloneqq \frac12\la(\varrho\bm{u}\cdot\na)\bm{v},\bm{w}\ra
    - \frac12\la(\varrho\bm{u}\cdot\na)\bm{w},\bm{v}\ra.
\end{equation}
This is skew-symmetric in $\bm{v}$, $\bm{w}$, i.e.~$\mathcal{B}(\varrho\bm{u};\bm{v},\bm{w}) = - \, \mathcal{B}(\varrho\bm{u};\bm{w},\bm{v})$.
Total mass conservation lets us write $\divM(\varrho\u) = - \, \dt \varrho$, implying
\begin{subequations}\label{eq:weak3_a}
\begin{align}
    \la\divM(\varrho\u\otimes\u), \bm{v}\ra
    &= - \, \frac12\la\dt \varrho\u, \bm{v}\ra + \mathcal{B}(\varrho\u; \u, \bm{v}),  \\
    \la\dt (\varrho\u), \bm{v}\ra + \la\divM(\varrho\u\otimes\u), \bm{v} \ra
    &= \!\left\la\! \varrho\dt \u + \frac12\dt \varrho\u, \bm{v} \!\right\ra\! + \mathcal{B}(\varrho\u; \u, \bm{v}).
\end{align}
\end{subequations}

Considering $\nabla(\varrho F - \sum_{i=1}^N\rho_i\mu_i)$, the momentum forcing term in \eqref{eq:3mom} from the conservative component of the stress $\bsig$, we substitute $\nabla(\varrho F)$ through the thermodynamic relations (\ref{1.thermo},~\ref{1.chem}),
\begin{subequations}
\begin{align}
    \nabla\!\left(\varrho F - \sum_{i=1}^N\rho_i\mu_i\right)\!
        &=  - \, \varrho s \nabla \vtheta + \sum_{i=1}^N (\mu_i - V_i p) \nabla \rho_i - \sum_{i=1}^N \nabla(\rho_i\mu_i)  \\
        &= - \, \varrho s \nabla \vtheta - \sum_{i=1}^N \rho_i\nabla\mu_i + \sum_{i=1}^N V_i \rho_i \nabla p - \nabla \!\left( \sum_{i=1}^N V_i \rho_i p \right)\!  \\
        &=  - \, \varrho s \nabla \vtheta - \sum_{i=1}^N \rho_i\nabla\mu_i + \sum_{i=1}^N V_i \rho_i \nabla p - \nabla p, \label{eq:consstressform}
\end{align}
\end{subequations}
where in the last identity we apply quasi-incompressibility \eqref{1.incom}\footnote{The final identity \eqref{eq:consstressform} can be cancelled further---the final two terms may both be eliminated---however this form is taken with hindsight to ensure the desired stability results hold.}.
Testing against $\bm{v}$ yields
\begin{equation}\label{eq:weak3_b}
    \!\left\la\! \nabla\!\left(\varrho F - \sum_{i=1}^N\rho_i\mu_i\right)\!, \bm{v} \!\right\ra
        =  \left\la \sum_{i=1}^N \rho_i (V_i \nabla p - \nabla\mu_i) - \varrho s \nabla \vtheta, \bm{v} \!\right\ra\! + \la p, \divM\bm{v} \ra
\end{equation}
by integration by parts.
Combining \eqref{eq:weak3_a} and \eqref{eq:weak3_b}, we arrive at our first variational formulation.

\begin{lemma}[First variational formulation]\label{lem:Eulerform}
    Smooth solutions $(\vec\rho, \vec\mu, \bm{u}, p, \vtheta)$ to \eqref{eq:fullsys} satisfy
    \begin{subequations}\label{eq:weakA}
    \begin{align}
        \la \dt\rho_i,\psi_i \ra - \la \rho_i\u,\nabla\psi_i \ra  &=  - \sum_{j=1}^N \!\left\la\! M_{ij}\nabla \frac{\mu_j}{\vtheta}, \nabla\psi_i \!\right\ra,  \label{eq:weak1}
        \\
        \la \mu_i,\zeta_i \ra  &=  \la \partial_{\rho_i} (\varrho F)(\vec\rho, \vtheta) + V_i p, \zeta_i \ra,  \label{eq:weak2}
        \\
        \left\la \varrho\dt\u + \frac12\dt\varrho\u, \v \!\right\ra + \mathcal{B}(\varrho&\u; \u, \v)  \label{eq:weak3}  \\
        &= - \, \la \mathbb{S}(\vec\rho,\vtheta;\mathrm{D}\u),\mathrm{D}\v \ra + \la p,\divM\v \ra    \notag
        \\  
        &\qquad\qquad - \, \!\left\la \sum_{i=1}^N \rho_i(\nabla\mu_i - V_i \na p)+ \varrho s\nabla\vtheta, \v \!\right\ra\!,  \notag
        \\
        \la \divM\u,q \ra  &=  - \, \!\left\la \sum_{i,j=1}^N V_i M_{ij}\nabla\frac{\mu_j}{\vtheta}, \nabla q \!\right\ra\!,  \label{eq:weak4}
        \\
        \la \dt (\varrho s),\omega \ra - \la \varrho s\u,\nabla\omega \ra  &=  \left\la \mathbb{S}(\vec\rho,\vtheta;\Du):\mathrm{D}\u,\frac{\omega}{\vtheta} \right\ra + \left\la\! K\nabla\frac{1}{\vtheta},\nabla\frac{\omega}{\vtheta} \right\ra\!  \label{eq:weak5}  \\
        &\qquad\qquad + \sum_{i,j=1}^N \!\left\la M_{ij}\nabla\frac{\mu_j}{\vtheta},\nabla\frac{\mu_i\omega}{\vtheta} \right\ra\!,  \notag
    \end{align}
    \end{subequations}
    for all smooth test functions $(\psi_i, \zeta_i, \v, q, \omega)$, $i=1,\ldots,N$, where the entropy density $\varrho s(\vec\rho, \vtheta)$ is determined as an explicit function of $\vec\rho$, $\vtheta$.
\end{lemma}

\begin{proof}
    Equations (\ref{eq:weak1},~\ref{eq:weak2},~\ref{eq:weak4},~\ref{eq:weak5}) follow directly from (\ref{eq:1mass},~\ref{eq:2chem},~\ref{eq:4divu},~\ref{eq:6entropy}) respectively, after multiplication with a test function and integration by parts.
    The momentum equation \eqref{eq:weak3} follows from \eqref{eq:3mom} after application of the identities (\ref{eq:weak3_a},~\ref{eq:weak3_b}) shown above.
\end{proof}

The following lemma shows that this variational formulation encodes the structural identities using only simple test functions.

\begin{lemma}[Test functions in first formulation]\label{lem:weak_test}
    Solutions of \eqref{eq:weakA} can be seen to fulfil the identities of \Cref{lem:struc} by considering the following test functions:
    \begin{itemize}
        \item Conservation of partial mass $\int_\Omega \rho_i$ \eqref{eq:struc_mass}: $\psi_i = 1$
        \item Conservation of energy $\int_\Omega (\frac12 \varrho \|\u\|^2 + \varrho e)$ \eqref{eq:struc_energy}: $\psi_i = \mu_i-V_i p$, $\zeta_i = \dt\rho_i$, $\v = \u$, $q = p$, $\omega = \vtheta$
        \item Generation of entropy $\int_\Omega \varrho s$ \eqref{eq:struc_entropy}: $\omega = 1$
    \end{itemize}
\end{lemma}

\begin{proof}
    The mass and entropy structures are immediate.
    For energy, we begin with the kinetic component,
    \begin{subequations}\label{eq:weak_test_proof_1}
    \begin{align}
        \frac{\dd}{\dd t} \int_\Omega \frac12 \varrho \|\u\|^2
            &=  \int_\Omega \varrho \u \cdot \partial_t\u + \frac12 \partial_t \varrho \|\u\|^2  \\
            &=  \int_\Omega - \, \mathbb{S}(\vec\rho,\vtheta;\mathrm{D}\u) : \mathrm{D}\u + p \divM\u  \\
            &\qquad\qquad\qquad\qquad  - \sum_{i=1}^N \rho_i (\nabla\mu_i - V_i \nabla p) \cdot \u - \varrho s \nabla \vtheta\cdot\u,  \notag
    \end{align}
    \end{subequations}
    where in the second equality we consider $\v = \u$ in \eqref{eq:weak3}, eliminating $\mathcal{B}(\varrho\u; \u, \u)$ through skew-symmetry.
    For the divergence term $p \divM\u$, consider $q = p$ in \eqref{eq:weak4},
    \begin{equation}\label{eq:weak_test_proof_2}
        \int_\Omega p \divM\u
            =  - \, \int_\Omega \sum_{i,j=1}^N V_i M_{ij}\nabla\frac{\mu_j}{\vtheta} \cdot \nabla p.
    \end{equation}

    For the internal component,
    \begin{subequations}
    \begin{align}
        \frac{\dd}{\dd t} \int_\Omega \varrho e(\vec\rho, \varrho s)
            &= \int_\Omega \sum_{i=1}^N \partial_{\rho_i}(\varrho e)(\vec\rho, \varrho s) \partial_t \rho_i + \partial_{\varrho s}(\varrho e)(\vec\rho, \varrho s) \partial_t (\varrho s)  \label{eq:weak_test_proof_2.5a}  \\
            &= \int_\Omega \sum_{i=1}^N \partial_{\rho_i}(\varrho F)(\vec\rho, \vtheta) \partial_t \rho_i + \vtheta \partial_t (\varrho s)  \label{eq:weak_test_proof_2.5b},
    \end{align}
    \end{subequations}
    where we substitute $\partial_{\rho_i}(\varrho F)(\vec\rho, \vtheta) = \partial_{\rho_i}(\varrho e)(\vec\rho, \varrho s)$ and $\vtheta = \partial_{\varrho s}(\varrho e)(\vec\rho, \varrho s)$ by the thermodynamic relations (\ref{eq:rhof_rhoe},~\ref{1.thermo_b}).
    For the former term,
    \begin{subequations}\label{eq:weak_test_proof_3}
    \begin{align}
        \int_\Omega \partial_{\rho_i}(\varrho e)(\vec\rho, \varrho s) \partial_t \rho_i
            &=  \int_\Omega (\mu_i - V_i p)\partial_t \rho_i  \\
            &=  \int_\Omega \rho_i \u \cdot (\nabla \mu_i - V_i \nabla p)  \\
            &\qquad\qquad\qquad\qquad - \sum_{j=1}^N M_{ij} \nabla\frac{\mu_j}{\vtheta}\cdot (\nabla \mu_i - V_i \nabla p),  \notag
    \end{align}
    \end{subequations}
    where the first equality considers $\zeta_i = \partial_t \rho_i$ in \eqref{eq:weak2}, and the second considers $\psi_i = \mu_i - V_i p$ in \eqref{eq:weak1}.
    For the latter term, consider $\omega = \vtheta$ in \eqref{eq:weak5},
    \begin{equation}\label{eq:weak_test_proof_4}
        \int_\Omega \vtheta \partial_t (\varrho s)
            =  \int_\Omega \varrho s \u \cdot \nabla \vtheta + \mathbb{S}(\vec\rho,\vtheta;\mathrm{D}\u) : \mathrm{D}\u + \sum_{i,j=1}^N M_{ij} \nabla\frac{\mu_j}{\vtheta} \cdot \nabla\mu_i.
    \end{equation}
    Summing (\ref{eq:weak_test_proof_1},~\ref{eq:weak_test_proof_2},~\ref{eq:weak_test_proof_3},~\ref{eq:weak_test_proof_4}), all right-hand side terms cancel, giving the desired result:
    \begin{equation}
        \frac{\dd}{\dd t} \int_\Omega \varrho e_{\mathrm{tot}}(\vec\rho, \varrho s, \u)
            =  \frac{\dd}{\dd t} \int_\Omega \left(\frac12\varrho\|\u\|^2 + \varrho e(\vec\rho, \varrho s)\right)
            =  0.
    \end{equation}
\end{proof}

\begin{example}[Georgiadis--J\"{u}ngel mixing energy]\label{ex:georgiadis_II}
    In the construction of the above variational formulation \eqref{eq:weakA}, we require the entropy density $\varrho s(\vec\rho, \vtheta)$ used in \eqref{eq:weak5} to be written as an explicit function of $\vec\rho$ and $\vtheta$.
    As an example of how this might proceed in practice, we continue with \Cref{ex:georgiadis_I}:
    \begin{equation}\label{eq:georgiadis_F}
        \varrho F
            =  \sum_{i=1}^N \vtheta\rho_i\log(\rho_i/\varrho) - \cv\varrho\vtheta(\log\vtheta-1).
    \end{equation}
    The thermodynamic relation $\varrho s = - \, \partial_\vtheta (\varrho F)(\vec\rho, \vtheta)$ \eqref{1.thermo_b} determines $\varrho s$ directly as
    \begin{equation}\label{eq:georgiadis_entropy}
        \varrho s
            =  \cv\varrho\log(\vtheta) - \sum_{i=1}^N \rho_i\log(\rho_i/\varrho).
    \end{equation}
    The second equation \eqref{eq:weak2} also uses the partial derivative
    \begin{equation}\label{eq:georgiadis_drhoF}
        \partial_{\rho_i}(\varrho F)(\vec\rho, \vtheta)
            =  \vtheta\log(\rho_i/\varrho) - \cv\vtheta(\log\vtheta - 1),
    \end{equation}
    which is equivalently $\mu_i - V_i p$, with the chemical potential $\mu_i$ as evaluated in \Cref{ex:georgiadis_I}.
\end{example}

The first variational formulation \eqref{eq:weakA} lends itself to a semi-implicit Euler discretisation in time that guarantees entropy generation and enforcement of the quasi-incompressibility constraint.
However, it will also dissipate total energy, which is conserved at the continuous level.

\begin{remark}[Scale of the thermodynamic forcing]\label{rem:Me}
    In nearly incompressible settings, the thermodynamic forcing introduced in \eqref{eq:weak3},
    \begin{equation}\label{eq:thermo_forcing}
        \dots - \left\la \sum_{i=1}^N \rho_i(\nabla\mu_i - V_i \na p)+ \varrho s\nabla\vtheta, \v \!\right\ra\!,
    \end{equation}
    can dominate the momentum balance.
    After non-dimensionalisation, it is multiplied by a dimensionless parameter $\textrm{La}$, sometimes called
    a Lagrange number, which quantifies the ratio between the
    pressure forces and viscous forces.
    For certain liquid mixtures, this factor can be very large, e.g.~for the benzene--cyclohexane mixture of \Cref{sec:app_ex}, $\textrm{La} \approx 10^{15}$ \eqref{eq:pressure_diffusion_number}.
    When $\varrho F$ is positively homogeneous of degree one in $\vec\rho$ (e.g.~with the above Georgiadis--J\"{u}ngel mixing energy), the term \eqref{eq:thermo_forcing} evaluates to zero on the continuous level, and can thus be removed, leaving a consistent discretisation of the Maxwell--Stefan system \eqref{eq:fullsys}.
    However, the energy stability result of \Cref{lem:weak_test} no longer holds under such a modification\footnote{
        The proof relies on \eqref{eq:thermo_forcing}, with $\v = \u$, cancelling with the advective terms in the mass and entropy equations.
        Conservation of mass, generation of entropy and quasi-incompressibility are unaffected.
    }, a failure which carries through to the energy stability of the fully discrete scheme below (see \Cref{thm:schemeA_global}).
    The same modification implies a similar energy stability failure of our second proposed scheme \eqref{eq:weakB} below (see \Cref{lem:weak_test2,thm:schemeB_global}).
\end{remark}

\subsection{Formulation II}

Recently, a subset of the authors proposed a general strategy for enforcing conservation laws and dissipation inequalities on discretisation \cite{andrews2025a,andrews2025b}.
Conservation and dissipation laws can be derived at the continuous level by testing the equations with suitable test functions;
the strategy introduces auxiliary variables to serve as discrete projections of these test functions.
These auxiliary variables are then valid choices of test functions at the discrete level, recovering the conservation and dissipation laws.
The resulting continuous formulation, additionally introducing a weighted velocity $\bm{m} = \sqrt{\varrho}\u$ and internal entropy density $\varrho s$, is given below.
When suitably discretised with a mixed continuous Petrov--Galerkin method, it enables a discretisation that generates entropy, conserves the total energy, and allows for high-order time approximations.

\begin{lemma}[Second variational formulation]\label{lem:PGform}
    Smooth solutions $(\vec\rho,\vec\mu,\bm{u}, \bm{m}, p, \varrho s, \vtheta)$ to \eqref{eq:fullsys} satisfy the following variational equations:
    \begin{subequations}\label{eq:weakB}
    \begin{align}
        \la \dt\rho_i,\psi_i \ra - \la \rho_i\u,\nabla\psi_i \ra  &=  - \sum_{j=1}^N \!\left\la\! M_{ij}\nabla \frac{\mu_j}{\vtheta}, \nabla\psi_i \!\right\ra, \label{eq:weakB1}
        \\
        \la \mu_i,\zeta_i \ra  &=  \la \partial_{\rho_i} (\varrho e)(\vec\rho,\varrho s) + V_i p,\zeta_i \ra,   \label{eq:weakB2}
        \\
        \la \sqrt{\varrho}\dt\bm{m}, \v \ra + \mathcal{B}(\varrho\u; \u, \v) &= - \, \la \mathbb{S}(\vec\rho,\vtheta;\mathrm{D}\u),\mathrm{D}\v \ra + \la p,\divM\v \ra  \label{eq:weakB3}  \\
        &\qquad\qquad - \, \!\left\la \sum_{i=1}^N \rho_i(\nabla\mu_i - V_i \na p)+ \varrho s\nabla\vtheta, \v \!\right\ra\!,  \notag
        \\
        \la \bm{m},\w \ra  &=  \la \sqrt{\varrho}\u,\w \ra, \label{eq:weakB4} 
        \\
        \la \divM\u,q \ra  &=  - \, \!\left\la \sum_{i,j=1}^N V_i M_{ij}\nabla\frac{\mu_j}{\vtheta}, \nabla q \!\right\ra\!,   \label{eq:weakB5}
        \\
        \la \dt(\varrho s),\omega \ra - \la \varrho s\u,\nabla\omega \ra  &=  \left\la \mathbb{S}(\vec\rho,\vtheta;\Du):\mathrm{D}\u,\frac{\omega}{\vtheta} \right\ra + \left\la\! K\nabla\frac{1}{\vtheta},\nabla\frac{\omega}{\vtheta} \right\ra\!  \label{eq:weakB6}  \\
        &\qquad\qquad + \sum_{i,j=1}^N \!\left\la M_{ij}\nabla\frac{\mu_j}{\vtheta},\nabla\frac{\mu_i\omega}{\vtheta} \right\ra\!,  \notag
        \\
        \la \vtheta, \gamma \ra  &=  \la \partial_{\varrho s} (\varrho e)(\vec\rho, \varrho s), \gamma\ra,  \label{eq:weakB7}
    \end{align}
    \end{subequations}
    for all smooth test functions $(\psi_i,\zeta_i,\v,\w,q,\omega,\gamma)$, $i=1,\ldots,N$.
\end{lemma}

\begin{proof}
    The mass, quasi-incompressibility, and entropy equations \eqref{eq:weakB1}, \eqref{eq:weakB5}, and \eqref{eq:weakB6} are identical to \eqref{eq:weak1}, \eqref{eq:weak4}, and \eqref{eq:weak5}, respectively.
    In the second equation \eqref{eq:weakB2} defining the chemical potential, we use \eqref{eq:rhof_rhoe} to replace $\partial_{\rho_i} (\varrho F)(\vec\rho,\vtheta)$ in \eqref{eq:weak2} with $\partial_{\rho_i} (\varrho e)(\vec\rho,\varrho s)$.
    The fourth equation \eqref{eq:weakB4} represents a variational definition for the weighted velocity $\bm{m} = \sqrt{\varrho}\u$;
    this is substituted into the momentum equation \eqref{eq:weakB3} as
    \begin{equation}
        \varrho\dt\u + \frac12\dt\varrho\u = \sqrt{\varrho}\dt(\sqrt{\varrho}\u) = \sqrt{\varrho}\dt\bm{m}.
    \end{equation}
    The final equation defining the temperature $\vtheta$ \eqref{eq:weakB7} uses the definition of $\vtheta$ given by the fundamental thermodynamic relation \eqref{1.thermo_b}.

\end{proof}

This second variational formulation now encodes the structures under consideration via simple test functions.
This differs from the above result (\Cref{lem:weak_test}) in that each of the test functions considered in this second formulation will remain a valid choice after CPG discretisation in time.

\begin{lemma}[Test functions in second formulation]\label{lem:weak_test2}
    Solutions of \eqref{eq:weakB} can be seen to fulfil the identities of \Cref{lem:struc} by considering the following test functions:
    \begin{itemize}
        \item Conservation of partial mass $\int_\Omega \rho_i$ \eqref{eq:struc_mass}: $\psi_i = 1$
        \item Conservation of energy $\int_\Omega (\frac12 \|\bm{m}\|^2 + \varrho e)$ \eqref{eq:struc_energy}: $\psi_i = \mu_i-V_i p$, $\zeta_i = \dt\rho_i$, $\v = \u$, $\bm{w} = \partial_t \bm{m}$, $q = p$, $\omega = \vtheta$, $\gamma = \partial_t(\varrho s)$
        \item Generation of entropy $\int_\Omega \varrho s$ \eqref{eq:struc_entropy}: $\omega = 1$
    \end{itemize}
\end{lemma}

\begin{proof}
    As in \Cref{lem:weak_test}, mass conservation and entropy generation are immediate.
    The proof of energy conservation differs in two ways.
    First, note that the kinetic component is instead defined in terms of $\bm{m}$.
    Evaluating its time derivative,
    \begin{equation}
        \frac{\dd}{\dd t} \int_\Omega \frac12 \|\bm{m}\|^2
            =  \int_\Omega \bm{m} \cdot \partial_t\bm{m}
            =  \int_\Omega \sqrt{\varrho}\u \cdot \partial_t\bm{m},
    \end{equation}
    where in the second equality we consider $\bm{w} = \partial_t\bm{m}$ in \eqref{eq:weakB4}.
    The proof then proceeds similarly.

    Second, for the internal energy component, note that $\varrho s$ is no longer interpreted as being defined explicitly in terms of $\vec\rho$ and $\vtheta$;
    instead $\vtheta$ and $\varrho s$ are implicitly related through \eqref{eq:weakB7}.
    We may no longer therefore make the substitution $\vtheta = \partial_{\varrho s}(\varrho e)(\vec\rho, \varrho s)$ going from \eqref{eq:weak_test_proof_2.5a} to \eqref{eq:weak_test_proof_2.5b}.
    Instead, we make the substitution by considering $\gamma = \partial_t(\varrho s)$ in \eqref{eq:weakB7}:
    \begin{equation}
        \int_\Omega \vtheta \partial_t(\varrho s)
            =  \int_\Omega \partial_{\varrho s}(\varrho e)(\vec\rho, \varrho s) \partial_t(\varrho s).
    \end{equation}
    With this, the remainder of the proof proceeds identically to \Cref{lem:weak_test}.
\end{proof}

\begin{example}[Georgiadis--J\"{u}ngel mixing energy]\label{ex:georgiadis_III}
    Consider again the running example of \Cref{ex:georgiadis_I,ex:georgiadis_II}.
    In this formulation, the second equation \eqref{eq:weakB2} uses the partial derivative $\partial_{\rho_i}(\varrho e)(\vec\rho, \varrho s)$.
    By the thermodynamic relation \eqref{eq:rhof_rhoe} and \eqref{eq:georgiadis_drhoF},
    \begin{equation}\label{eq:georgiadis_drhoe}
        \partial_{\rho_i}(\varrho e)(\vec\rho, \varrho s)
            =  \partial_{\rho_i}(\varrho F)(\vec\rho, \vtheta)
            =  \vtheta\log(\rho_i/\varrho) - \cv\vtheta(\log\vtheta - 1).
    \end{equation}
    To write the temperature $\vtheta$ then in terms of $\vec\rho$ and $\varrho s$, we may solve \eqref{eq:georgiadis_entropy} to find
    \begin{equation}\label{eq:georgiadis_theta}
        \vtheta
            =  \exp\left(\frac{\varrho s + \sum_{j=1}^N \rho_j\log(\rho_j/\varrho)}{\cv\varrho}\right)\!.
    \end{equation}
    The final equation determining $\vtheta$ uses the partial derivative $\partial_{\varrho s} (\varrho e)(\vec\rho, \varrho s)$.
    By \eqref{1.thermo_b}, finding this is equivalent to writing $\vtheta$ in terms of $\vec\rho$ and $\varrho s$, as above:
    \begin{equation}\label{eq:georgiadis_dsrhoe}
        \partial_{\varrho s}(\varrho e)(\vec\rho, \varrho s)
            =  \vtheta
            =  \exp\left(\frac{\varrho s + \sum_{j=1}^N \rho_j\log(\rho_j/\varrho)}{\cv\varrho}\right)\!.
    \end{equation}
\end{example}

\begin{remark}[Boundary conditions]\label{rem:bcs}
    Until now, boundary conditions have been neglected under the assumption of periodicity.
    A natural choice of homogeneous boundary conditions on $\partial\Omega$ might include Dirichlet boundary conditions on the velocity $\u$, and Neumann boundary conditions for the chemical potentials $\mu_i$ and temperature $\vtheta$:
    \begin{equation}
        \u = \bm{0},  \qquad
        \bm{n}\cdot\sum_{j=1}^N M_{ij}\nabla\frac{\mu_j}{\vtheta} = \bm{n}\cdot K\nabla\frac{1}{\vtheta} = 0,  \qquad
        \text{on } \partial\Omega.
    \end{equation}
    The former condition can be imposed strongly on the vector solution and test spaces in \eqref{eq:weakA} and \eqref{eq:weakB}.
    The latter condition is naturally imposed weakly when no boundary terms are introduced or imposed on the scalar solution (and test) spaces.
    As this leaves the weak formulation unchanged, the results in \Cref{lem:weak_test,lem:weak_test2}, alongside the results in \Cref{thm:schemeA_global,thm:schemeB_global} below after discretisation, also hold under such conditions.
\end{remark}

\section{Discretisation and analysis}\label{sec.num}

This section constructs discretisations of the proposed variational formulations (\ref{eq:weakA},~\ref{eq:weakB}).
For the spatial discretisation, we require that the mesh $\Th$ partitions the domain $\Omega$ into simplices $K \in \Th$ that can be extended periodically to periodic extensions of $\Omega$.
The mesh size parameter $h > 0$ denotes the maximal diameter of the simplices, while the boundary of the triangulation $\Th$ must coincide with that of the domain $\Omega$.

We define the following finite element spaces over $\Th$:
\begin{subequations}\label{eq:fe_spaces}
\begin{align}
    \Vh  &\coloneqq  \big\{v\in H^1(\Omega)\cap C^0(\overline\Omega):
    v|_K\in \mathbb{P}_\ell(K), \ \forall K\in\Th\big\},
    \\
    \Xh  &\coloneqq  \big\{\bm{v}\in [H^1(\Omega)\cap C^0(\overline\Omega)]^d:
    \bm{v}|_K\in [\mathbb{P}_{\ell+1}(K)]^d, \ \forall K\in\Th\big\},
    \\
    \Qh  &\coloneqq  \big\{v\in\Vh:\langle v,1\rangle=0\big\},
\end{align}
\end{subequations}
where, for $k\in\NN_0$, $\mathbb{P}_k(K)$ denotes the space of polynomials on $K$ of total degree at most $k$.

\begin{remark}[Choice of spaces]
    The key requirements on the spaces $\Xh$ and $\Vh$ \eqref{eq:fe_spaces} are (i)~inf-sup stability, and (ii)~$H^1$-conformity on $\Vh$.
    The latter condition is necessary, since gradients of both $p \in \Qh \subset \Vh$ and the partial densities $\rho_i \in \Vh$ appear in the formulation.
    We consider the Taylor--Hood pair as it is almost universally implemented, and satisfies both properties \cite[Thm.~8.8.1]{Boffi2013}.
    Other compatible spaces however are available, e.g.~over an Alfeld-split mesh we may take $\Xh$ to be the degree-$(\ell+1)$ Alfeld--Sorokina space \cite{Alfeld2016}, with $\Vh$ (and $\Qh$) unchanged as the degree-$\ell$ Lagrange space.  
\end{remark}

\begin{remark}[Non-positivity for  higher-order methods]\label{remark:positivity_a_cutoff}
    In both our proposed schemes, the discrete temperature $\vtheta_h$ and partial mass densities $\rho_{ih}$ will be taken as elements of $\Vh$.
    For $\ell \ge 2$, we are unable to guarantee positivity in $\Vh$, even if the property holds at all nodal values.
    To prevent solver divergence at high order, therefore, it may be necessary to cut off or regularise.
    Similar problems might occur when considering Scheme~II with $k \geq 2$ stages.
\end{remark}

To construct the time discretisation, the time interval $[0,T)$ is divided into $n_T$ subintervals $I_n \coloneqq [t^n,t^{n+1}) \in \Itau$, $n=0,\ldots,n_T-1$.
For notational convenience, we assume the partition is uniform, so $t^n=n\tau$, $n=0,\ldots,n_T$, and $\tau = T/n_T$ is the time step size, but this is not essential.

\subsection{Scheme~I}

The first scheme we propose is a semi-implicit Euler discretisation of the first weak formulation \eqref{eq:weakA}.
It generates entropy, dissipates total energy, and enforces the quasi-incompressibility constraint \eqref{1.incom} pointwise.

We denote $g^n \coloneqq g|_{t^n}$ and $g^{n+1} \coloneqq g|_{t^{n+1}}$ for a general $g$.
The notation $g^*$ serves as a placeholder for either $g^{n+1}$ or $g^n$, when either choice is admissible, while we denote the discrete time derivative $\dttau g^n \coloneqq (g^{n+1} - g^n)/\tau$.
For the discrete entropy density, set
\begin{equation}
    (\varrho s)_h^n \coloneqq \varrho s(\pvec{\rho}_h^n,\vtheta_h^n),
    \qquad
    (\varrho s)_h^{n+1} \coloneqq \varrho s(\pvec{\rho}_h^{n+1},\vtheta_h^{n+1}).
\end{equation}
We introduce also the following four abbreviations:
\begin{subequations}
\begin{align}
    \nabla_h \frac{\mu_{jh}^{n+1}}{\vtheta_h^{n+1}} &\coloneqq \frac{1}{\vtheta_h^{n+1}}\nabla \mu_{jh}^{n+1} - \frac{\mu_{jh}^{n+1}}{\vtheta_h^{*}\vtheta_h^{n+1}}\nabla \vtheta_h^{n+1},
    \\
    \nabla_h \frac{\omega_h\mu_{ih}^{n+1}}{\vtheta_h^{n+1}} &\coloneqq \frac{\omega_h}{\vtheta_h^{n+1}}\nabla\mu_{ih}^{n+1} + \frac{\mu_{ih}^{n+1}}{\vtheta_h^{*}}\nabla\omega_h - \frac{\mu_{ih}^{n+1}\omega_h}{\vtheta_h^*\vtheta_h^{n+1}}\nabla\vtheta_h^{n+1}, 
    \\
    \nabla_h \frac{1}{\vtheta_h^{n+1}} &\coloneqq - \, \frac{1}{\vtheta_h^{*}\vtheta_h^{n+1}}\nabla \vtheta_h^{n+1},
    \\
    \nabla_h \frac{\omega_h}{\vtheta_h^{n+1}} &\coloneqq \frac{1}{\vtheta_h^{*}}\nabla \omega_h - \frac{\omega_h}{\vtheta_h^{*}\vtheta_h^{n+1}}\nabla \vtheta_h^{n+1}.
\end{align}
\end{subequations}

Our semi-implicit Euler discretisation is then as follows:
given $(\pvec{\rho}_h^n, \u_h^n, \vtheta_h^n)\in \Vh^N\times\Xh\times\Vh$, find $(\pvec{\rho}_h^{n+1},\pvec{\mu}_h^{n+1},\u_h^{n+1},p_h^{n+1},\vtheta_h^{n+1})\in \Vh^N\times\Vh^N\times\Xh\times\Qh\times\Vh$ such that
\begin{subequations}\label{eq:schemeA}
   \begin{align}
    \la \dttau\rho_{ih}^n,\psi_{ih} \ra \, - \, &\la \rho_{ih}^*\u_h^{n+1},\nabla\psi_{ih} \ra  \label{eq:schemeA1}  \\
        &= - \sum_{j=1}^N \left\la \! M_{ij}^* \nabla_h\frac{\mu_{jh}^{n+1}}{\vtheta_h^{n+1}},\nabla\psi_{ih} \right\ra,  \notag
    \\
    \la \mu_{ih}^{n+1},\zeta_{ih} \ra
        &=  \la \partial_{\rho_i}(\varrho F)(\pvec{\rho}_h^{n+1}, \vtheta_h^{n+1}) + V_i p_h^{n+1},\zeta_{ih} \ra,  \label{eq:schemeA2}
    \\
    \la \varrho_h^n\dttau\u_h^n, \v_h\ra \, + \, &\frac{1}{2}\la \u_h^{n+1}\dttau\varrho_h^n,\v_h \ra + \mathcal{B}(\varrho_h^*\u_h^*;\u_h^{n+1},\v_h)  \label{eq:schemeA3}  \\
        &= - \, \la \mathbb{S}(\pvec{\rho}_h^*,\vtheta_h^*;\mathrm{D}\u_h^{n+1}),\mathrm{D}\v_h \ra + \la p_h^{n+1},\divM\v_h \ra \notag  \\
        &\qquad - \left\la \sum_{i=1}^N\rho_{ih}^*\nabla(\mu_{ih}^{n+1}-V_i p_h^{n+1})+(\varrho s)_h^*\nabla\vtheta_h^{n+1},\v_h \!\right\ra,  \notag
    \\
    \la \divM\u_h^{n+1},q_h \ra
        &=  - \left\la \sum_{i,j=1}^N V_i M_{ij}^*\nabla_h\frac{\mu_{jh}^{n+1}}{\vtheta_h^{n+1}}, \nabla q_h \right\ra,  \label{eq:schemeA4}
    \\
    \la \dttau(\varrho s)_h^n,\omega_h \ra \, - \, &\la (\varrho s)_h^*\u_h^{n+1},\nabla\omega_h \ra  \label{eq:schemeA5}  \\
        &=  \left\la \mathbb{S}(\pvec{\rho}_h^*,\vtheta_h^*;\mathrm{D}\u_h^{n+1}):\mathrm{D}\u_h^{n+1},\frac{\omega_h}{\vtheta_h^{n+1}} \right\ra  \notag  \\
        &\qquad\qquad+ \left\la K^*\nabla_h \frac{1}{\vtheta_h^{n+1}},\nabla_h \frac{\omega_h}{\vtheta_h^{n+1}}\right\ra  \notag  \\
        &\qquad\qquad\qquad\qquad + \sum_{i,j=1}^N \left\la M_{ij}^*\nabla_h \frac{\mu_{jh}^{n+1}}{\vtheta_h^{n+1}},\nabla_h \frac{\omega_h\mu_{ih}^{n+1}}{\vtheta_h^{n+1}} \right\ra,  \notag
\end{align}
\end{subequations}
for all $(\pvec{\psi}_h,\pvec{\zeta}_h,\v_h,q_h,\omega_h)\in\Vh^N\times\Vh^N\times\Xh\times\Qh\times\Vh$.

We now prove that this semi-implicit Euler discretisation preserves the mass and entropy structures of the continuous system (\Cref{lem:struc}) through the conservation of mass \eqref{eq:struc_mass} and the generation of entropy \eqref{eq:struc_entropy} at an appropriate rate mimicking \eqref{eq:calP},
and also preserves the conservation of energy \eqref{eq:struc_energy} up to some artificial dissipation.
The proof follows a similar structure to that of \Cref{lem:weak_test};
for the energy dissipation, we must appeal to certain convexity conditions on the free energy density $\varrho F$ so as to apply the bound of \Cref{lem:rhoe_bound}.

\begin{theorem}[Global structures for Scheme~I]\label{thm:schemeA_global}
    Supposing the Helmholtz free energy $\varrho F$ is (i)~convex in $\vec\rho$ for fixed $\vtheta$, and (ii)~concave in $\vtheta$ for fixed $\vec\rho$, discrete solutions to \eqref{eq:schemeA} satisfy:
    \begin{subequations}\label{eq:schemeA_global}
    \begin{align}
        \int_\Omega \rho_{ih}^{n+1}  &=  \int_\Omega \rho_{ih}^n, \qquad i = 1, \ldots, N,  \label{eq:schemeA_mass}
        \\
        \int_\Omega \frac{1}{2}\varrho_h^{n+1}\|\u_h^{n+1}\|^2 + \varrho e(\pvec\rho_h^{n+1}, \vtheta_h^{n+1})  \\
        &\le  \int_\Omega \frac{1}{2}\varrho_h^n\|\u_h^n\|^2 + \varrho e(\pvec\rho_h^n, \vtheta_h^n),  \notag
        \\
        \int_\Omega (\varrho s)_h^{n+1}  &=  \int_\Omega (\varrho s)_h^n + \tau\int_\Omega\mathcal{P}_h^n(\pvec\rho_h, \vtheta_h, \pvec\mu_h, \Du_h),
    \end{align}
    \end{subequations}
    where $\mathcal{P}_h^n$ denotes a discrete entropy production rate over $I_n$, analogous to \eqref{eq:calP},
    \begin{multline}
        \mathcal{P}_h^n(\pvec\rho_h, \vtheta_h, \pvec\mu_h, \Du_h)
            \coloneqq
            \frac{1}{\vtheta_h^{n+1}}\mathbb{S}(\pvec\rho_h^*,\vtheta_h^*;\Du_h^{n+1}):\Du_h^{n+1}  \\
            + \begin{pmatrix}
                \nabla_h\frac{\vec{\mu}_h^{n+1}}{\vtheta_h^{n+1}} \\ \\ \nabla_h\frac{1}{\vtheta_h^{n+1}}
            \end{pmatrix}^\top \!\!\!\LL^* \begin{pmatrix}
            \nabla_h\frac{\vec{\mu}_h^{n+1}}{\vtheta_h^{n+1}} \\  \\\nabla_h\frac{1}{\vtheta_h^{n+1}}
            \end{pmatrix}
            \ge  0.
    \end{multline}
    The energy inequality requires $\varrho_h^n \ge 0$ pointwise (see \Cref{remark:positivity_a_cutoff}).
\end{theorem}

\begin{proof}
    Conservation of partial mass follows directly by using $\psi_{ih} = 1$ in \eqref{eq:schemeA1}, while entropy production is obtained by inserting $\omega_h = 1$ into \eqref{eq:schemeA5}.
    Mirroring the test functions in \Cref{lem:weak_test}, we prove energy dissipation by considering $\psi_{ih} = \mu_{ih}^{n+1}-V_i p_h^{n+1}$, $\zeta_{ih} = \dttau\rho_{ih}^n$, $\v_h = \u_h^{n+1}$, $q_h = p_h^{n+1}$, $\omega_h = \vtheta_h^{n+1}$.
    Following the same cancellations as in the proof of \Cref{lem:weak_test},
    \begin{multline}\label{eq:schemeA_combined}
        \int_\Omega \varrho_h^n \dttau\u_h^n \cdot \u_h^{n+1}
            + \frac12 \dttau\varrho_h^n\|\u_h^{n+1}\|^2  \\
            + \sum_{i=1}^N \partial_{\rho_i}(\varrho F)(\pvec{\rho}_h^{n+1}, \vtheta_h^{n+1})\,\dttau\rho_{ih}^n
            + \vtheta_h^{n+1}\,\dttau(\varrho s)_h^n
            = 0.
    \end{multline}
    For the kinetic contribution, assuming $\varrho_h^n \ge 0$ pointwise,
    \begin{subequations}\label{eq:schemeA_bound_a}
    \begin{align}
        \frac12\varrho_h^{n+1}\|\u_h^{n+1}\|^2 &- \frac12\varrho_h^n\|\u_h^n\|^2  \\
            &=  \tau\left(\varrho_h^n\dttau\u_h^n\cdot\u_h^{n+1} + \frac12\dttau\varrho_h^n\|\u_h^{n+1}\|^2\right) - \frac{\tau^2}{2}\varrho_h^n\|\dttau\u_h^n\|^2  \notag  \\
            &\le  \tau\left(\varrho_h^n\dttau\u_h^n\cdot\u_h^{n+1} + \frac12\dttau\varrho_h^n\|\u_h^{n+1}\|^2\right).
    \end{align}
    \end{subequations}
    For the internal energy, we apply \Cref{lem:rhoe_bound} \eqref{eq:rhoe_bound} pointwise to the states $*_- = (\pvec{\rho}_h^n, (\varrho s)_h^n)$ and $*_+ = (\pvec{\rho}_h^{n+1}, (\varrho s)_h^{n+1})$:
    \begin{multline}\label{eq:schemeA_bound_b}
        \varrho e(\pvec{\rho}_h^{n+1}, (\varrho s)_h^{n+1}) - \varrho e(\pvec{\rho}_h^n,(\varrho s)_h^n)  \\
            \le  \tau\left(\sum_{i=1}^N \partial_{\rho_i}(\varrho F)(\pvec{\rho}_h^{n+1}, \vtheta_h^{n+1})\,\dttau\rho_{ih}^n + \vtheta_h^{n+1}\,\dttau(\varrho s)_h^n\right).
    \end{multline}
    Integrating over the domain $\Omega$, summing the two bounds (\ref{eq:schemeA_bound_a},~\ref{eq:schemeA_bound_b}), and applying \eqref{eq:schemeA_combined} gives the energy dissipation result.
\end{proof}

We next show that Scheme~I preserves the quasi-incompressibility constraint \eqref{1.incom}.

\begin{theorem}[Quasi-incompressibility for Scheme~I]\label{thm:schemeA_incom}
    Discrete solutions to \eqref{eq:schemeA} satisfy the quasi-incompressibility constraint \eqref{1.incom} pointwise for all $n$,
    \begin{align}\label{eq:schemeA_incom}
        \sum_{i=1}^N V_i\rho_{ih}^n(\mathbf{x})  =  1, \qquad
        \text{for all } \mathbf{x}\in\Omega,
    \end{align}
    under either of the following conditions on the choice of $\rho_{ih}^*$:

    \vspace{1mm}
    \textit{Case~(i): $\rho_{ih}^* = \rho_{ih}^n$.}
    Quasi-incompressibility holds unconditionally.

    \vspace{1mm}
    \textit{Case~(ii): $\rho_{ih}^* = \rho_{ih}^{n+1}$.}
    Quasi-incompressibility holds under the time step condition for all $n$,
    \begin{equation}\label{eq:schemeA_CFL}
        \tau\,\|\divM\u_h^{n+1}\|_{L^\infty(\Omega)} < 2.
    \end{equation}
\end{theorem}

\begin{proof}
    We proceed by induction.
    Define $Z_h^n \coloneqq \sum_i V_i\rho_{ih}^n - 1$; by assumption on the initial data, $Z_h^0 = 0$ everywhere.
    Suppose $Z_h^n = 0$ for some $n \ge 0$.
    Taking $\psi_{ih} = V_i q_h$ for $q_h \in \Qh$ in \eqref{eq:schemeA1}, summing over $i$, and using $Z_h^n = 0$ to simplify the time-difference term,
    \begin{subequations}
    \begin{equation}
        \frac{1}{\tau} \la Z_h^{n+1}, q_h \ra
            - \left\la \sum_{i=1}^N V_i\rho_{ih}^*\,\u_h^{n+1}, \nabla q_h \right\ra
            = - \left\la \sum_{i,j=1}^N V_i M_{ij}^*\nabla_h\frac{\mu_{jh}^{n+1}}{\vtheta_h^{n+1}}, \nabla q_h \right\ra.
    \end{equation}
    Substituting the right-hand side for $\la\divM\u_h^{n+1}, q_h\ra = - \, \la\u_h^{n+1}, \nabla q_h\ra$ by \eqref{eq:schemeA4},
    \begin{align}
        \frac{1}{\tau} \la Z_h^{n+1}, q_h \ra
            &= \left\la\!\left(\sum_{i=1}^N V_i\rho_{ih}^* - 1\right)\u_h^{n+1}, \nabla q_h \right\ra.
    \intertext{
    Conservation of partial masses \eqref{eq:schemeA_mass} ensures $\int_\Omega Z_h^{n+1} = 0$, implying $Z_h^{n+1} \in \Qh$.
    Thus taking $q_h = Z_h^{n+1}$,
    }
        \frac{1}{\tau} \| Z_h^{n+1} \|^2
            &= \left\la\!\left(\sum_{i=1}^N V_i\rho_{ih}^* - 1\right)\u_h^{n+1}, \nabla Z_h^{n+1} \right\ra.  \label{eq:schemeA_Z}
    \end{align}
    \end{subequations}

    \vspace{1mm}
    \textit{Case~(i): $\rho_{ih}^* = \rho_{ih}^n$.}
    The inductive hypothesis gives $\sum_i V_i\rho_{ih}^n - 1 = Z_h^n = 0$, so the right-hand side of \eqref{eq:schemeA_Z} vanishes, implying $Z_h^{n+1} = 0$.

    \vspace{1mm}
    \textit{Case~(ii): $\rho_{ih}^* = \rho_{ih}^{n+1}$.}
    Here instead $\sum_i V_i\rho_{ih}^{n+1} - 1 = Z_h^{n+1}$.
    Applying integration by parts on the right-hand side,
    \begin{multline}
        \frac{1}{\tau}\|Z_h^{n+1}\|^2
            = \la Z_h^{n+1}\u_h^{n+1}, \nabla Z_h^{n+1}\ra  \\
            = -\frac{1}{2}\la\divM\u_h^{n+1}, (Z_h^{n+1})^2\ra
            \le \frac{1}{2}\|\divM\u_h^{n+1}\|_{L^\infty(\Omega)}\|Z_h^{n+1}\|^2.
    \end{multline}
    Under the time step condition \eqref{eq:schemeA_CFL}, rearranging gives $\|Z_h^{n+1}\|^2 \le 0$, hence $Z_h^{n+1} = 0$.
\end{proof}

\begin{corollary}[Density bounds for Scheme~I]\label{cor:densitybounds}
    Suppose the quasi-incompressibility condition \eqref{eq:schemeA_incom} holds and each partial density $\rho_{ih}^n$ is non-negative everywhere.
    Then for all $n$, the total mass density $\varrho_h^n$ is bounded by
    \begin{subequations}\label{eq:schemeA_density_bounds}
    \begin{equation}\label{eq:schemeA_density_total}
        \frac{1}{\max_i V_i} \le \varrho_h^n(\mathbf{x}) \le \frac{1}{\min_i V_i},  \qquad
        \text{for all } \mathbf{x}\in\Omega.
    \end{equation}
    Moreover,
    \begin{equation}\label{eq:schemeA_density_partial}
        \rho_{ih}^n(\mathbf{x}) \le \frac{1}{V_i},  \qquad
        \text{for all } \mathbf{x}\in\Omega.
    \end{equation}
    \end{subequations}
\end{corollary}

\begin{proof}
    Both these bounds are immediate from \eqref{eq:schemeA_incom}.
    For the total mass density bounds,
    \begin{subequations}
    \begin{equation}
        \min_i V_i \sum_{i=1}^N \rho_{ih}^n  \le  1  \le  \max_i V_i \sum_{i=1}^N \rho_{ih}^n
        \;\;\;\implies\;\;\;
        \frac{1}{\max_i V_i} \le \varrho_h^n(\mathbf{x}) \le \frac{1}{\min_i V_i}.
    \end{equation}
    For the partial density bound, from the non-negativity of $\rho_{ih}^n$,
    \begin{equation}
        V_i \rho_{ih}^n \le 1
        \;\;\;\implies\;\;\;
        \rho_{ih}^n \le \frac{1}{V_i}.
    \end{equation}
    \end{subequations}
\end{proof}

\subsection{Scheme~II}

Our second proposed scheme is a continuous Petrov--Galerkin (CPG) discretisation of the second weak formulation \eqref{eq:weakB}.
Like the semi-implicit Euler discretisation, it generates entropy and enforces the quasi-incompressibility constraint \eqref{1.incom} pointwise. It is more expensive than Scheme~I, but in exchange conserves total energy exactly and is defined for arbitrary polynomial degrees.

For a general discrete space $\Uh \in \{\Vh, \Xh, \Qh\}$, define
\begin{subequations}\label{eq:scheme_ii_pg}
\begin{align}
    \Pi_{\mathrm{disc}}^k(\Itau; \Uh)
        &\coloneqq  \big\{v \in L^1([0,T]; \Uh):
        v|_{I_n} \in \mathbb{P}_k(I_n;\Uh), \ \forall I_n\in\Itau \big\},
    \\
    \Pi_{\mathrm{cts}}^k(\Itau; \Uh)
        &\coloneqq  \Pi_{\mathrm{disc}}^k(\Itau;\Uh) \cap C^0([0,T]; \Uh).
\end{align}
\end{subequations}
The former $\Pi_{\mathrm{disc}}^k(\Itau; \Uh)$ is the space of discontinuous degree-$k$ polynomials from $\Itau$ to $\Uh$;
the latter $\Pi_{\mathrm{cts}}^{k}(\Itau; \Uh)$ is the analogous space of continuous polynomials.
We write $\sigma \coloneqq (h, \tau)$ for the combined discretisation index, so that $g_\sigma$ denotes a discrete field depending jointly on the spatial mesh parameter $h$ and the time step $\tau$.

Our CPG discretisation is then as follows:
find $(\vec{\rho}_{\sigma},\bm{m}_{\sigma},(\varrho s)_{\sigma})\in \Pi^k_{\mathrm{cts}}(\Itau;\Vh^N\times\Xh\times\Vh)$ satisfying known discrete initial data at $t = 0$,
\begin{equation}
    \vec{\rho}_{\sigma}(0) = \pvec{\rho}_h^0,  \qquad
    \bm{m}_{\sigma}(0) = \bm{m}_h^0,  \qquad
    (\varrho s)_{\sigma}(0) = (\varrho s)_h^0,
\end{equation}
and $(\vec{\mu}_{\sigma},\u_{\sigma},p_{\sigma},\vtheta_{\sigma})\in \Pi^{k-1}_{\mathrm{disc}}(\Itau;\Vh^N\times\Xh\times\Qh\times\Vh)$, such that on each time step $n = 0, \ldots, n_T-1$,
\begin{subequations}\label{eq:schemeB}
\begin{align}
    &\int_{I_n}\!\left[\la \dt\rho_{i\sigma},\psi_{i\sigma} \ra \, - \, \la \rho_{i\sigma}\u_{\sigma},\nabla\psi_{i\sigma} \ra\right] \dd t  \label{eq:schemeB1}  \\
        &=  - \int_{I_n} \sum_{j=1}^N \left\la M_{ij}\nabla \frac{\mu_{j\sigma}}{\vtheta_{\sigma}}, \nabla\psi_{i\sigma} \!\right\ra \dd t,  \notag
    \\
    \int_{I_n}\la \mu_{i\sigma},\zeta_{i\sigma} \ra\,\dd t
        &=  \int_{I_n} \la \partial_{\rho_i} (\varrho e)(\vec\rho_{\sigma},(\varrho s)_{\sigma}) + V_i p_{\sigma},\zeta_{i\sigma} \ra\,\dd t,   \label{eq:schemeB2}
    \\
    &\int_{I_n}\!\left[\la \sqrt{\varrho_{\sigma}}\dt\bm{m}_{\sigma}, \v_{\sigma} \ra + \mathcal{B}(\varrho_{\sigma}\u_{\sigma}; \u_{\sigma}, \v_{\sigma})\right] \dd t  \label{eq:schemeB3}  \\
        &=  \int_{I_n}\!\left[- \, \la \mathbb{S}(\vec\rho_{\sigma},\vtheta_{\sigma};\mathrm{D}\u_{\sigma}),\mathrm{D}\v_{\sigma} \ra + \la p_{\sigma},\divM\v_{\sigma} \ra\right.  \notag  \\
        &\qquad\quad \left.- \left\la \sum_{i=1}^N \rho_{i\sigma}(\nabla\mu_{i\sigma} - V_i \na p_{\sigma}) + (\varrho s)_{\sigma}\nabla\vtheta_{\sigma}, \v_{\sigma} \!\right\ra\right] \dd t,  \notag
    \\
    \int_{I_n}\la \bm{m}_{\sigma}, \w_{\sigma} \ra\,\dd t
        &=  \int_{I_n} \la \sqrt{\varrho_{\sigma}}\u_{\sigma}, \w_{\sigma} \ra\,\dd t, \label{eq:schemeB4}
    \\
    \int_{I_n}\la \divM\u_{\sigma}, q_{\sigma} \ra\,\dd t
        &=  - \int_{I_n} \sum_{i,j=1}^N \left\la V_i M_{ij}\nabla\frac{\mu_{j\sigma}}{\vtheta_{\sigma}}, \nabla q_{\sigma} \!\right\ra\,\dd t,   \label{eq:schemeB5}
    \\
    &\int_{I_n}\!\left[\la \dt(\varrho s)_{\sigma}, \omega_{\sigma} \ra \, - \, \la (\varrho s)_{\sigma}\u_{\sigma}, \nabla\omega_{\sigma} \ra\right] \dd t  \label{eq:schemeB6}  \\
        &=  \int_{I_n}\!\left[\left\la \mathbb{S}(\vec\rho_{\sigma},\vtheta_{\sigma};\mathrm{D}\u_{\sigma}):\mathrm{D}\u_{\sigma},\frac{\omega_{\sigma}}{\vtheta_{\sigma}} \right\ra + \left\la K\nabla\frac{1}{\vtheta_{\sigma}},\nabla\frac{\omega_{\sigma}}{\vtheta_{\sigma}} \right\ra\right.  \notag  \\
        &\qquad\qquad \left.+ \sum_{i,j=1}^N \!\left\la M_{ij}\nabla\frac{\mu_{j\sigma}}{\vtheta_{\sigma}},\nabla\frac{\mu_{i\sigma}\omega_{\sigma}}{\vtheta_{\sigma}} \right\ra\right] \dd t,  \notag
    \\
    \int_{I_n}\la \vtheta_{\sigma}, \gamma_{\sigma} \ra\,\dd t
        &= \int_{I_n} \la \partial_{\varrho s} (\varrho e)(\vec\rho_{\sigma}, (\varrho s)_{\sigma}), \gamma_{\sigma}\ra\,\dd t,  \label{eq:schemeB7}
\end{align}
\end{subequations}
for all $(\vec\psi_\sigma,\vec\zeta_\sigma,\v_\sigma,\w_\sigma,q_\sigma,\omega_\sigma,\gamma_\sigma)\in \Pi^{k-1}_{\mathrm{disc}}(\Itau;\Vh^N\times\Vh^N\times\Xh\times\Xh\times\Qh\times\Vh\times\Vh)$.

\begin{theorem}[Global structures for Scheme~II]\label{thm:schemeB_global}
    Discrete solutions to \eqref{eq:schemeB} satisfy:
    \begin{subequations}\label{eq:schemeB_global}
    \begin{align}
        \int_\Omega \rho_{i\sigma}(t^{n+1})  &=  \int_\Omega \rho_{i\sigma}(t^n), \qquad i = 1, \ldots, N,  \label{eq:schemeB_mass}
        \\
        \mathcal{E}_\sigma(t^{n+1}) &= \mathcal{E}_\sigma(t^n),
        \\
        \int_\Omega (\varrho s)_\sigma(t^{n+1})  &=  \int_\Omega (\varrho s)_\sigma(t^n) \notag \\
        &\qquad + \int_{I_n}\int_\Omega\mathcal{P}(\vec\rho_\sigma, \vtheta_\sigma, \vec\mu_\sigma, \Du_\sigma)\,\dd x\,\dd t.
    \end{align}
    \end{subequations}
    Here the discrete total energy is
    \begin{equation}
        \mathcal{E}_\sigma(t)
        \coloneqq \int_\Omega \left[\frac{1}{2}\|\bm{m}_\sigma(t)\|^2
        + \varrho e(\vec\rho_\sigma(t), (\varrho s)_\sigma(t))\right]\dd x.
    \end{equation}
\end{theorem}

\begin{proof}
    Conservation of partial mass follows directly by considering $\psi_{i\sigma} = 1$ in \eqref{eq:schemeB1}, while entropy production is obtained with $\omega_\sigma = 1$ in \eqref{eq:schemeB6}.
    Considering the change in total energy,
    \begin{subequations}\label{eq:schemeB_combined}
    \begin{align}
        \int_\Omega &\left[\frac{1}{2}\|\bm{m}_\sigma(t^{n+1})\|^2 + \varrho e(\vec\rho_\sigma(t^{n+1}), (\varrho s)_\sigma(t^{n+1}))\right]  \\
                &\qquad\qquad- \int_\Omega \left[\frac{1}{2}\|\bm{m}_\sigma(t^n)\|^2 + \varrho e(\vec\rho_\sigma(t^n), (\varrho s)_\sigma(t^n))\right]  \notag  \\
            &= \int_{I_n}\ddt\left[\int_\Omega \left(\frac{1}{2}\|\bm{m}_\sigma\|^2 + \varrho e(\vec\rho_\sigma, (\varrho s)_\sigma)\right)\right]\dd t  \\
            &=  \int_{I_n}\!\bigl[\la\bm{m}_\sigma, \dt\bm{m}_\sigma\ra
                    + \sum_{i=1}^N \la\partial_{\rho_i}(\varrho e)(\vec\rho_\sigma, (\varrho s)_\sigma), \dt\rho_{i\sigma}\ra  \\
                    &\qquad\qquad\qquad\qquad+ \la\partial_{\varrho s}(\varrho e)(\vec\rho_\sigma, (\varrho s)_\sigma), \dt(\varrho s)_\sigma\ra\bigr]\dd t,  \notag  \\
    \intertext{
    where the first equality holds by the fundamental theorem of calculus, and the second by differentiation of the energy under the integral.
    The continuous unknowns $\bm{m}_\sigma$, $\rho_{i\sigma}$, $(\varrho s)_\sigma \in \Pi^k_{\mathrm{cts}}$ have time derivatives in $\Pi^{k-1}_{\mathrm{disc}}$, implying $\dt\bm{m}_\sigma$, $\dt\rho_{i\sigma}$, $\dt(\varrho s)_\sigma$ represent valid test functions in \eqref{eq:schemeB4}, \eqref{eq:schemeB2} and \eqref{eq:schemeB7} respectively, allowing us to simplify further:
    }
            &=  \int_{I_n}\!\left[\la\sqrt{\varrho_\sigma}\u_\sigma, \dt\bm{m}_\sigma\ra
                    + \sum_{i=1}^N \la\mu_{i\sigma} - V_i p_\sigma, \dt\rho_{i\sigma}\ra
                    + \la\vtheta_\sigma, \dt(\varrho s)_\sigma\ra\right]\dd t.  \label{eq:schemeB_combined_4}
    \end{align}
    \end{subequations}
    Considering $\v_\sigma = \u_\sigma$, $\psi_{i\sigma} = \mu_{i\sigma} - V_i p_\sigma$, $q_\sigma = p_\sigma$ and $\omega_\sigma = \vtheta_\sigma$ in \eqref{eq:schemeB3}, \eqref{eq:schemeB1}, \eqref{eq:schemeB5} and \eqref{eq:schemeB6} respectively, the right-hand side of \eqref{eq:schemeB_combined} vanishes, implying energy conservation.
    Similar to the proof of \Cref{thm:schemeA_global}, each of these test functions mirrors those in \Cref{lem:weak_test2}.
\end{proof}

\begin{remark}[Inexact quadrature in the CPG discretisation]\label{rem:schemeB_quadrature}
    Our proposed CPG discretisation \eqref{eq:schemeB} posits the use of exact time integrals $\int_{I_n}$, which are in general computationally inaccessible.
    However, the majority of integrals in the scheme can safely be approximated by quadrature (of at least $k$ stages) without affecting the above thermodynamic stability properties \eqref{eq:schemeB_global}.
    Provided the same quadrature approximation is used throughout, this is true of all terms except for the following:
    \begin{subequations}
    \begin{align}
        \int_{I_n} \la \partial_{\rho_i} (\varrho e)(\vec\rho_{\sigma},(\varrho s)_{\sigma}),\zeta_{i\sigma} \ra\,\dd t
            &\text{ in \eqref{eq:schemeB2}},  \label{eq:exact_integrals_1}  \\
        \int_{I_n}\la \bm{m}_{\sigma}, \w_{\sigma} \ra\,\dd t
            &\text{ in \eqref{eq:schemeB4}},  \label{eq:exact_integrals_2}  \\
        \int_{I_n} \la \partial_{\varrho s} (\varrho e)(\vec\rho_{\sigma}, (\varrho s)_{\sigma}), \gamma_{\sigma}\ra\,\dd t
            &\text{ in \eqref{eq:schemeB7}}.  \label{eq:exact_integrals_3}
    \end{align}
    \end{subequations}
    These integrals are required to be evaluated exactly for \eqref{eq:schemeB_combined_4} to hold in the proof of the global structures (\Cref{thm:schemeB_global}), as these integrals over $I_n$ arise naturally after application of the fundamental theorem of calculus.
    The second term \eqref{eq:exact_integrals_2} can be evaluated exactly by any chosen quadrature rule of degree $2k-1$ or higher.
    However this is not true in general for the first \eqref{eq:exact_integrals_1} and final \eqref{eq:exact_integrals_3} terms.
    In practice we employ a quadrature rule of sufficiently high degree so that the resulting energy error remains below that induced by inexact nonlinear solves.
\end{remark}

We now prove that the lowest-order CPG discretisation preserves the quasi-incompressibility constraint \eqref{1.incom} under a time step condition similar to \eqref{eq:schemeA_CFL}.

\begin{theorem}[Quasi-incompressibility for Scheme~II]\label{thm:schemeB_incom}
    For $k = 1$, discrete solutions to \eqref{eq:schemeB} satisfy the quasi-incompressibility constraint \eqref{1.incom} pointwise for all $n$,
    \begin{align}\label{eq:schemeB_incom}
        \sum_{i=1}^N V_i\rho_{i\sigma}(t, \mathbf{x})  =  1, \qquad
        \text{for all } t \in [0, T], \mathbf{x}\in \Omega,
    \end{align}
    under the time step condition
    \begin{equation}\label{eq:schemeB_CFL}
        \tau\,\|\divM\u_\sigma\|_{L^\infty([0, T]; L^\infty(\Omega))} < 4.
    \end{equation}
\end{theorem}

\begin{proof}
    The proof is largely similar to that of \Cref{thm:schemeA_incom}, proceeding by induction.
    Define again $Z_\sigma \coloneqq \sum_i V_i\rho_{i\sigma} - 1 \in \Pi^1_{\mathrm{cts}}(\Itau; \Vh)$, with $Z_\sigma(t^0) = 0$ everywhere by assumption on the initial data.
    Suppose then $Z_\sigma(t^n) = 0$.
    Taking $\psi_{i\sigma} = V_i q_\sigma$ for some general $q_\sigma \in \Pi^0_{\mathrm{disc}}(\Itau; \Qh)$ in \eqref{eq:schemeB1}, summing over $i$, and substituting the right-hand side for $\int_{I_n}\la \divM\u_{\sigma}, q_{\sigma} \ra = - \int_{I_n}\la \u_{\sigma}, \nabla q_{\sigma} \ra$ via \eqref{eq:schemeB5},
    \begin{subequations}
    \begin{align}
        \int_{I_n}\la \dt Z_\sigma, q_\sigma \ra
            &=  \int_{I_n}\la Z_\sigma\u_\sigma, \nabla q_\sigma \ra.  \\
    \intertext{
    Partial-mass conservation and the inductive hypothesis give $\int_\Omega Z_\sigma(t^{n+1}) = 0$, implying $Z_\sigma(t^{n+1}) \in \Qh$.
    Thus $\dt Z_\sigma \in \Pi^0_{\mathrm{disc}}(\Itau; \Qh)$, such that we may consider $q_\sigma = \dt Z_\sigma$:
    }
        \int_{I_n} \| \dt Z_\sigma \|^2
            &=  \int_{I_n}\la Z_\sigma\u_\sigma, \nabla \dt Z_\sigma \ra.
    \end{align}
    With $Z_\sigma(t^n) = 0$, $Z_\sigma = \frac{t - t^n}{\tau}Z_\sigma(t^{n+1})$ and $\dt Z_\sigma = \frac{1}{\tau}Z_\sigma(t^{n+1})$ on $I_n$;
    similarly, since $\u_\sigma$ is constant on $I_n$, we may write $\u_\sigma = \u_\sigma(t^{n+1})$.
    We may thus simplify,
    \begin{align}
        \frac{1}{\tau} \| Z_\sigma(t^{n+1}) \|^2
            &=  \frac{1}{2} \la Z_\sigma(t^{n+1})\,\u_\sigma(t^{n+1}), \nabla Z_\sigma(t^{n+1}) \ra  \\
            &=  - \frac{1}{4} \la \divM\u_\sigma(t^{n+1}), Z_\sigma(t^{n+1})^2 \ra  \\
            &\le  \frac{1}{4} \| \divM\u_\sigma(t^{n+1}) \|_{L^\infty(\Omega)} \| Z_\sigma(t^{n+1}) \|^2,
    \end{align}
    \end{subequations}
    where in the second equality we apply integration by parts.
    Under the time step condition \eqref{eq:schemeB_CFL}, rearranging gives $\| Z_\sigma(t^{n+1}) \|^2 \le 0$, hence $Z_\sigma(t^{n+1}) = 0$. By induction, the quasi-incompressibility condition $Z_\sigma = 0$ holds for all time.
\end{proof}

\begin{corollary}[Density bounds for Scheme~II]\label{cor:densitybounds2}
    For $k = 1$, suppose the quasi-incompressibility condition \eqref{eq:schemeB_incom} holds and each partial density $\rho_{i\sigma}$ is non-negative everywhere.
    Then the total mass density $\varrho_\sigma$ is bounded by
    \begin{subequations}\label{eq:schemeB_density_bounds}
    \begin{equation}\label{eq:schemeB_density_total}
        \frac{1}{\max_i V_i} \le \varrho_\sigma(t, \mathbf{x}) \le \frac{1}{\min_i V_i},  \qquad
        \text{for all } t \in [0, T], \mathbf{x}\in\Omega.
    \end{equation}
    Moreover,
    \begin{equation}\label{eq:schemeB_density_partial}
        \rho_{i\sigma}(t, \mathbf{x}) \le \frac{1}{V_i},  \qquad
        \text{for all } t \in [0, T], \mathbf{x}\in\Omega.
    \end{equation}
    \end{subequations}
\end{corollary}

\begin{proof}
    The density bounds follow from the same arguments as in \Cref{cor:densitybounds} at each time node $t^n$.
    For $k = 1$, each $\rho_{i\sigma}$ is linear in time on each interval $I_n$, and the bounds are linear inequalities in the partial densities, so they extend to all $t \in [0, T]$ through linear interpolation.
\end{proof}

To make the CPG family concrete, we state the lowest-order-in-time discretisation ($k = 1$) explicitly.
On each interval $I_n$, the continuous unknowns $(\vec\rho_\sigma, \bm{m}_\sigma, (\varrho s)_\sigma)$ are determined by their values $g_h^n \coloneqq g_\sigma(t^n)$ and $g_h^{n+1} \coloneqq g_\sigma(t^{n+1})$ (for general $g$) at the interval endpoints:
\begin{equation}
    g_\sigma  =  \frac{t^{n+1} - t}{\tau} g_h^n + \frac{t - t^n}{\tau} g_h^{n+1}.
\end{equation}
We recall the discrete time derivative $\dttau g_h^n \coloneqq (g_h^{n+1} - g_h^n)/\tau$ from the semi-implicit Euler discretisation \eqref{eq:schemeA}, such that $\dt g_\sigma = \dttau g_h^n$ on $I_n$.
The auxiliary fields $(\vec\mu_\sigma, \u_\sigma, p_\sigma, \vtheta_\sigma)$ are constant on $I_n$;
we denote their single value on $I_n$ as $g^{n+1}_h$.

Following on from \Cref{rem:schemeB_quadrature}, we shall approximate all integrals except \eqref{eq:exact_integrals_1} and \eqref{eq:exact_integrals_3} by a midpoint rule, noting that this is exact for \eqref{eq:exact_integrals_2}.
For the continuous unknowns, we denote the midpoint value by $g_h^{n+1/2} \coloneqq \tfrac12(g_h^n + g_h^{n+1})$.
We characterise integrals in time via $\overline{(\cdot)} \coloneqq \tfrac{1}{\tau}\int_{I_n}(\cdot)$, e.g.
\begin{subequations}
\begin{align}
    \overline{\partial_{\varrho s} (\varrho e)(\vec\rho_{\sigma}, (\varrho s)_{\sigma})}
        &\coloneqq  \frac{1}{\tau} \int_{I_n} \partial_{\varrho s} (\varrho e)(\vec\rho_{\sigma}, (\varrho s)_{\sigma})\,\dd t  \\
        &=  \frac{1}{\tau} \int_{I_n} \partial_{\varrho s} (\varrho e)\!\left(\frac{t^{n+1} - t}{\tau} \pvec\rho_h^n + \frac{t - t^n}{\tau} \pvec\rho_h^{n+1},\right.  \\
        &\qquad\qquad\qquad\qquad\left.\frac{t^{n+1} - t}{\tau} (\varrho s)_h^n + \frac{t - t^n}{\tau} (\varrho s)_h^{n+1}\right)\!\dd t.  \notag
\end{align}
\end{subequations}

Dividing each equation of \eqref{eq:schemeB} by $\tau$ and carrying out the corresponding quadrature and integrals, the scheme then reads:
given $(\pvec{\rho}_h^n, \bm{m}_h^n, (\varrho s)_h^n)\in \Vh^N\times\Xh\times\Vh$, find $(\pvec{\rho}_h^{n+1},\bm{m}_h^{n+1},(\varrho s)_h^{n+1},\pvec\mu_h^{n+1},\u_h^{n+1},p_h^{n+1},\vtheta_h^{n+1})\in \Vh^N\times\Xh\times\Vh\times\Vh^N\times\Xh\times\Qh\times\Vh$ such that
\begin{subequations}\label{eq:schemeBlow}
\begin{align}
    \la \dttau\rho_{ih}^n,\psi_{ih} \ra \, - \, &\la \rho_{ih}^{n+1/2}\u_h^{n+1},\nabla\psi_{ih} \ra  \label{eq:schemeBlow1}  \\
        &=  - \sum_{j=1}^N \left\la M_{ij}^{n+1/2}\,\nabla \frac{\mu_{jh}^{n+1}}{\vtheta_h^{n+1}}, \nabla\psi_{ih} \!\right\ra,  \notag
    \\
    \la \mu_{ih}^{n+1},\zeta_{ih} \ra
        &=  \!\left\la\, \overline{\partial_{\rho_i} (\varrho e)(\vec\rho_\sigma,(\varrho s)_\sigma)} + V_i p_h^{n+1},\zeta_{ih} \right\ra\!,   \label{eq:schemeBlow2}
    \\
    \left\la\! \sqrt{\varrho_h^{n+1/2}}\,\dttau\bm{m}_h^n,\right. &\left.\phantom{\sqrt{\varrho_h^{n+1/2}}}\hspace{-15mm}\v_h \!\right\ra\! + \mathcal{B}(\varrho_h^{n+1/2}\u_h^{n+1}; \u_h^{n+1}, \v_h)  \label{eq:schemeBlow3}  \\
        &=  - \, \la \mathbb{S}(\pvec\rho_h^{n+1/2},\vtheta_h^{n+1};\mathrm{D}\u_h^{n+1}),\mathrm{D}\v_h \ra + \la p_h^{n+1},\divM\v_h \ra  \notag  \\
        &\qquad\qquad\quad - \left\la \sum_{i=1}^N \rho_{ih}^{n+1/2}(\nabla\mu_{ih}^{n+1} - V_i \na p_h^{n+1})\right.  \notag  \\
        &\qquad\qquad\qquad\qquad\qquad\left.\phantom{\sum_{i=1}^N}+ (\varrho s)_h^{n+1/2}\,\nabla\vtheta_h^{n+1}, \v_h \!\right\ra,  \notag
    \\
    \la \bm{m}_h^{n+1/2}, \w_h \ra
        &=  \!\left\la\! \sqrt{\varrho_h^{n+1/2}}\,\u_h^{n+1}, \w_h \!\right\ra\!,  \label{eq:schemeBlow4}
    \\
    \la \divM\u_h^{n+1}, q_h \ra
        &=  - \sum_{i,j=1}^N \left\la V_i M_{ij}^{n+1/2}\,\nabla\frac{\mu_{jh}^{n+1}}{\vtheta_h^{n+1}}, \nabla q_h \!\right\ra,   \label{eq:schemeBlow5}
    \\
    \la \dttau(\varrho s)_h^n, \omega_h \ra \, - \, &\la (\varrho s)_h^{n+1/2}\u_h^{n+1}, \nabla\omega_h \ra  \label{eq:schemeBlow6}  \\
        &=  \left\la \mathbb{S}(\pvec\rho_h^{n+1/2},\vtheta_h^{n+1};\mathrm{D}\u_h^{n+1}):\mathrm{D}\u_h^{n+1},\frac{\omega_h}{\vtheta_h^{n+1}} \right\ra  \notag  \\
        &\qquad\quad+ \left\la K^{n+1/2}\,\nabla\frac{1}{\vtheta_h^{n+1}},\nabla\frac{\omega_h}{\vtheta_h^{n+1}} \right\ra\!  \notag  \\
        &\qquad\qquad\qquad + \sum_{i,j=1}^N \!\left\la M_{ij}^{n+1/2}\,\nabla\frac{\mu_{jh}^{n+1}}{\vtheta_h^{n+1}},\nabla\frac{\mu_{ih}^{n+1}\omega_h}{\vtheta_h^{n+1}} \right\ra,  \notag
    \\
    \la \vtheta_h^{n+1}, \gamma_h \ra
        &=  \!\left\la\, \overline{\partial_{\varrho s} (\varrho e)(\vec\rho_{\sigma}, (\varrho s)_{\sigma})}, \gamma_h \right\ra\!,  \label{eq:schemeBlow7}
\end{align}
\end{subequations}
for all $(\vec\psi_h,\vec\zeta_h,\v_h,\w_h,q_h,\omega_h,\gamma_h)\in \Vh^N\times\Vh^N\times\Xh\times\Xh\times\Qh\times\Vh\times\Vh$.

\begin{remark}[Unconditional quasi-incompressibility in the CPG discretisation]
    Analogous to the $g^*$ terms in the semi-implicit Euler discretisation \eqref{eq:schemeA}, like terms in \eqref{eq:schemeBlow} may be replaced by suitable approximations, again without affecting the mass, energy and entropy structures \eqref{eq:schemeB_global} and their proofs.
    In particular, if an approximation for the partial density $\rho_{ih}^{n+1/2}$ in the advective term $\la \rho_{ih}^{n+1/2}\u_h^{n+1},\nabla\psi_{ih} \ra$ of \eqref{eq:schemeBlow1} is chosen independently of $\rho_{ih}^{n+1}$, then the quasi-incompressibility constraint \eqref{eq:schemeB_incom} will hold independently of the time step $\tau$ (i.e.~independently of the condition \eqref{eq:schemeB_CFL}).
    This is analogous to Case~(i) of the semi-implicit Euler discretisation (\Cref{thm:schemeA_incom}).
    By way of contrast, however, the one-stage CPG scheme is second-order in time, whereas the semi-implicit Euler scheme is only first-order;
    approximating $\rho_{ih}^{n+1/2} \approx \rho_{ih}^n$ therefore would affect the convergence rate of the CPG discretisation.
    Beyond the first time step, a natural alternative to guarantee unconditional quasi-incompressibility while preserving the second-order convergence rate would be an approximation $\rho_{ih}^{n+1/2} \approx \frac{3}{2}\rho_{ih}^n - \frac{1}{2}\rho_{ih}^{n-1}$.
\end{remark}

\section{Numerical simulations} \label{Sec:Simul}
The schemes of \Cref{sec.num} (\ref{eq:schemeA},~\ref{eq:schemeB}) are implemented in NGSolve \cite{schoberl2014c++} for the unit square tests in \Cref{sec:unit_square}, and Firedrake \cite{Ham23} for the application example in \Cref{sec:app_ex}, with the continuous Petrov--Galerkin time discretisation of Scheme~II generated by Irksome \cite{FKM21,ABFKM26}.
Nonlinear systems arising at each time step are solved by Newton's method, with an absolute tolerance of $10^{-11}$ on the $L^2$ norm of the update
for the NGSolve simulations, and an absolute tolerance of $10^{-11}$ on the Euclidean norm of the residual for the Firedrake simulations. The source code used to generate the numerical results presented in this manuscript is publicly available at \url{https://github.com/AaronBrunk1/ThermalQuasi-incompressibleMSNS}.

\subsection{Tests on the unit square}\label{sec:unit_square}
We first consider two tests on $\Omega = (0,1)^2$, with points denoted by $(x,y)\in\Omega$:
a convergence test (\Cref{sec:conv}), and a test of structure preservation (\Cref{sec:struc}).
Both use the Georgiadis--J\"{u}ngel mixing energy \eqref{eq:georgiadis} with $\cv = 1$, the mobility
\begin{equation}\label{eq:num_constitutive}
    M_{ij}(\vec\rho,\vtheta) = \bar{M}\!\left(\delta_{ij}\rho_i - \frac{\rho_i\rho_j}{\varrho}\right)\!,
\end{equation}
the constant heat conductivity $K = 10^{-2}$, and the viscosities $\eta = \bar\eta$ and $\lambda = 0$, for constants $\bar M, \bar\eta > 0$ specified below.

\subsubsection{Spatial convergence test}\label{sec:conv}
We first consider a mixture of $N=2$ components up to the final time $T=0.1$, with $\bar M = 10^{-1}$, $\bar\eta = 10^{-3}$, $V_1=0.8$ and $V_2=0.2$.
The initial data are
\begin{subequations}\label{eq:conv_ic}
\begin{align}
    \rho_1^0(x,y) &= 1 + 0.2\sin(4\pi x)\cos(2\pi y), \label{eq:conv_ic_rho1}  \\
    \rho_2^0(x,y) &= V_2^{-1}\bigl(1 - V_1\rho_1^0(x,y)\bigr), \label{eq:conv_ic_rho2}  \\
    \u^0(x,y) &= \bigl(-\sin(\pi x)^2\sin(2\pi y),\ \sin(2\pi x)\sin(\pi y)^2\bigr)^\top, \label{eq:conv_ic_u}  \\
    \vtheta^0(x,y) &= 1 + 0.1\sin(2\pi x)\cos(2\pi y), \label{eq:conv_ic_theta}
\end{align}
\end{subequations}
where the second partial density is chosen so that the initial data satisfy the quasi-incompressibility constraint \eqref{1.incom}.
For Scheme~II, the initial entropy density $(\varrho s)^0$ is computed from $\pvec{\rho}^0$ and $\vtheta^0$ via its constitutive relation.
Snapshots of the solution at the initial and final times are shown in \Cref{fig:conv_snapshots}.

\begin{figure}[!htbp]
\centering
\begin{tabular}{c@{\hspace{0.6em}}c@{\hspace{2pt}}c@{\hspace{3pt}}l}
    & \small $t = 0$ & \small $t = T$ \\[3pt]
    $\rho_1$ & \includegraphics[valign=c,trim={471bp 221bp 471bp 221bp},clip,width=96bp]{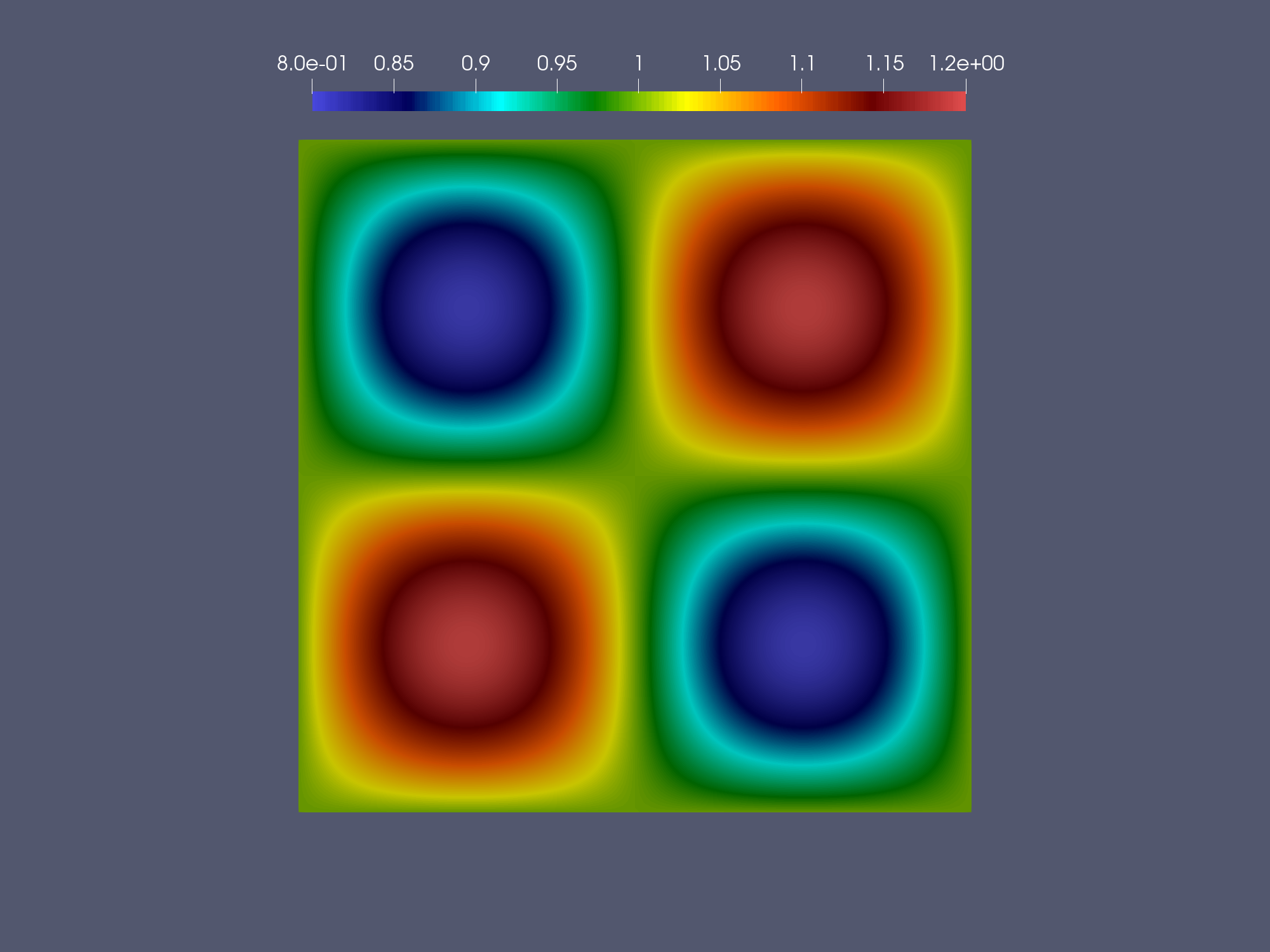} & \includegraphics[valign=c,trim={471bp 221bp 471bp 221bp},clip,width=96bp]{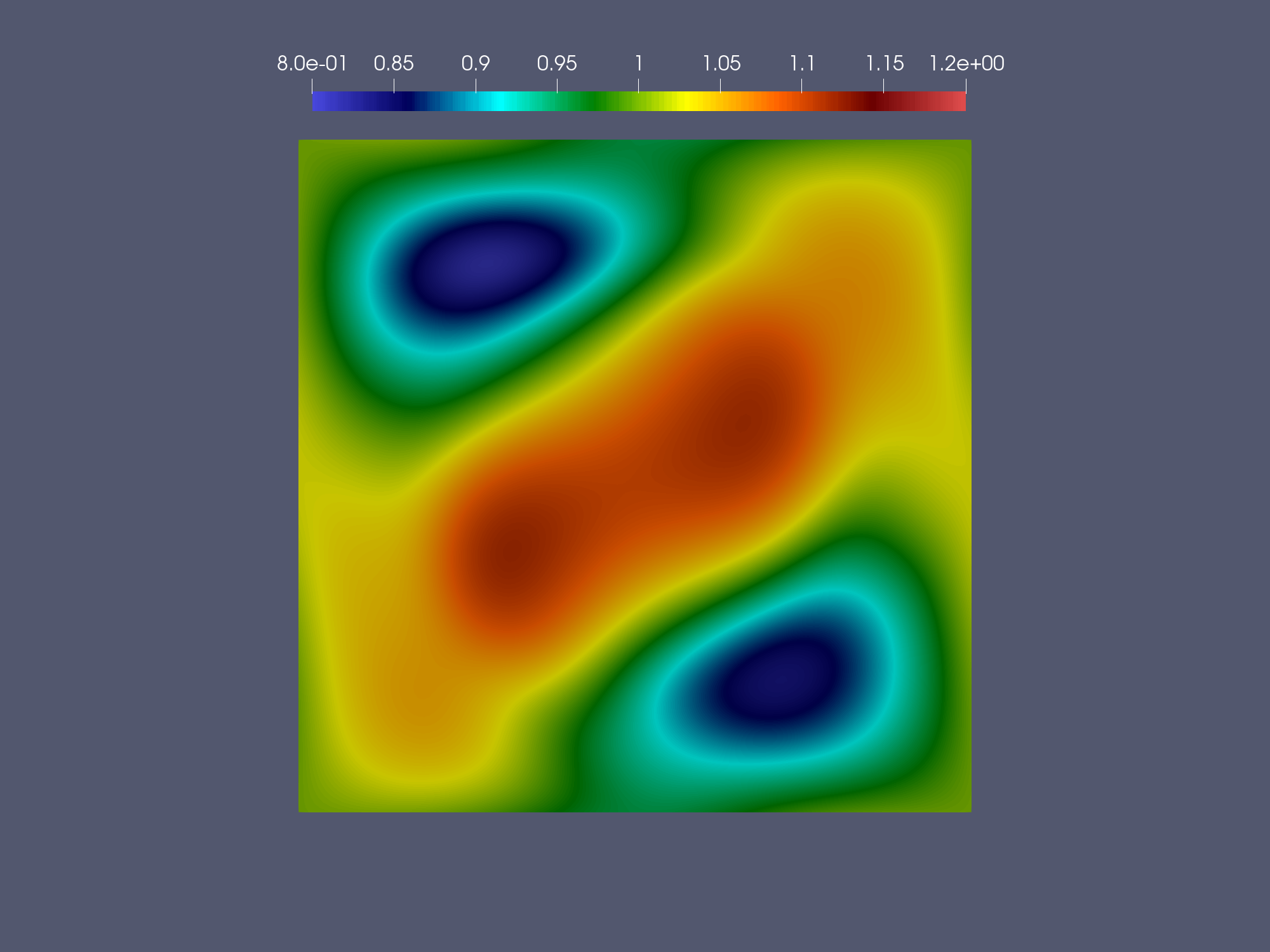} & \cbar{0.8}{1.2}{90bp}{0.8,0.9,1,1.1,1.2} \\ \noalign{\vskip 2pt}
    $\rho_2$ & \includegraphics[valign=c,trim={471bp 221bp 471bp 221bp},clip,width=96bp]{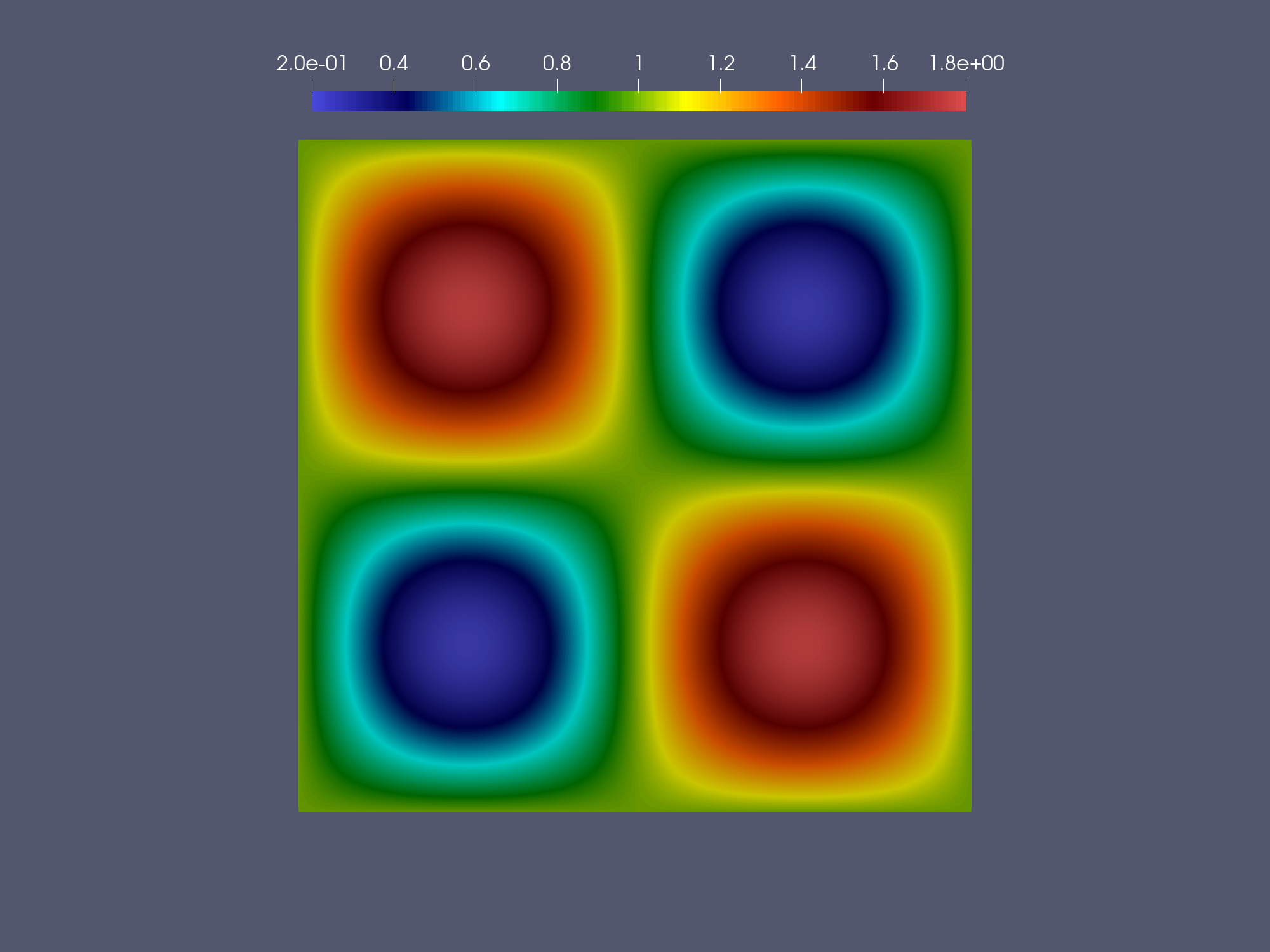} & \includegraphics[valign=c,trim={471bp 221bp 471bp 221bp},clip,width=96bp]{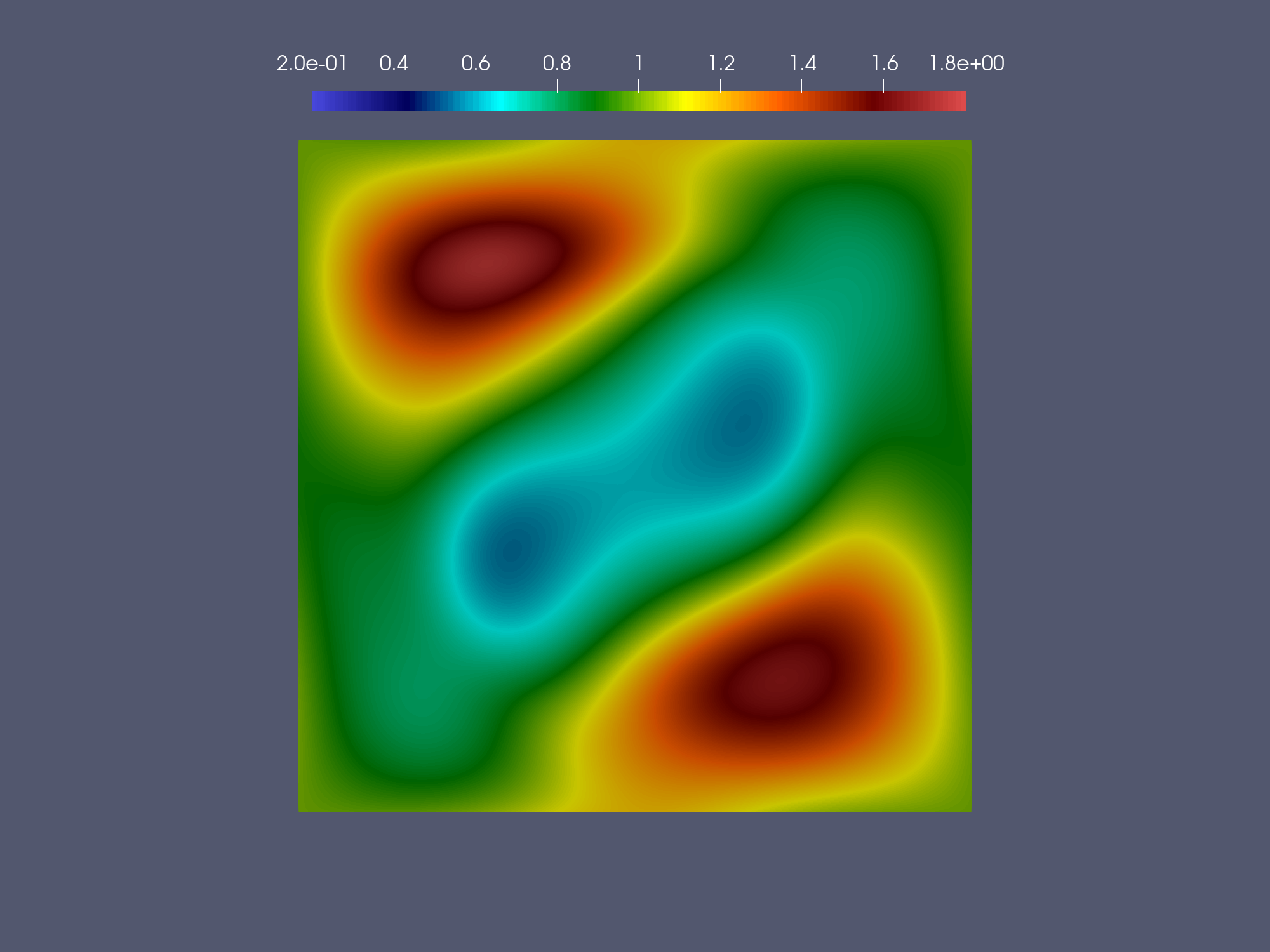} & \cbar{0.2}{1.8}{90bp}{0.2,0.6,1,1.4,1.8} \\ \noalign{\vskip 2pt}
    $\vtheta$ & \includegraphics[valign=c,trim={471bp 221bp 471bp 221bp},clip,width=96bp]{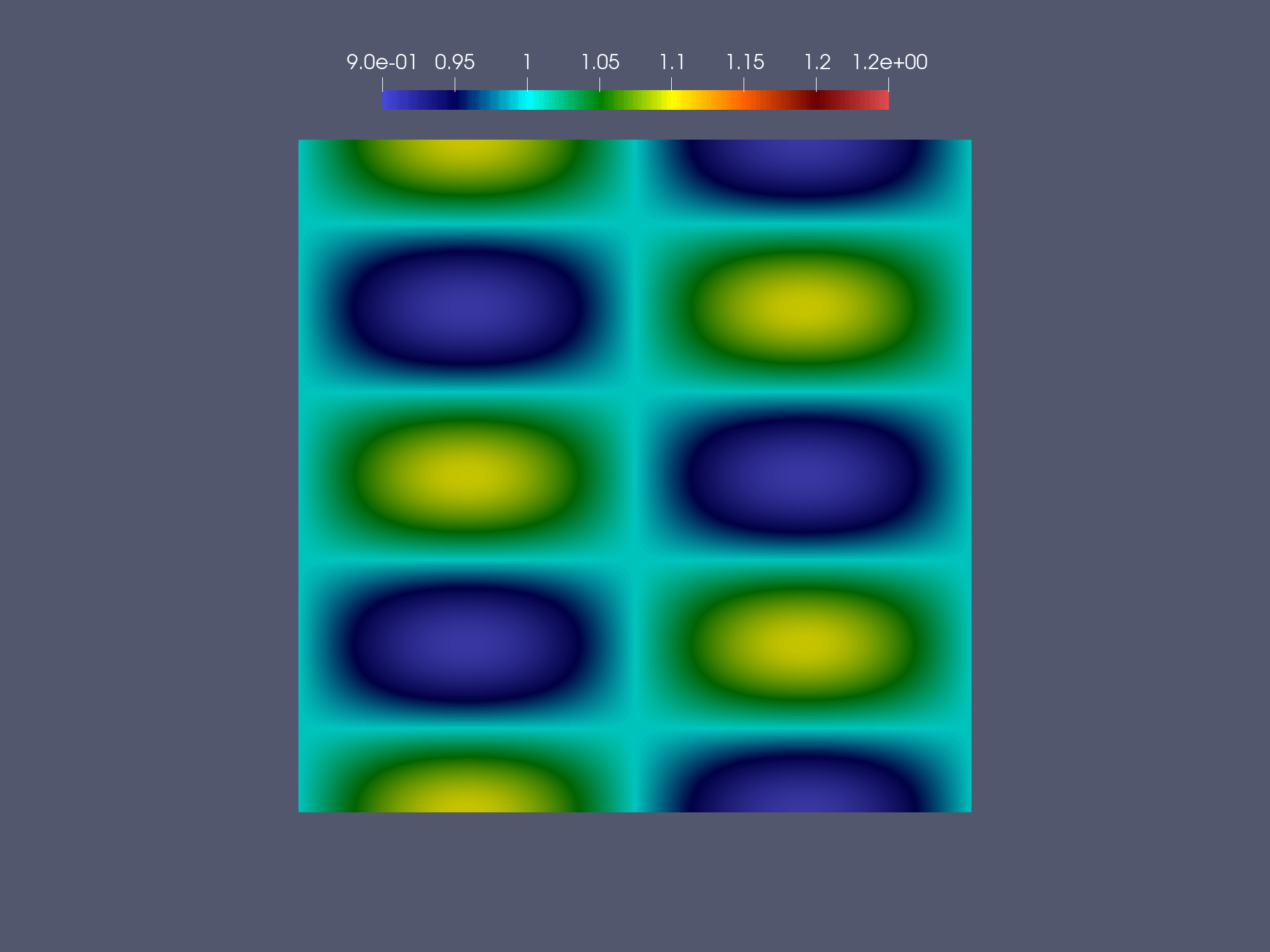} & \includegraphics[valign=c,trim={471bp 221bp 471bp 221bp},clip,width=96bp]{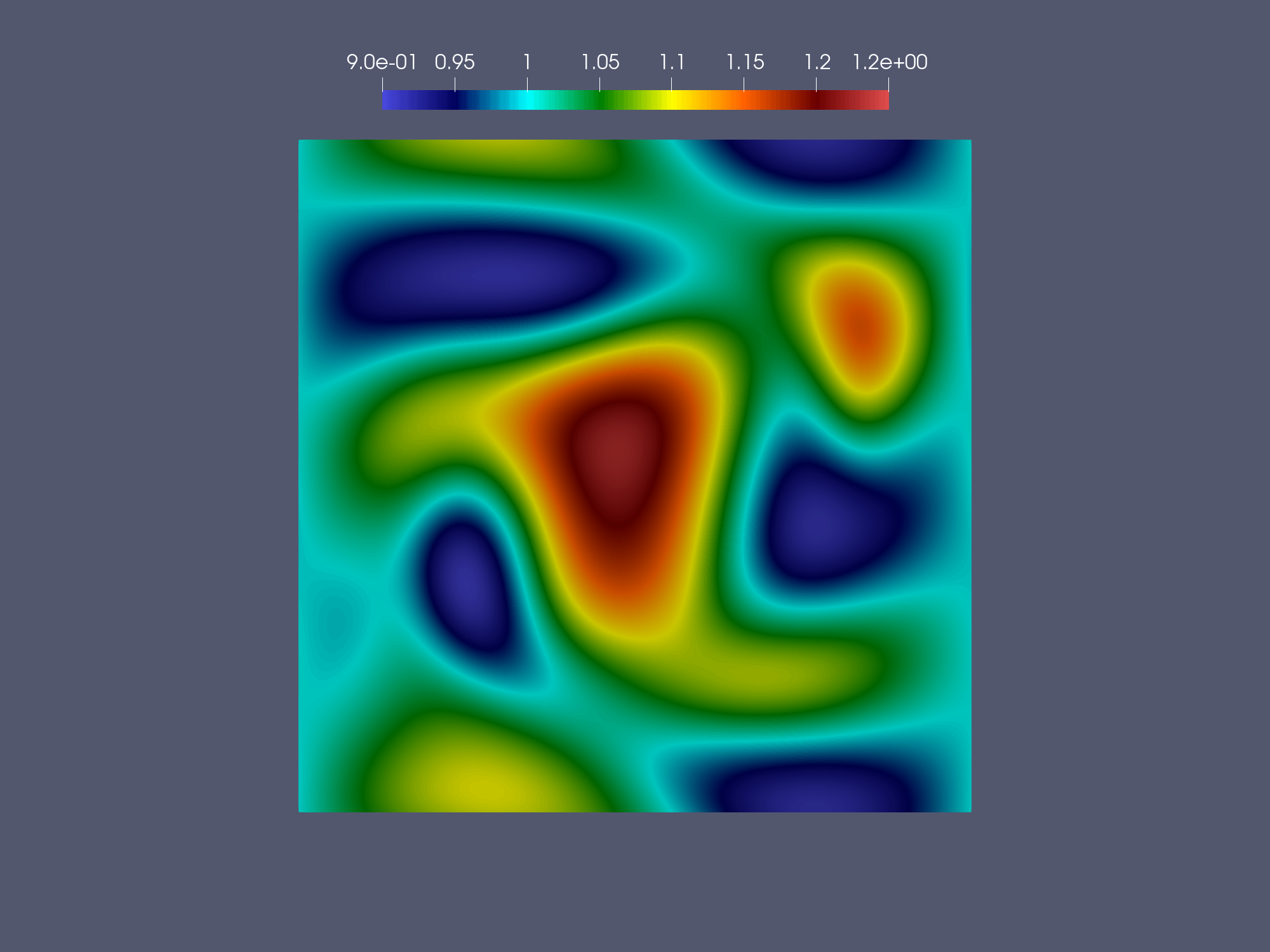} & \cbar{0.9}{1.2}{90bp}{0.9,1,1.1,1.2} \\ \noalign{\vskip 2pt}
    $\|\u\|$ & \includegraphics[valign=c,trim={471bp 221bp 471bp 221bp},clip,width=96bp]{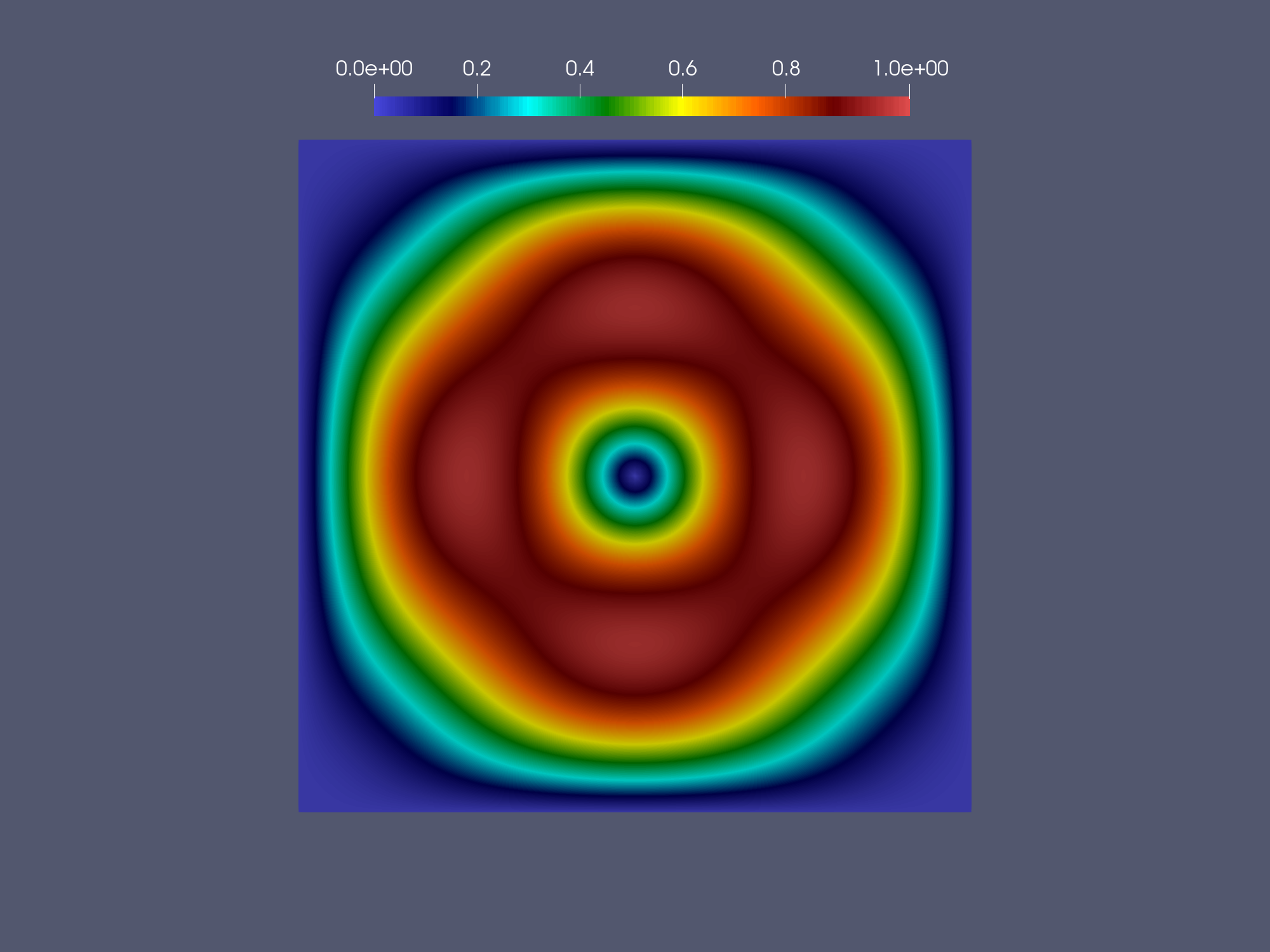} & \includegraphics[valign=c,trim={471bp 221bp 471bp 221bp},clip,width=96bp]{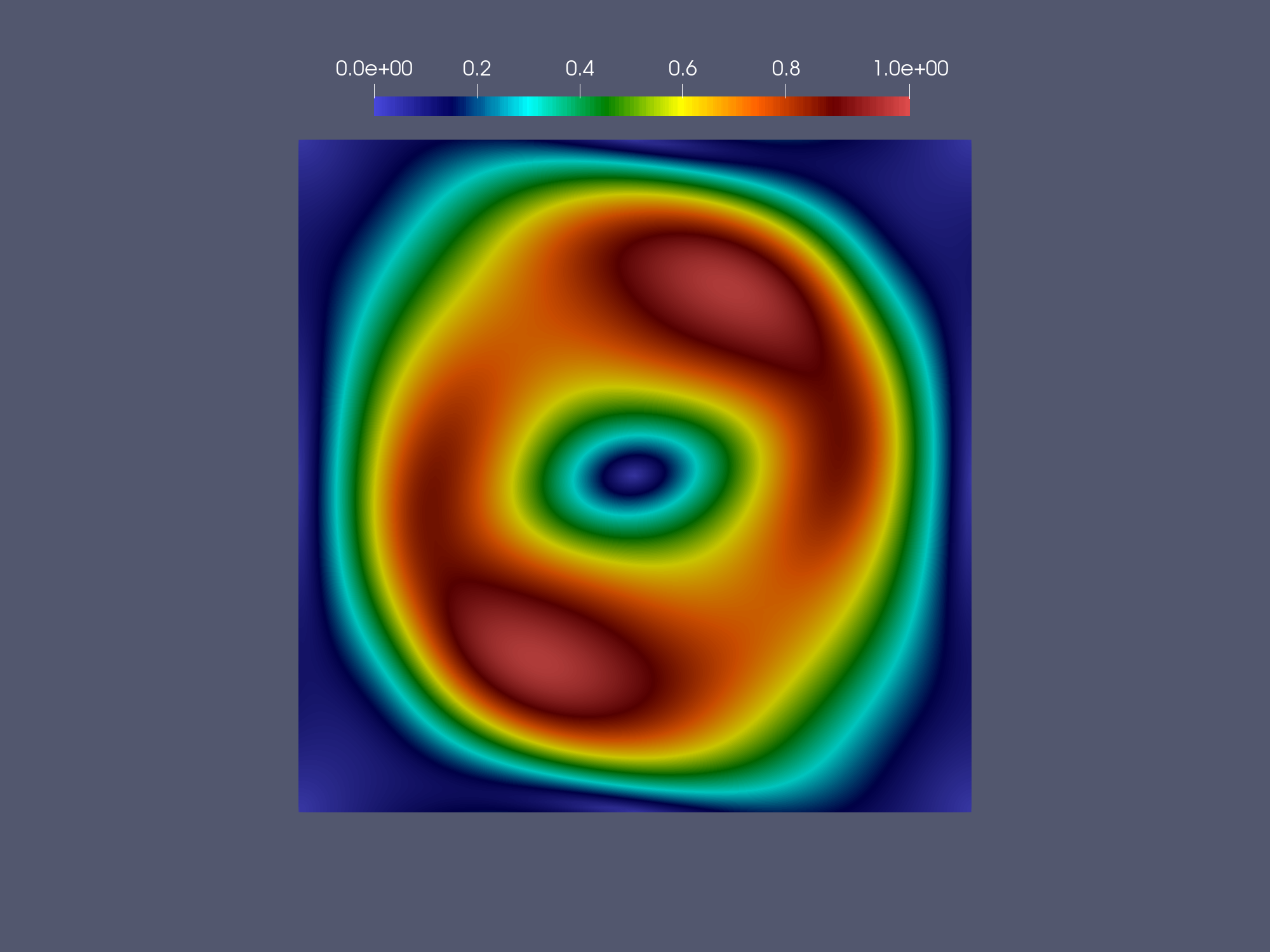} & \cbar{0}{1}{90bp}{0,0.5,1} \\ \noalign{\vskip 2pt}
\end{tabular}
\caption{Spatial convergence test: partial densities $\rho_1$ and $\rho_2$, temperature $\vtheta$, and velocity $\|\u\|$, at the initial and final times.}
\label{fig:conv_snapshots}

\end{figure}

Since no exact solution is available, we estimate errors by comparing the solutions on successive meshes $h_k$ and $h_{k+1}$, with the time step fixed at $\tau = 10^{-4}$.
Let $n_T=T/\tau$ denote the number of time steps.
For Scheme~I, we consider the errors
\begin{subequations}\label{eq:errA}
\begin{align}
    \text{err}(\vec{\rho},h_k)&\coloneqq \max_{0\le n\le n_T}\norm{\pvec{\rho}^n_{{h_k}}-\pvec{\rho}^n_{h_{k+1}}}_{L^2}, \label{eq:errA_rho}  \\
    \text{err}(\u,h_k)&\coloneqq\max_{0\le n\le n_T}\norm{\u^n_{{h_k}}-\u^n_{h_{k+1}}}_{L^2}, \label{eq:errA_u}  \\
    \text{err}(\vtheta,h_k)&\coloneqq\max_{0\le n\le n_T}\norm{\vtheta^n_{{h_k}}-\vtheta^n_{h_{k+1}}}_{L^2}, \label{eq:errA_theta}  \\
    \text{err}(\nabla\vec\mu,h_k)&\coloneqq\sqrt{\tau\sum\nolimits_{n=1}^{n_T}\norm{\pvec{\mu}^n_{h_k}-\pvec{\mu}^n_{h_{k+1}}}_{H^1}^2}, \label{eq:errA_mu}  \\
    \text{err}(\nabla\u,h_k)&\coloneqq\sqrt{\tau\sum\nolimits_{n=1}^{n_T}\norm{\u^n_{h_k}-\u^n_{h_{k+1}}}_{H^1}^2}, \label{eq:errA_gradu}  \\
    \text{err}(\nabla\vtheta,h_k)&\coloneqq\sqrt{\tau\sum\nolimits_{n=1}^{n_T}\norm{\vtheta^n_{h_k}-\vtheta^n_{h_{k+1}}}_{H^1}^2}, \label{eq:errA_gradtheta}  \\
    \text{err}(p,h_k)&\coloneqq\sqrt{\tau\sum\nolimits_{n=1}^{n_T}\norm{p_{h_k}^n-p^n_{h_{k+1}}}_{L^2}^2}. \label{eq:errA_p}
\end{align}
\end{subequations}
For Scheme~II, we replace $\u$ and $\vtheta$ by the continuous-in-time unknowns $\bm{m}$ and $\varrho s$ in the first three errors, and the sums over time steps by time integrals in the remainder:
\begin{subequations}\label{eq:errB}
\begin{align}
    \text{err}(\vec{\rho},h_k)&\coloneqq \max_{0\le n\le n_T}\norm{\pvec{\rho}^n_{{h_k}}-\pvec{\rho}^n_{h_{k+1}}}_{L^2}, \label{eq:errB_rho}  \\
    \text{err}(\bm{m},h_k)&\coloneqq\max_{0\le n\le n_T}\norm{\bm{m}^n_{{h_k}}-\bm{m}^n_{h_{k+1}}}_{L^2}, \label{eq:errB_m}  \\
    \text{err}(\varrho s,h_k)&\coloneqq\max_{0\le n\le n_T}\norm{(\varrho s)^n_{{h_k}}-(\varrho s)^n_{h_{k+1}}}_{L^2}, \label{eq:errB_s}  \\
    \text{err}(\nabla\vec\mu,h_k)&\coloneqq\sqrt{\int_0^T\norm{\pvec{\mu}_{h_k}-\pvec{\mu}_{h_{k+1}}}_{H^1}^2\,\dd t}, \label{eq:errB_mu}  \\
    \text{err}(\nabla\u,h_k)&\coloneqq\sqrt{\int_0^T\norm{\u_{h_k}-\u_{h_{k+1}}}_{H^1}^2\,\dd t}, \label{eq:errB_gradu}  \\
    \text{err}(\nabla\vtheta,h_k)&\coloneqq\sqrt{\int_0^T\norm{\vtheta_{h_k}-\vtheta_{h_{k+1}}}_{H^1}^2\,\dd t}, \label{eq:errB_gradtheta}  \\
    \text{err}(p,h_k)&\coloneqq\sqrt{\int_0^T\norm{p_{h_k}-p_{h_{k+1}}}_{L^2}^2\,\dd t}. \label{eq:errB_p}
\end{align}
\end{subequations}
We use meshes of size $h_k \approx 2^{-k}$, $k=1,\ldots,7$.
The errors and experimental orders of convergence (EOCs) are reported in \Cref{tab:conv_I} for Scheme~I and \Cref{tab:conv_II} for Scheme~II.
Both schemes converge at second order in $L^\infty(0,T;L^2(\Omega))$ for the partial densities $\vec\rho$, temperature $\vtheta$ and entropy density $\varrho s$ respectively, at second order in $L^2(0,T;L^2(\Omega))$ for the pressure $p$, and at first order in $L^2(0,T;H^1(\Omega))$ for $\vtheta$ and the chemical potentials $\vec\mu$.
The errors in the velocity $\u$ and modified momentum density $\bm{m}$ decay at higher (albeit less regular) rates, reflecting the higher polynomial degree of $\Xh$.

\def\phm{}
\begin{table}[htbp!]
\centering
\small
\caption{Spatial convergence test, Scheme~I (semi-implicit Euler): errors and experimental orders of convergence.}
\label{tab:conv_I}
\begin{subtable}{\textwidth}
\centering
\caption{Errors in $L^\infty(0,T;L^2(\Omega))$.}
\label{tab:conv_I_Linf}
\begin{tabular}{|c||c|c|c|c|c|c|}
		\hline
		$ k $ & $\text{err}(\vec\rho)$ & EOC & $\text{err}(\bm{u})$ & EOC & $\text{err}(\vtheta)$ & EOC  \\
		\hline
		1 & $2.92\cdot 10^{-1}$ & -- & $1.72\cdot 10^{-1}$ & -- & $4.08\cdot 10^{-2}$ & -- \\
		2 & $1.46\cdot 10^{-1}$ & $\phm0.99$ & $1.03\cdot 10^{-1}$ & $\phm0.75$ & $2.63\cdot 10^{-2}$ & $\phm0.64$ \\
		3 & $4.86\cdot 10^{-2}$ & $\phm1.59$ & $3.90\cdot 10^{-2}$ & $\phm1.40$ & $1.31\cdot 10^{-2}$ & $\phm1.00$ \\
		4 & $1.13\cdot 10^{-2}$ & $\phm2.11$ & $9.58\cdot 10^{-3}$ & $\phm2.03$ & $3.56\cdot 10^{-3}$ & $\phm1.89$ \\
		5 & $2.67\cdot 10^{-3}$ & $\phm2.08$ & $1.24\cdot 10^{-3}$ & $\phm2.95$ & $9.12\cdot 10^{-4}$ & $\phm1.96$ \\
		6 & $6.41\cdot 10^{-4}$ & $\phm2.06$ & $1.29\cdot 10^{-4}$ & $\phm3.27$ & $2.29\cdot 10^{-4}$ & $\phm1.99$ \\
		\hline
	\end{tabular}
\end{subtable}

\vspace{1em}
\begin{subtable}{\textwidth}
\centering
\caption{Errors in $L^2(0,T;H^1(\Omega))$ and $L^2(0,T;L^2(\Omega))$.}
\label{tab:conv_I_L2}
\begin{tabular}{|c||c|c|c|c|c|c|c|c|}
        \hline
		$ k $ & $\text{err}(\nabla\vec\mu)$ & EOC & $\text{err}(\nabla\bm{u})$ & EOC & $\text{err}(\nabla\vtheta)$ & EOC & $\text{err}(p)$ & EOC \\
		\hline
		1 &   $9.97\cdot 10^{-1}$ & -- & $8.35\cdot 10^{-1}$ & -- & $2.03\cdot 10^{-1}$ & --& $9.64\cdot 10^{-2}$ & --\\
		2 &   $7.34\cdot 10^{-1}$ & $\phm0.44$ & $5.62\cdot 10^{-1}$ & $\phm0.57$ &$1.44\cdot 10^{-1}$ & $\phm0.50$& $5.60\cdot 10^{-2}$ & $\phm0.78$\\
		3 &   $4.90\cdot 10^{-1}$ & $\phm0.58$ & $4.22\cdot 10^{-1}$ & $\phm0.41$ & $1.08\cdot 10^{-1}$ & $\phm0.42$& $1.49\cdot 10^{-2}$ & $\phm1.91$\\
		4 &  $2.49\cdot 10^{-1}$ & $\phm0.98$ & $2.27\cdot 10^{-1}$ & $\phm0.89$ & $5.61\cdot 10^{-2}$ & $\phm0.94$& $3.44\cdot 10^{-3}$ & $\phm2.12$\\
		5 &   $1.23\cdot 10^{-1}$ & $\phm1.02$ & $7.31\cdot 10^{-2}$ & $\phm1.64$ & $2.82\cdot 10^{-2}$ & $\phm0.99$& $8.47\cdot 10^{-4}$ & $\phm2.02$\\
		6 &   $6.12\cdot 10^{-2}$ & $\phm1.01$ & $1.31\cdot 10^{-2}$ & $\phm2.48$ & $1.41\cdot 10^{-2}$ & $\phm1.00$& $2.12\cdot 10^{-4}$ & $\phm2.00$\\
		\hline
	\end{tabular}
\end{subtable}
\end{table}

\begin{table}[htbp!]
\centering
\small
\caption{Spatial convergence test, Scheme~II (CPG): errors and experimental orders of convergence.}
\label{tab:conv_II}
\begin{subtable}{\textwidth}
\centering
\caption{Errors in $L^\infty(0,T;L^2(\Omega))$.}
\label{tab:conv_II_Linf}
\begin{tabular}{|c||c|c|c|c|c|c|}
		\hline
		$ k $ & $\text{err}(\vec\rho)$ & EOC & $\text{err}(\bm{m})$ & EOC & $\text{err}(\varrho s)$ & EOC  \\
		\hline
		1 & $6.15\cdot 10^{-1}$ & -- & $2.04\cdot 10^{-1}$ & -- & $2.28\cdot 10^{-1}$ & -- \\
		2 & $1.33\cdot 10^{-1}$ & $\phm2.20$ & $4.96\cdot 10^{-2}$ & $\phm2.04$ & $8.14\cdot 10^{-2}$ & $\phm1.49$ \\
		3 & $3.90\cdot 10^{-2}$ & $\phm1.78$ & $1.51\cdot 10^{-2}$ & $\phm1.72$ & $3.57\cdot 10^{-2}$ & $\phm1.19$ \\
		4 & $1.02\cdot 10^{-2}$ & $\phm1.93$ & $3.69\cdot 10^{-3}$ & $\phm2.03$ & $9.74\cdot 10^{-3}$ & $\phm1.87$ \\
		5 & $2.56\cdot 10^{-3}$ & $\phm2.00$ & $6.09\cdot 10^{-4}$ & $\phm2.60$ & $2.48\cdot 10^{-3}$ & $\phm1.97$ \\
		6 & $6.41\cdot 10^{-4}$ & $\phm2.00$ & $1.37\cdot 10^{-4}$ & $\phm2.16$ & $6.24\cdot 10^{-4}$ & $\phm1.99$ \\
		\hline
	\end{tabular}
\end{subtable}

\vspace{1em}
\begin{subtable}{\textwidth}
\centering
\caption{Errors in $L^2(0,T;L^2(\Omega))$.}
\label{tab:conv_II_L2}
\begin{tabular}{|c||c|c|c|c|c|c|c|c|}
		\hline
		$ k $ &  $\text{err}(\nabla\vec\mu)$ & EOC & $\text{err}(\nabla\bm{u})$ & EOC & $\text{err}(\nabla\vtheta)$ & EOC & $\text{err}(p)$ & EOC \\
		\hline
		1 &  $1.50\cdot 10^{0}$ & -- & $7.09\cdot 10^{-1}$ & -- & $2.34\cdot 10^{-1}$ & --& $1.75\cdot 10^{-1}$ & --\\
		2 &  $1.07\cdot 10^{0}$ & $\phm0.49$ & $2.49\cdot 10^{-1}$ & $\phm1.51$ & $1.60\cdot 10^{-1}$ & $\phm0.54$& $8.70\cdot 10^{-2}$ & $\phm1.01$\\
		3 &  $7.42\cdot 10^{-1}$ & $\phm0.53$ & $1.10\cdot 10^{-1}$ & $\phm1.18$ & $1.06\cdot 10^{-1}$ & $\phm0.60$& $2.89\cdot 10^{-2}$ & $\phm1.59$\\
		4 &  $4.00\cdot 10^{-1}$ & $\phm0.89$ & $5.84\cdot 10^{-2}$ & $\phm0.91$ & $5.53\cdot 10^{-2}$ & $\phm0.94$& $7.27\cdot 10^{-3}$ & $\phm1.99$\\
		5 &  $1.99\cdot 10^{-1}$ & $\phm1.01$ & $1.74\cdot 10^{-2}$ & $\phm1.75$ & $2.79\cdot 10^{-2}$ & $\phm0.99$& $1.81\cdot 10^{-3}$ & $\phm2.01$\\
		6 &  $9.92\cdot 10^{-2}$ & $\phm1.00$ & $2.84\cdot 10^{-3}$ & $\phm2.61$ & $1.40\cdot 10^{-2}$ & $\phm1.00$& $4.53\cdot 10^{-4}$ & $\phm2.00$\\
		\hline
	\end{tabular}
\end{subtable}
\end{table}

\subsubsection{Structure-preserving test}\label{sec:struc}
We consider now a mixture of $N=3$ components, based on the experiment in \cite{Brunk2026}, with $\bar M = 1$, $\bar\eta = 10^{-2}$ and specific volumes $V_1=V_2=0.35$ and $V_3=0.8$.
We take the final time $T=0.1$, the time step $\tau=10^{-4}$ and the maximal mesh size $h=2^{-7}$, with initial data
\begin{subequations}\label{eq:struc_ic}
\begin{align}
    \rho_1^0(x,y) &= 0.2 + 0.9\bigl(\chi_{A_y} + \chi_{A_x} - \chi_{B_1} - \chi_{B_2}\bigr)(x,y), \label{eq:struc_ic_rho1}  \\
    \rho_2^0(x,y) &= 0.2 + 0.9\chi_C(x,y), \label{eq:struc_ic_rho2}  \\
    \rho_3^0(x,y) &= V_3^{-1}\big(1-V_1\rho_1^0(x,y) - V_2\rho_2^0(x,y)\big), \label{eq:struc_ic_rho3}  \\
    \u^0(x,y) &= \bigl(-\sin(\pi x)^2\sin(2\pi y),\ \sin(2\pi x)\sin(\pi y)^2\bigr)^\top, \label{eq:struc_ic_u}  \\
    \vtheta^0(x,y) &= 1 + 0.1\sin(2\pi x)\cos(2\pi y). \label{eq:struc_ic_theta}
\end{align}
\end{subequations}
Here $\chi_A$ denotes the characteristic function of a set $A\subset\Omega$, and
\begin{subequations}\label{eq:struc_sets}
\begin{align}
    A_y &= \{(x,y)\in\Omega : y\in[0.1,0.2]\cup[0.8,0.9]\}, \label{eq:struc_sets_Ay}  \\
    A_x &= \{(x,y)\in\Omega : x\in[0.1,0.2]\cup[0.8,0.9]\}, \label{eq:struc_sets_Ax}  \\
    B_1 &= [0.1,0.2]\times\bigl([0.1,0.2]\cup[0.8,0.9]\bigr), \label{eq:struc_sets_B1}  \\
    B_2 &= [0.8,0.9]\times\bigl([0.1,0.2]\cup[0.8,0.9]\bigr), \label{eq:struc_sets_B2}  \\
    C &= \{(x,y)\in\Omega : (x-0.5)^2+(y-0.5)^2\leq0.25^2\}. \label{eq:struc_sets_C}
\end{align}
\end{subequations}
As in \Cref{sec:conv}, the initial entropy density for Scheme~II is computed from these data.

Snapshots of the partial densities, temperature and velocity are shown in \Cref{fig:rho,fig:velotemp}.
The two schemes produce qualitatively similar dynamics: the initial layers of $\rho_1$ and the disc of $\rho_2$ are advected by the rotating flow and smoothed by diffusion, with small differences between the schemes visible at later times, most noticeably in $\rho_3$.
\Cref{fig:meta,fig:masses} confirm the structure-preservation results.
The total energy of Scheme~I decreases monotonically, as guaranteed by \Cref{thm:schemeA_global}, while that of Scheme~II is conserved to within about $10^{-13}$, as guaranteed by \Cref{thm:schemeB_global}\footnote{
    The initial total energy and entropy of the two schemes therefore differ slightly, due to projection errors.
}.
For both schemes, the total entropy increases, the partial masses are conserved to within about $10^{-13}$, and the quasi-incompressibility constraint holds to within $10^{-14}$ in $L^\infty(\Omega)$, in agreement with \Cref{thm:schemeA_incom,thm:schemeB_incom}.

\begin{figure}[!htbp]
\centering
\scalebox{0.86}{%
\begin{tabular}{c@{\hspace{0.5em}}c@{\hspace{0.5em}}c@{\hspace{2pt}}c@{\hspace{2pt}}c@{\hspace{2pt}}c@{\hspace{2pt}}c@{\hspace{3pt}}l}
    & & \small $t = 0.01$ & \small $t = 0.02$ & \small $t = 0.04$ & \small $t = 0.06$ & \small $t = 0.1$ \\[3pt]
    \raisebox{-35.1bp}[0pt][0pt]{$\rho_1$} & \small I & \includegraphics[valign=c,trim={471bp 221bp 471bp 221bp},clip,width=68.25bp]{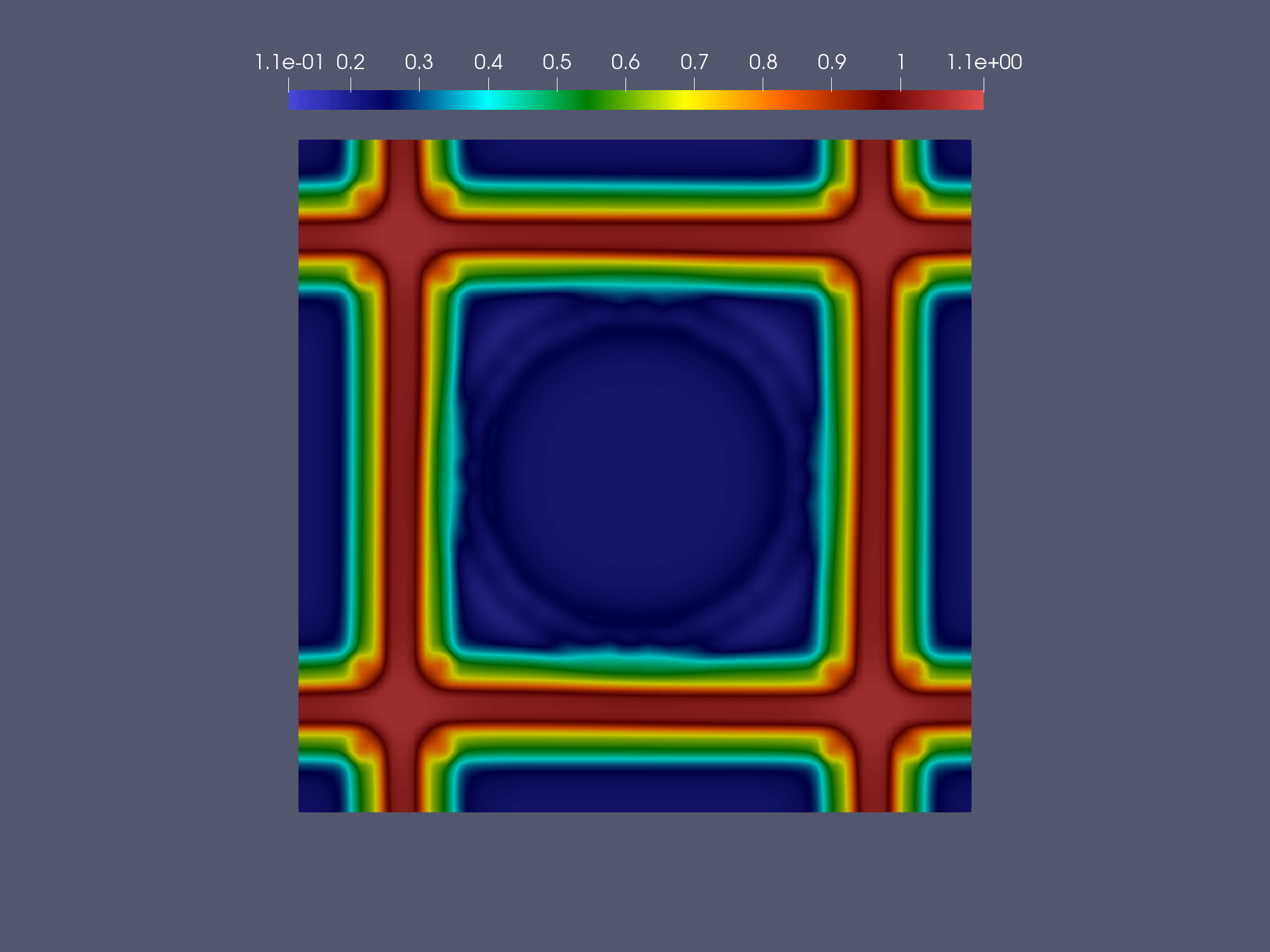} & \includegraphics[valign=c,trim={471bp 221bp 471bp 221bp},clip,width=68.25bp]{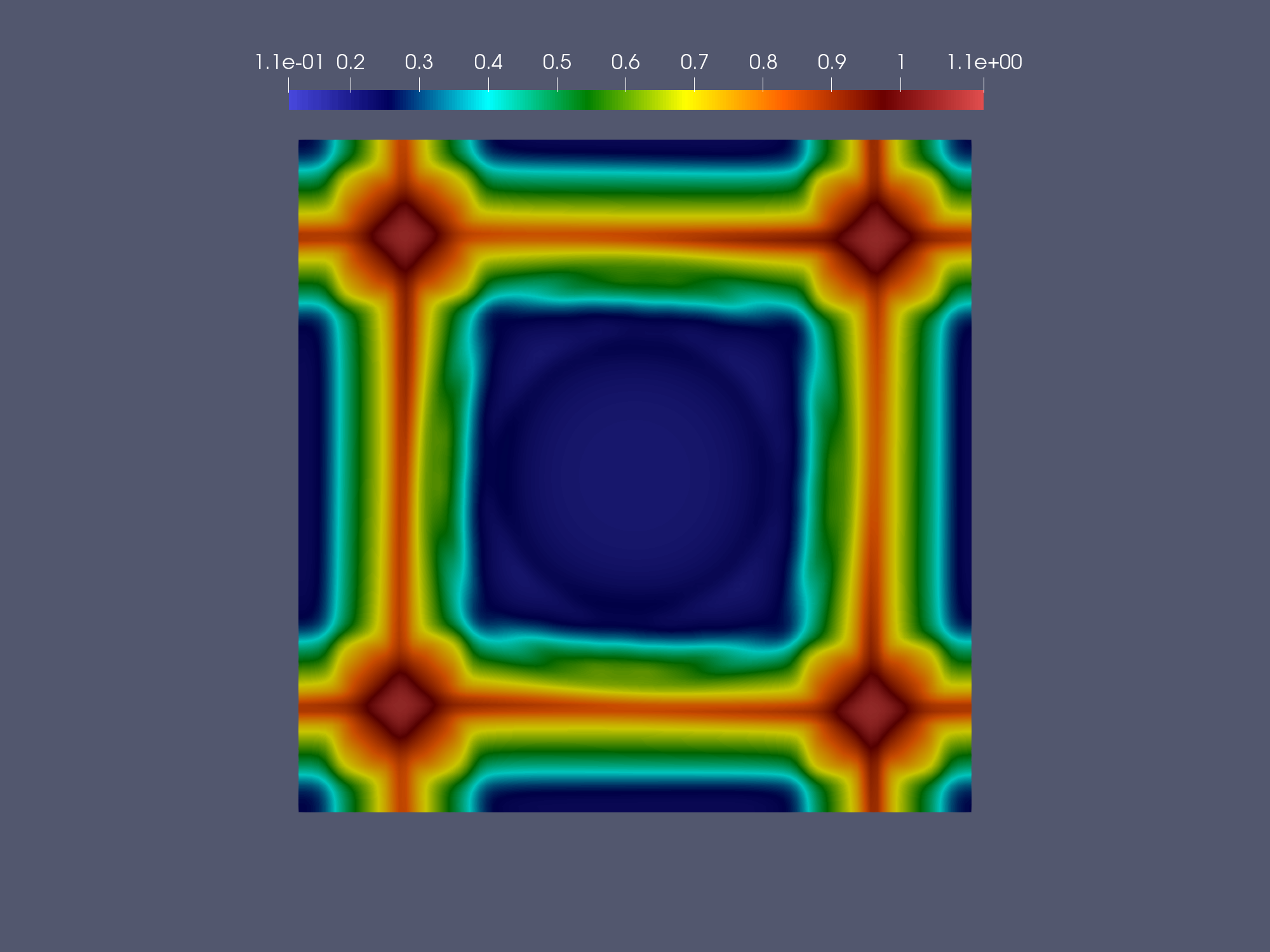} & \includegraphics[valign=c,trim={471bp 221bp 471bp 221bp},clip,width=68.25bp]{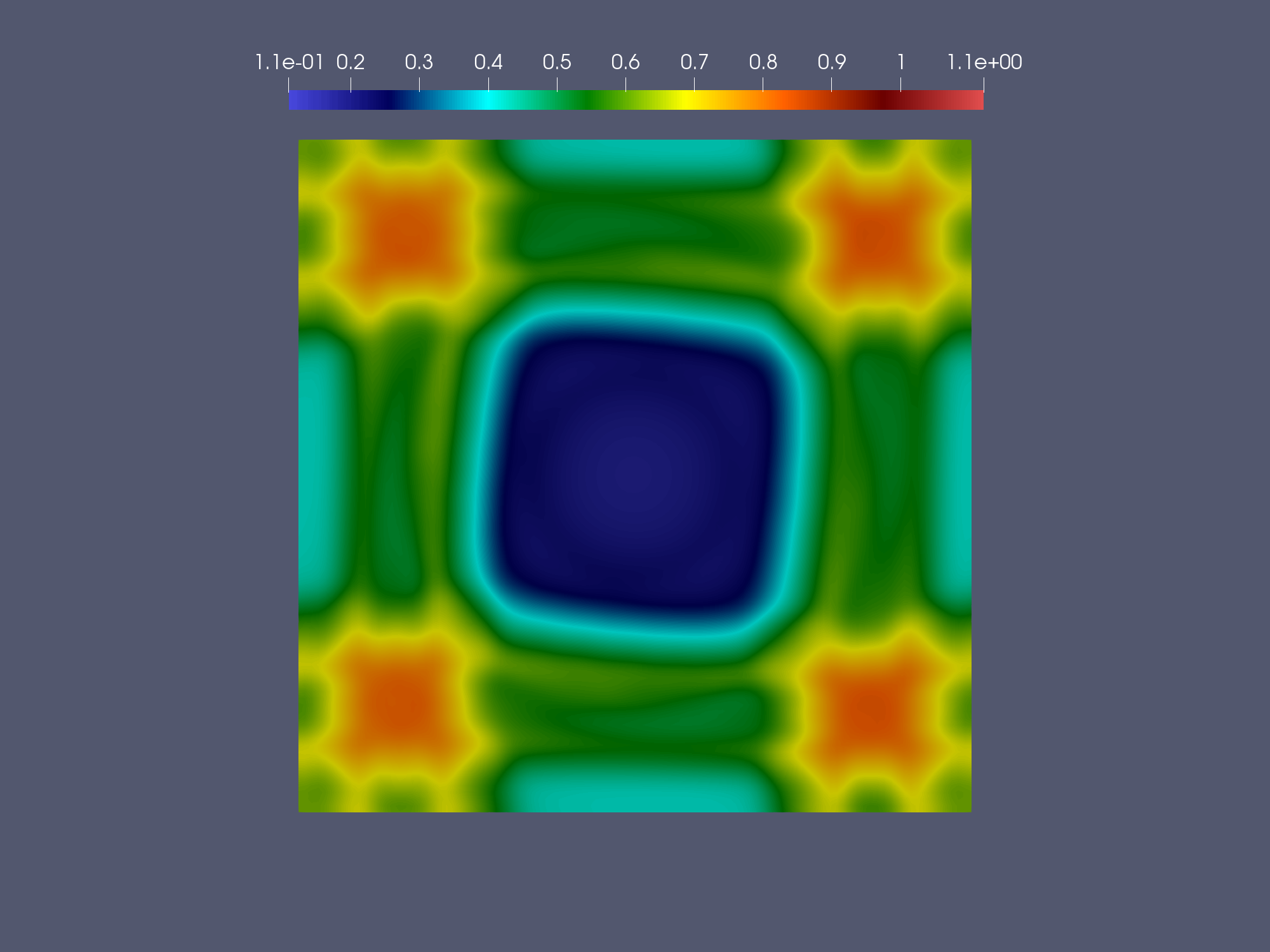} & \includegraphics[valign=c,trim={471bp 221bp 471bp 221bp},clip,width=68.25bp]{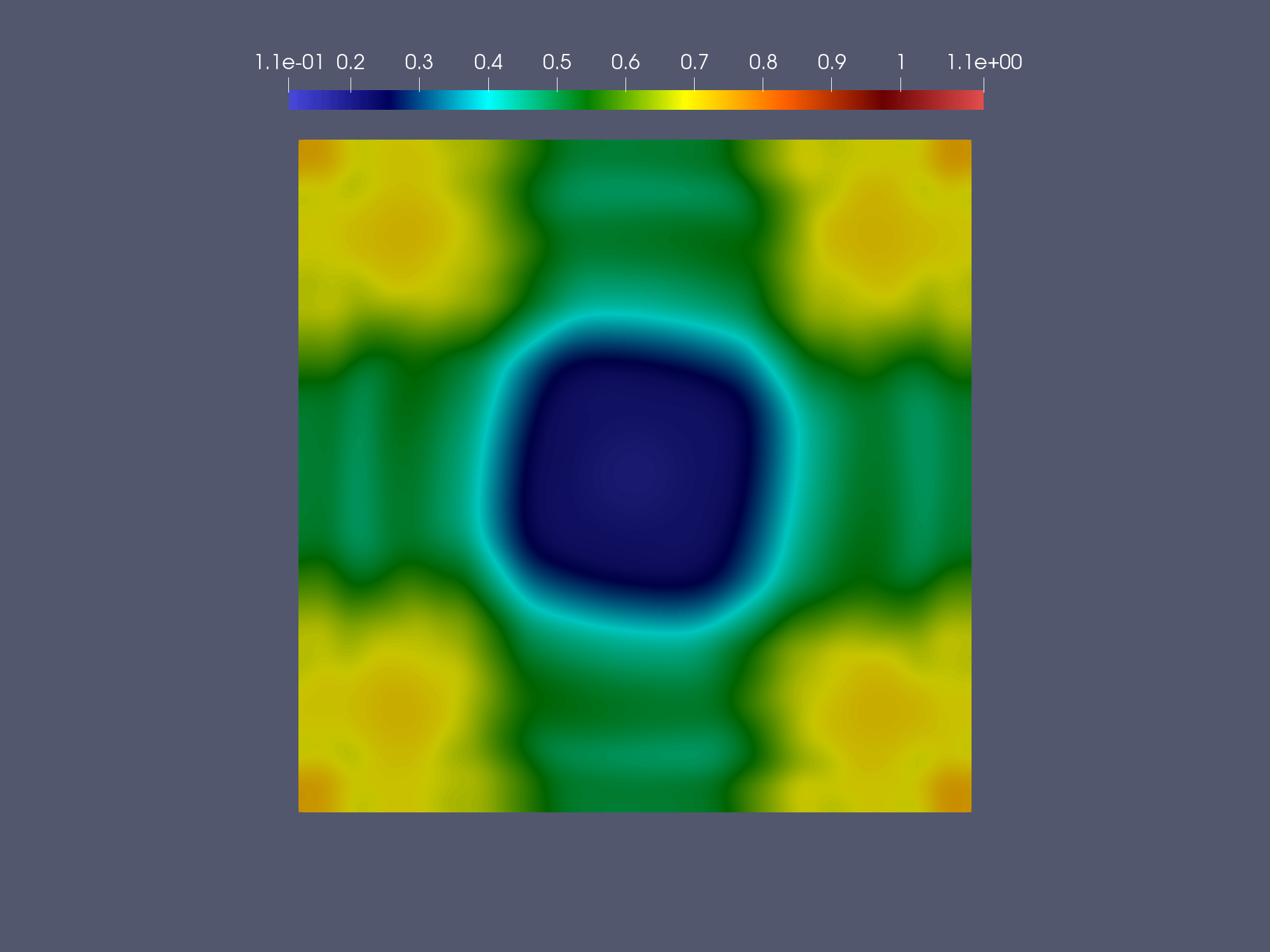} & \includegraphics[valign=c,trim={471bp 221bp 471bp 221bp},clip,width=68.25bp]{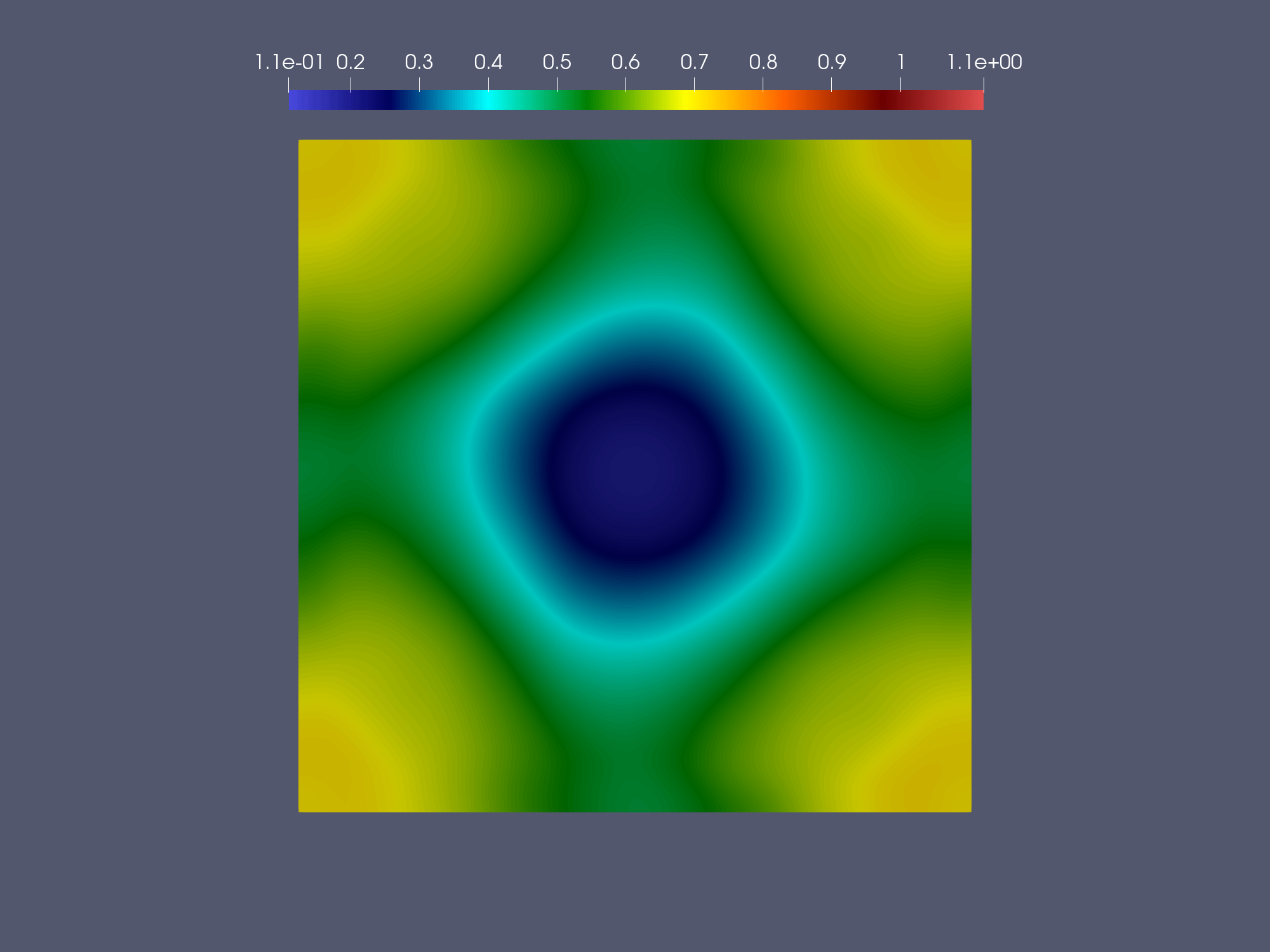} & \raisebox{-35.1bp}[0pt][0pt]{\cbar{0.11}{1.1}{134bp}{0.2,0.4,0.6,0.8,1}} \\ \noalign{\vskip 2pt}
     & \small II & \includegraphics[valign=c,trim={471bp 221bp 471bp 221bp},clip,width=68.25bp]{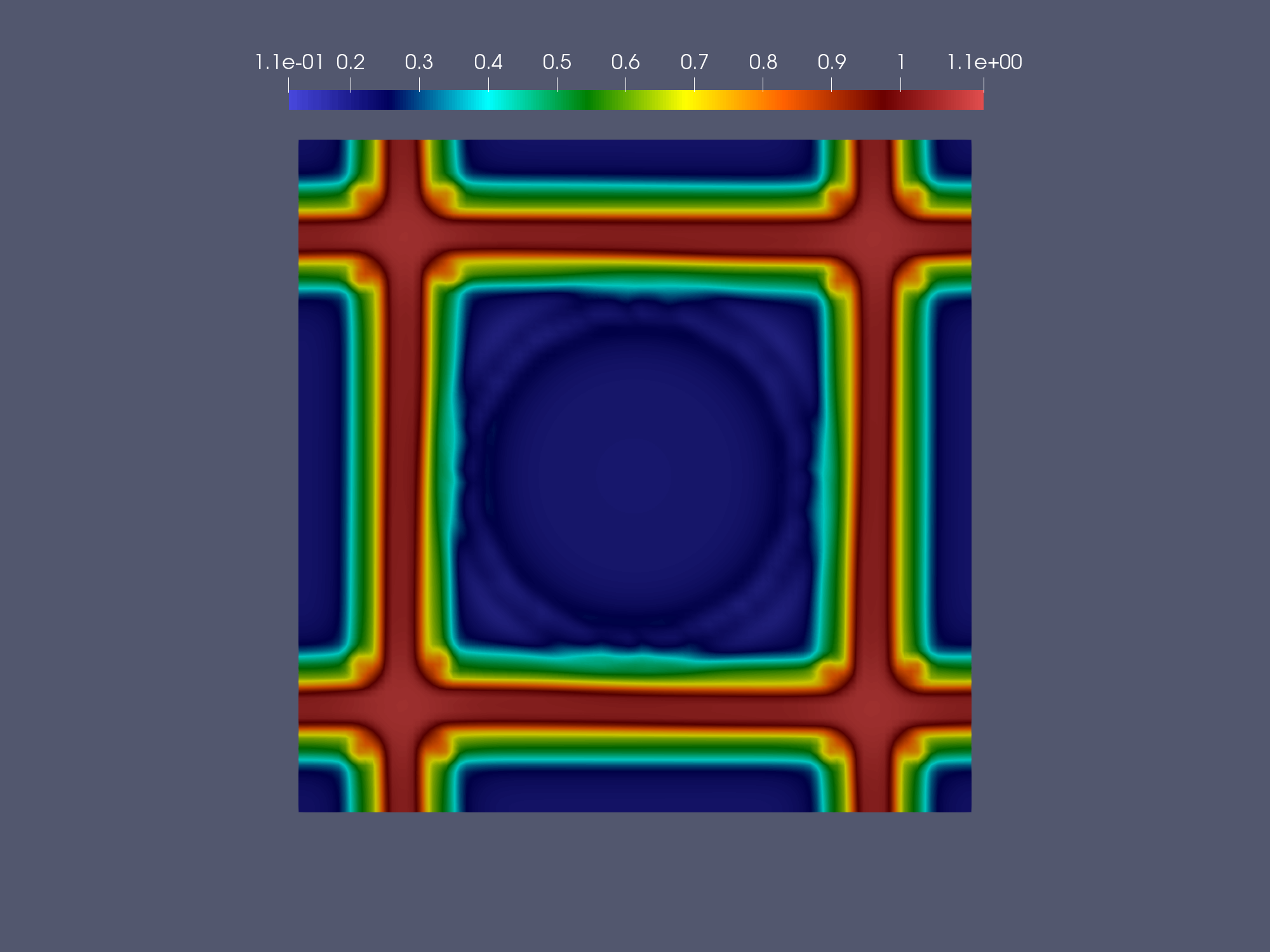} & \includegraphics[valign=c,trim={471bp 221bp 471bp 221bp},clip,width=68.25bp]{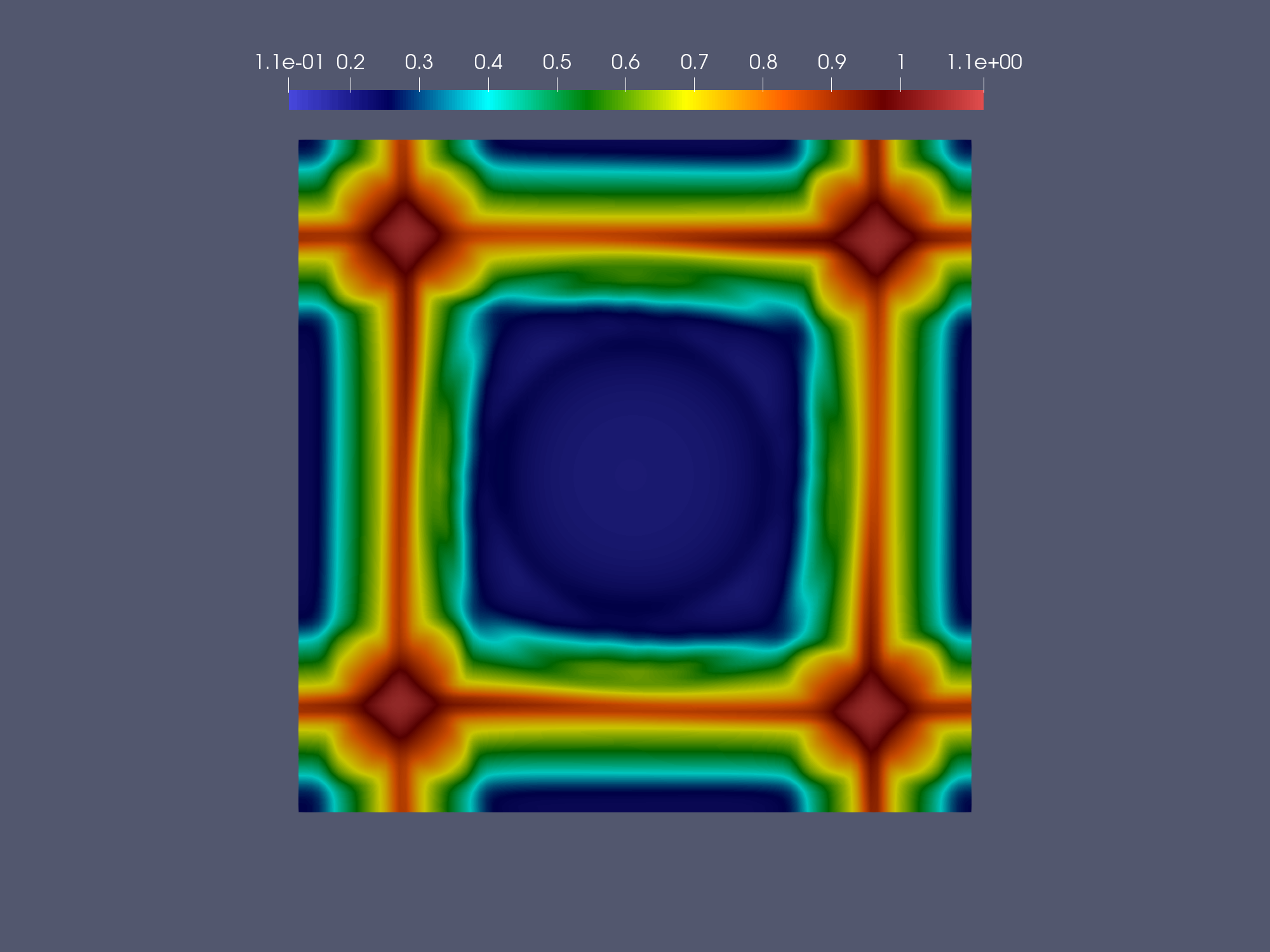} & \includegraphics[valign=c,trim={471bp 221bp 471bp 221bp},clip,width=68.25bp]{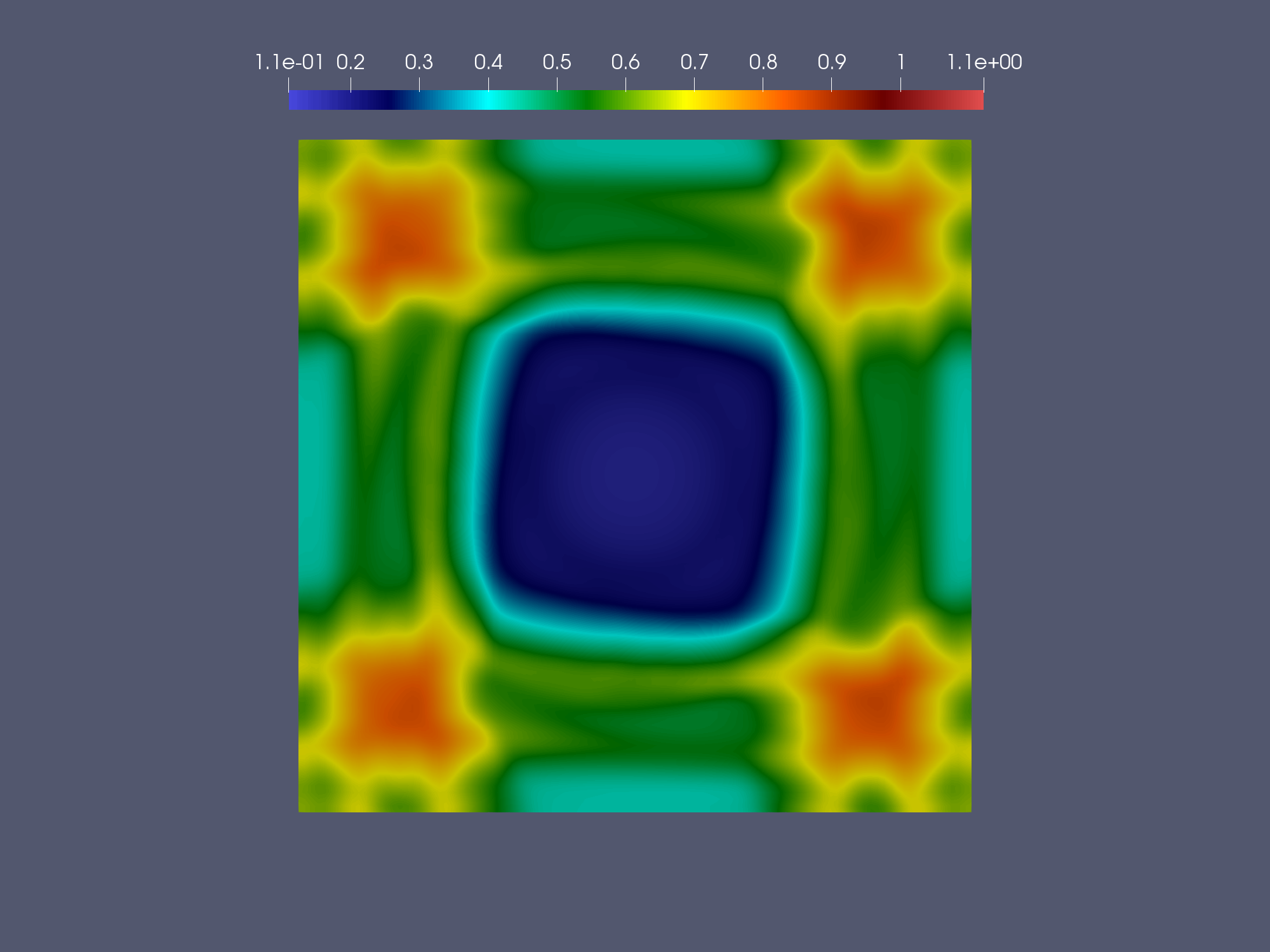} & \includegraphics[valign=c,trim={471bp 221bp 471bp 221bp},clip,width=68.25bp]{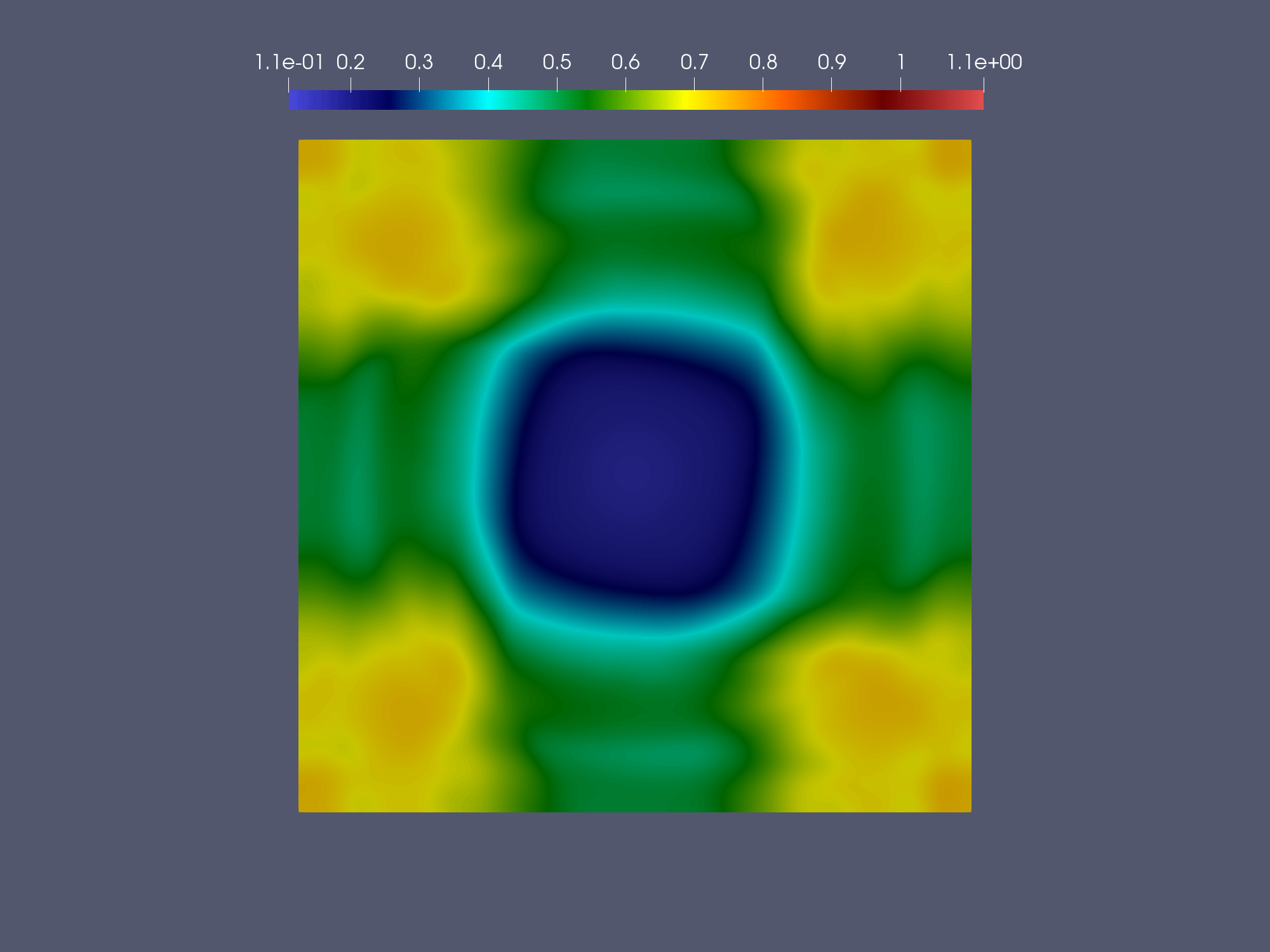} & \includegraphics[valign=c,trim={471bp 221bp 471bp 221bp},clip,width=68.25bp]{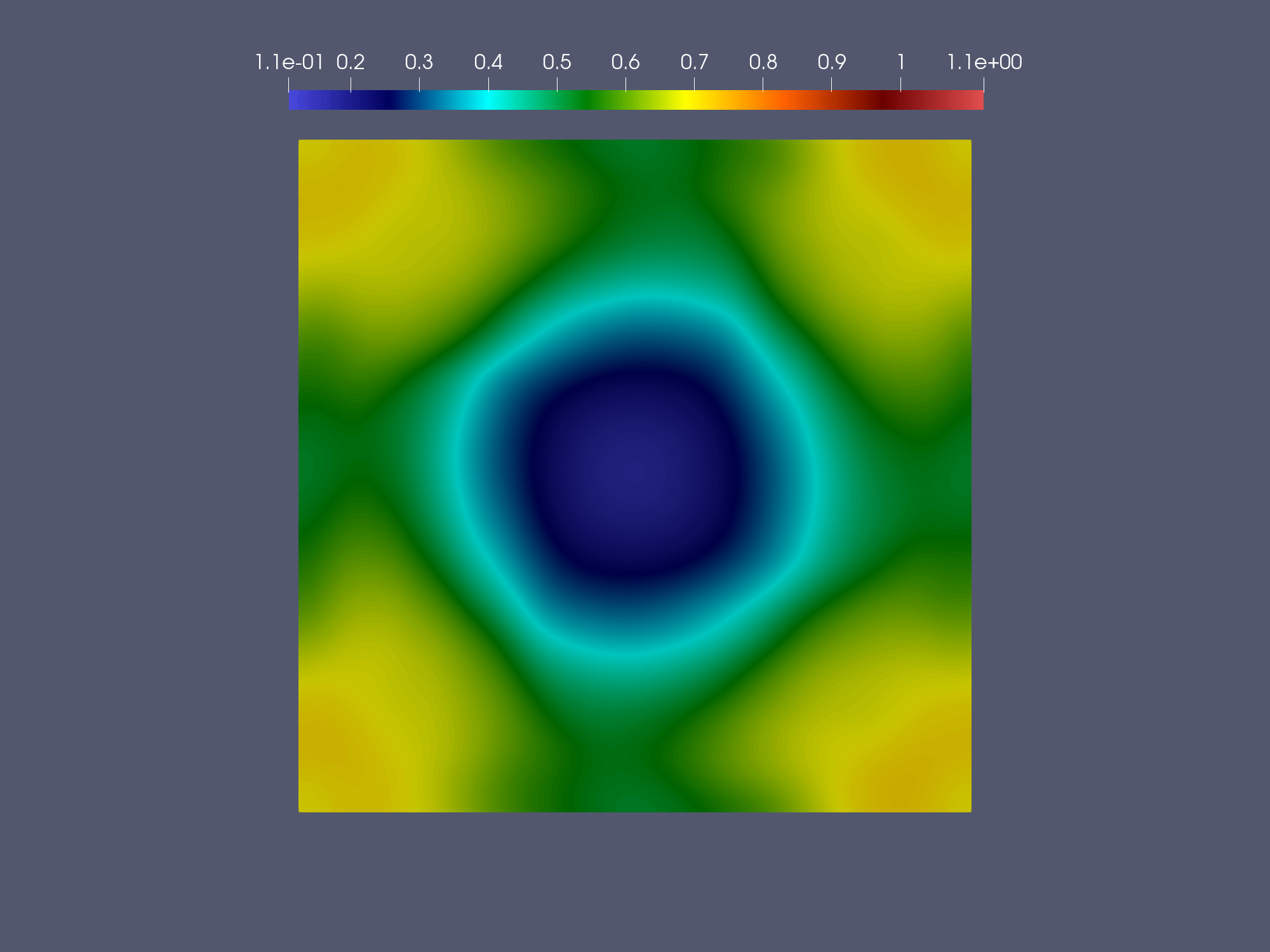} \\ \noalign{\vskip 6pt}
    \raisebox{-35.1bp}[0pt][0pt]{$\rho_2$} & \small I & \includegraphics[valign=c,trim={471bp 221bp 471bp 221bp},clip,width=68.25bp]{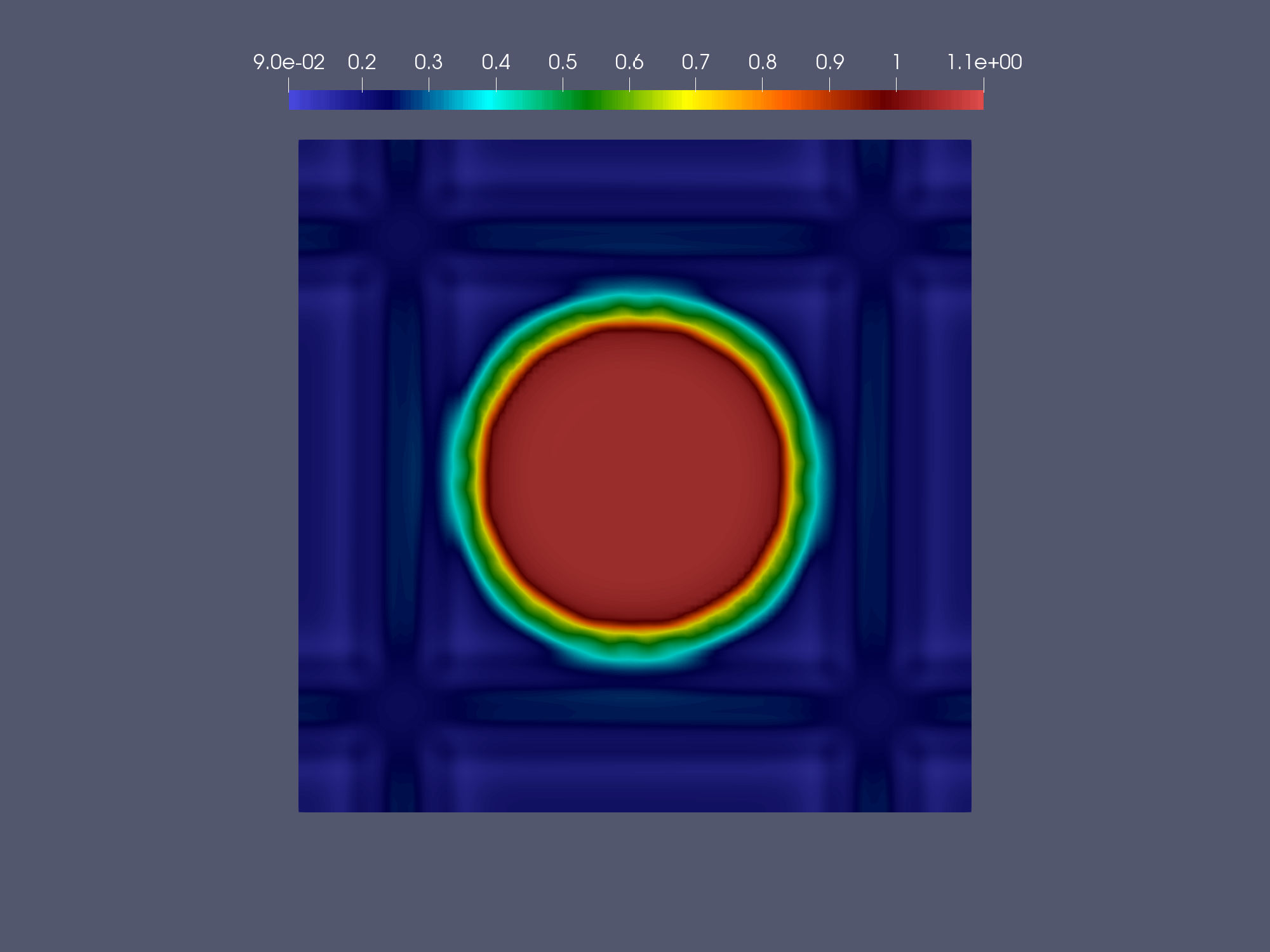} & \includegraphics[valign=c,trim={471bp 221bp 471bp 221bp},clip,width=68.25bp]{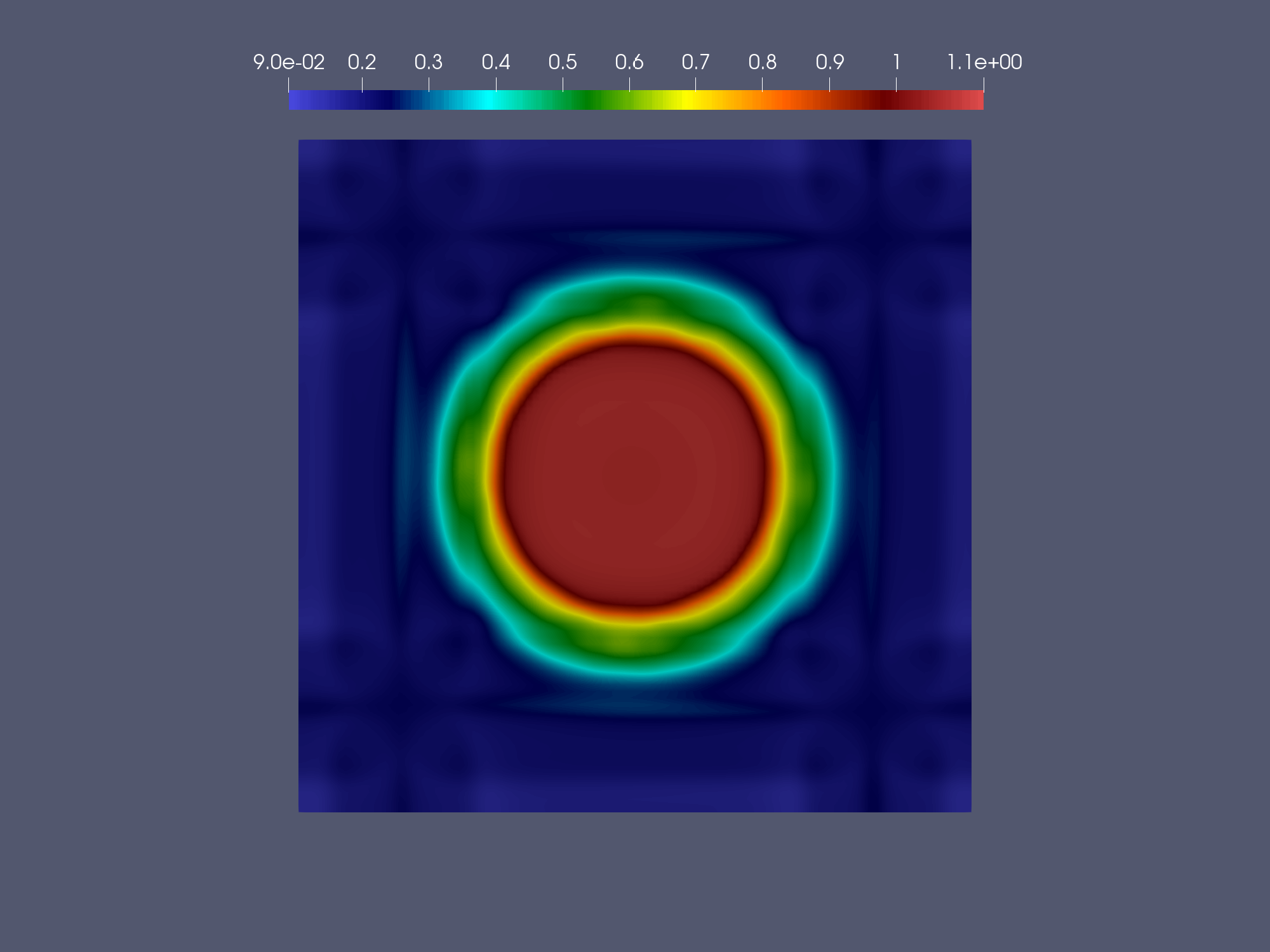} & \includegraphics[valign=c,trim={471bp 221bp 471bp 221bp},clip,width=68.25bp]{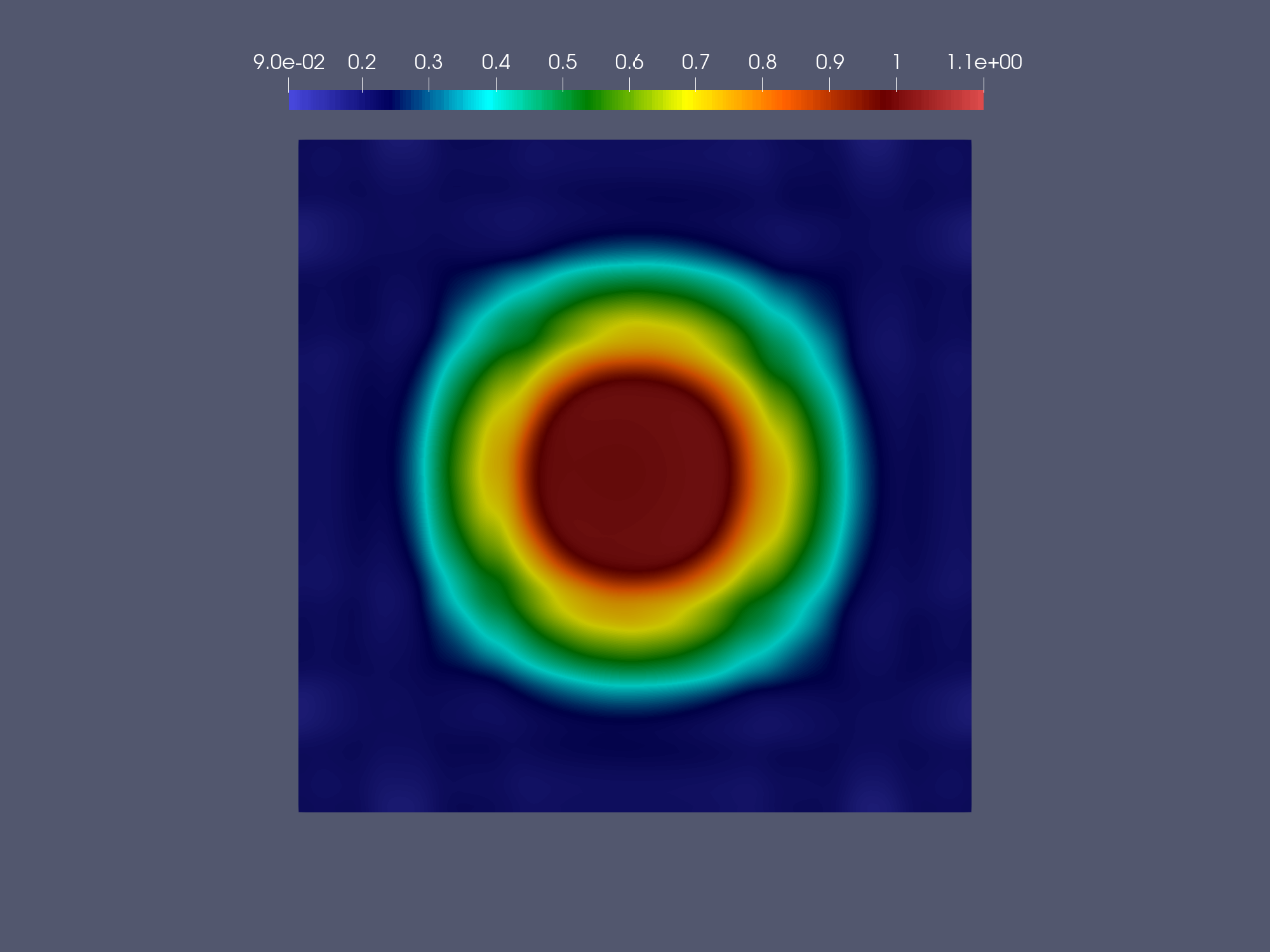} & \includegraphics[valign=c,trim={471bp 221bp 471bp 221bp},clip,width=68.25bp]{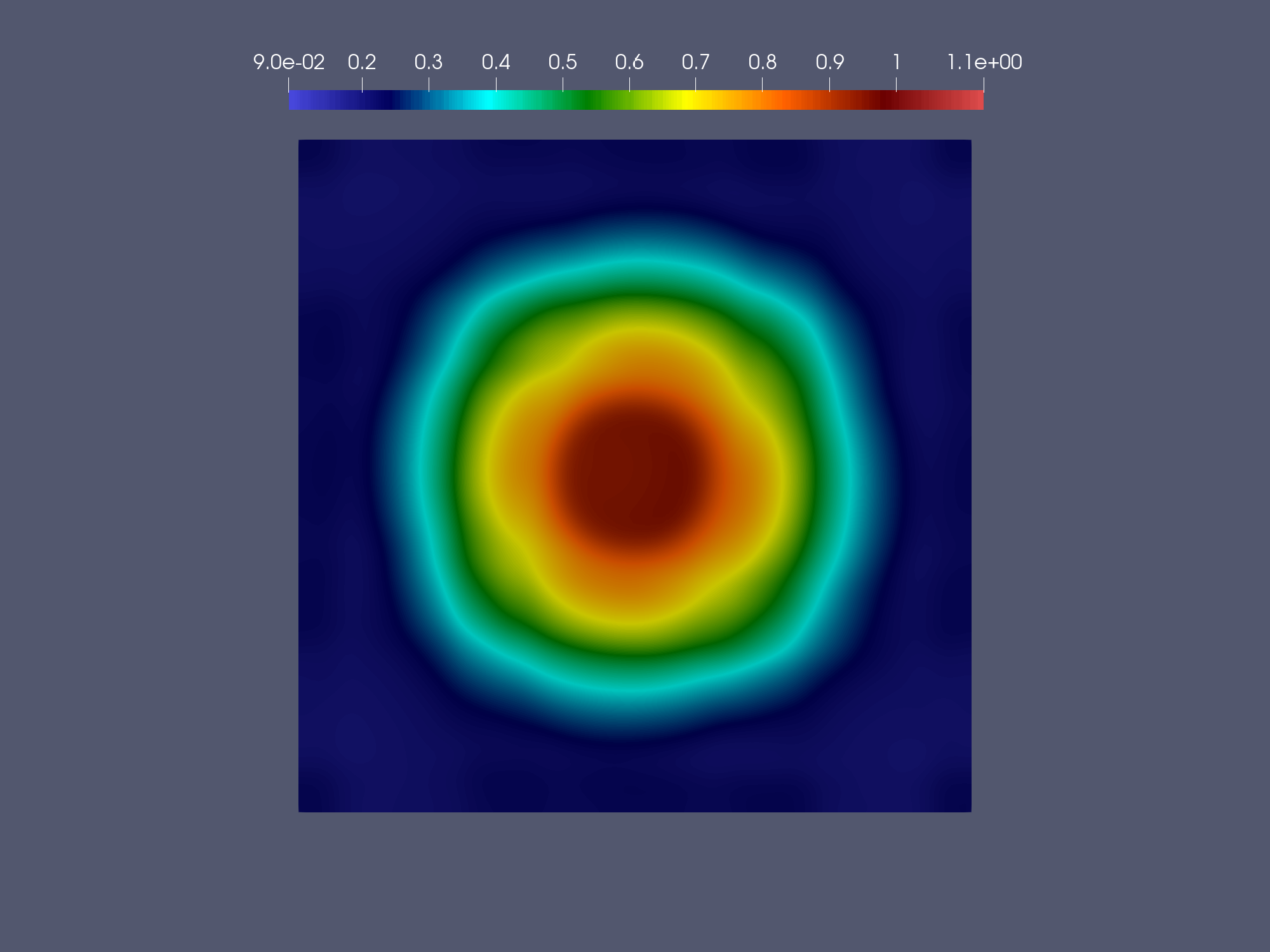} & \includegraphics[valign=c,trim={471bp 221bp 471bp 221bp},clip,width=68.25bp]{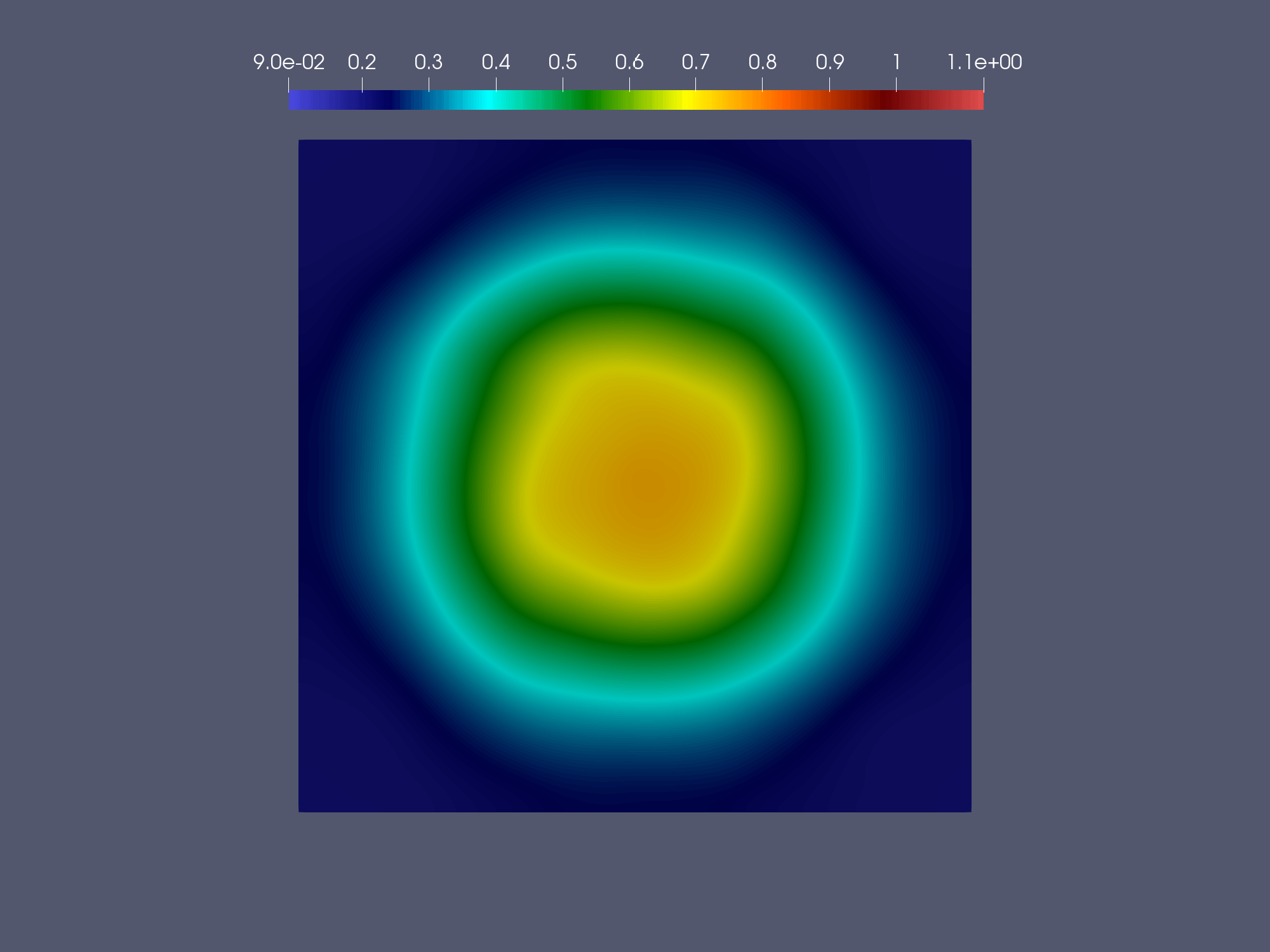} & \raisebox{-35.1bp}[0pt][0pt]{\cbar{0.09}{1.1}{134bp}{0.2,0.4,0.6,0.8,1}} \\ \noalign{\vskip 2pt}
     & \small II & \includegraphics[valign=c,trim={471bp 221bp 471bp 221bp},clip,width=68.25bp]{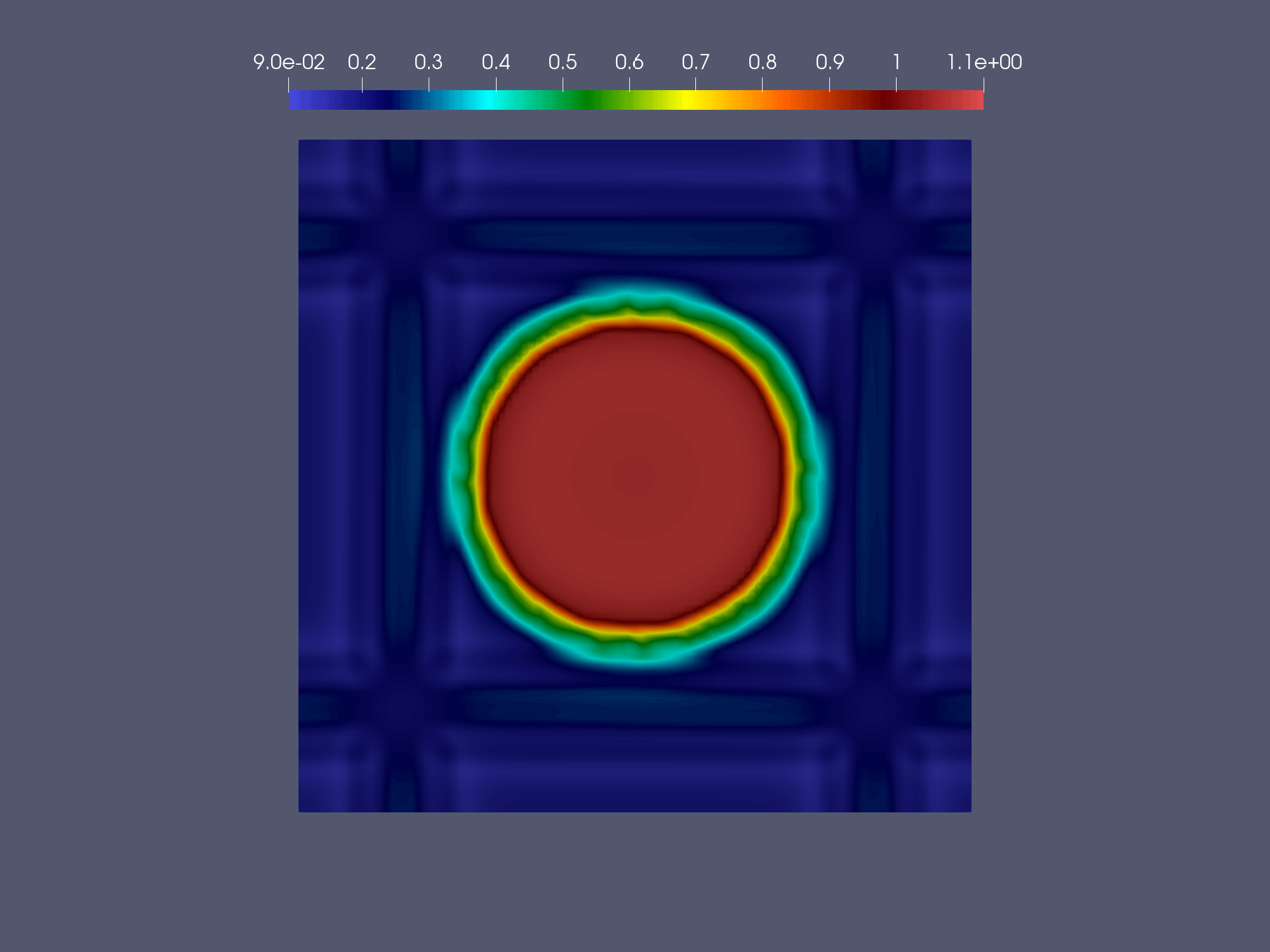} & \includegraphics[valign=c,trim={471bp 221bp 471bp 221bp},clip,width=68.25bp]{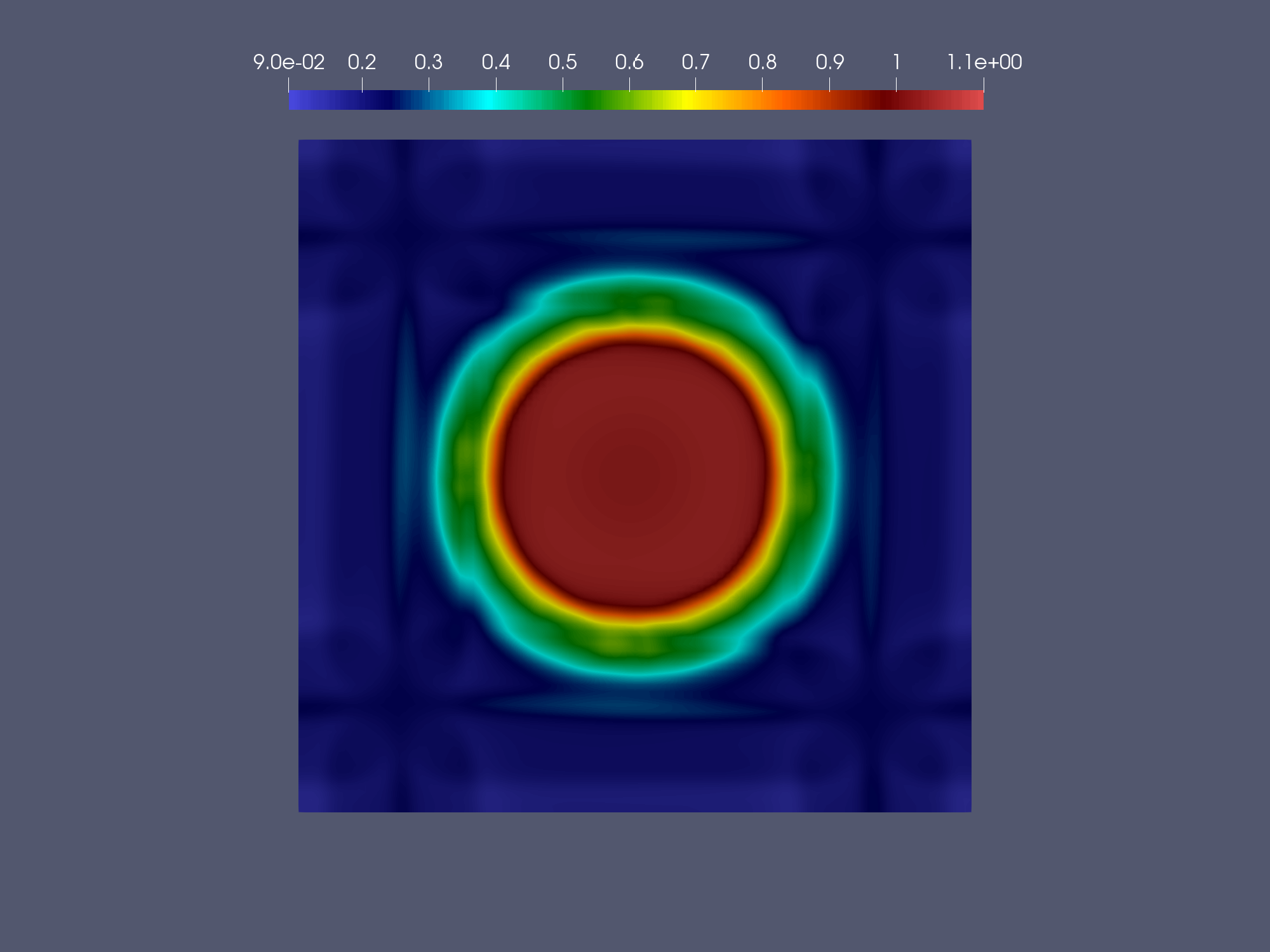} & \includegraphics[valign=c,trim={471bp 221bp 471bp 221bp},clip,width=68.25bp]{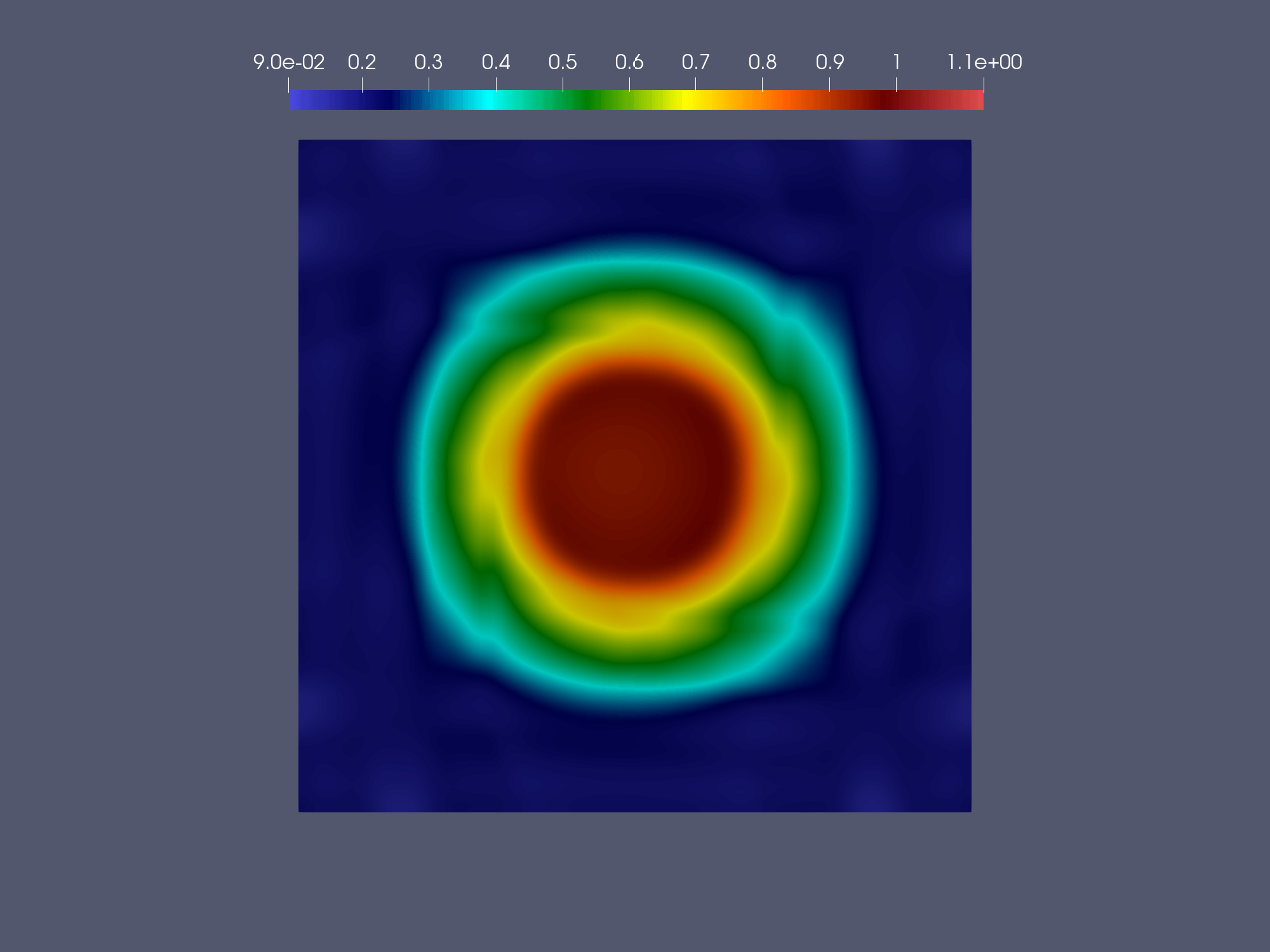} & \includegraphics[valign=c,trim={471bp 221bp 471bp 221bp},clip,width=68.25bp]{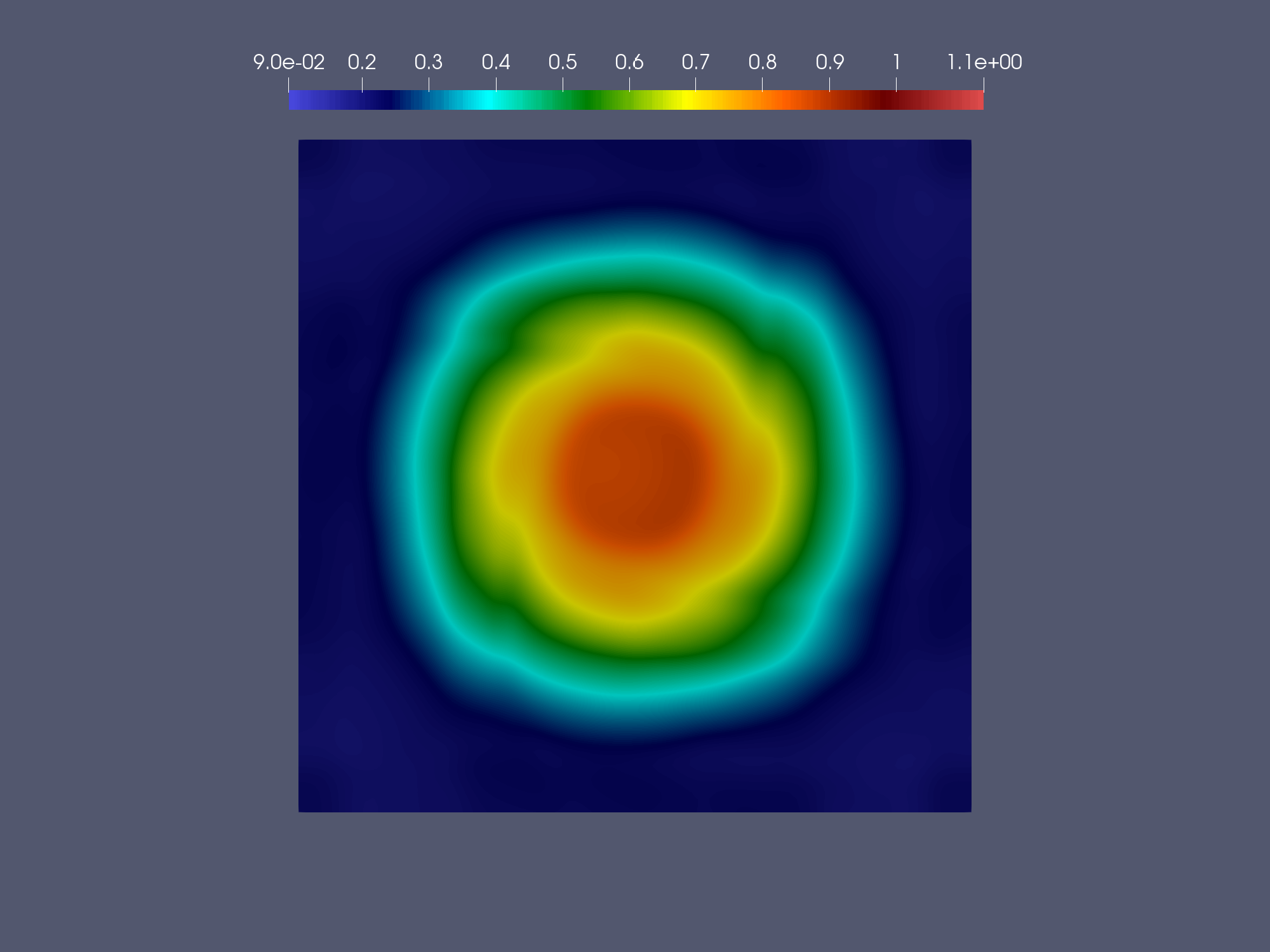} & \includegraphics[valign=c,trim={471bp 221bp 471bp 221bp},clip,width=68.25bp]{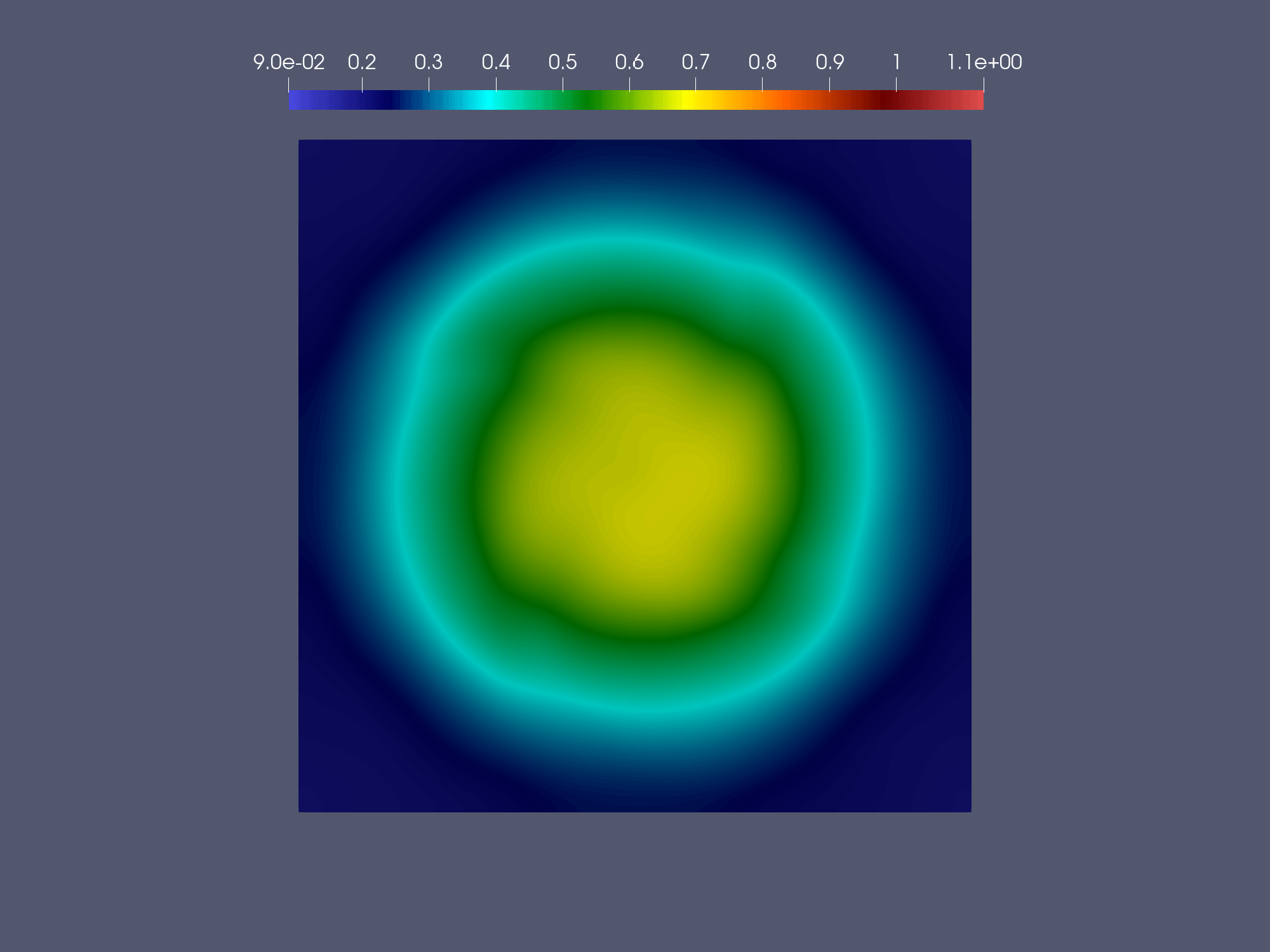} \\ \noalign{\vskip 6pt}
    \raisebox{-35.1bp}[0pt][0pt]{$\rho_3$} & \small I & \includegraphics[valign=c,trim={471bp 221bp 471bp 221bp},clip,width=68.25bp]{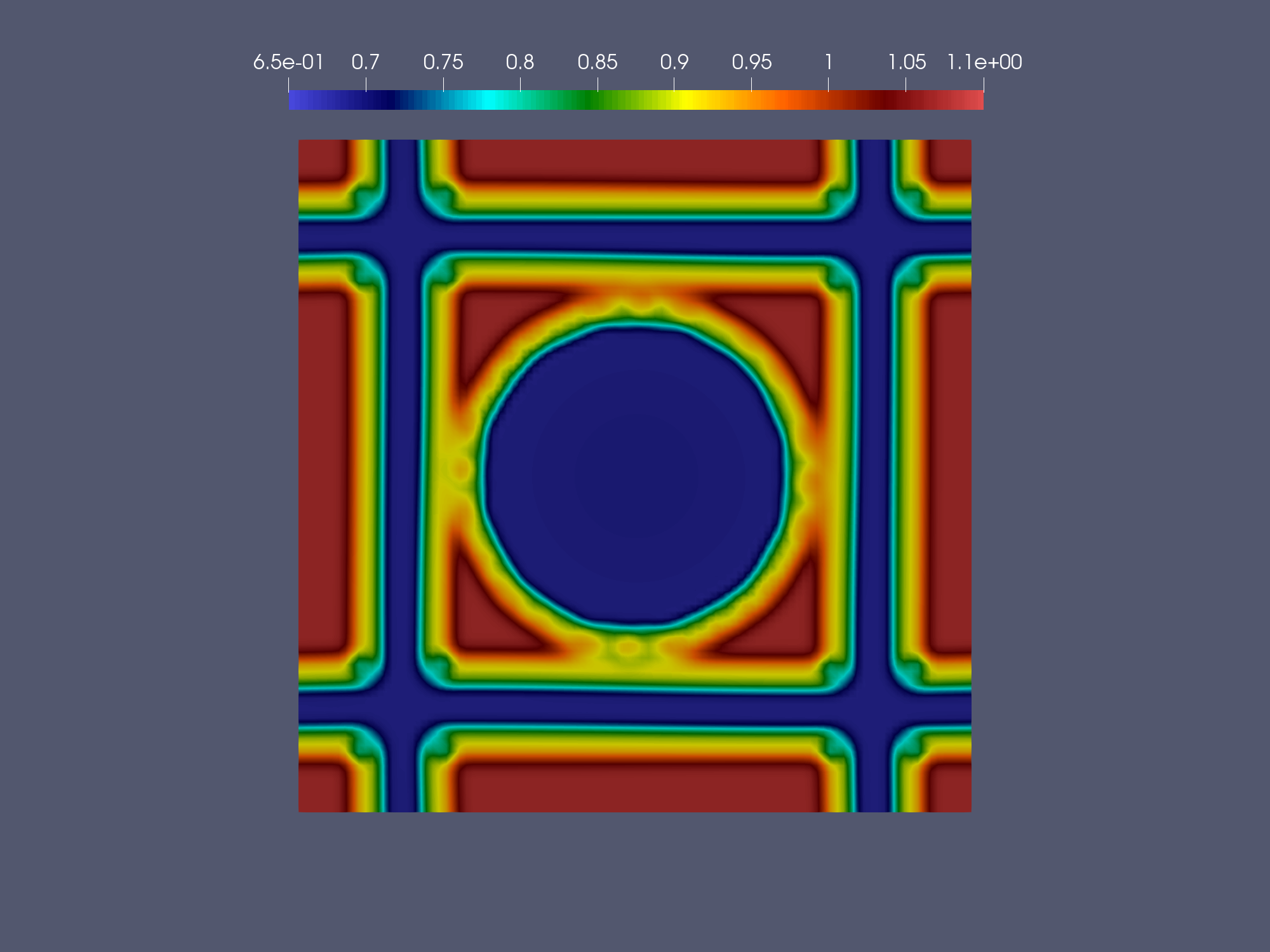} & \includegraphics[valign=c,trim={471bp 221bp 471bp 221bp},clip,width=68.25bp]{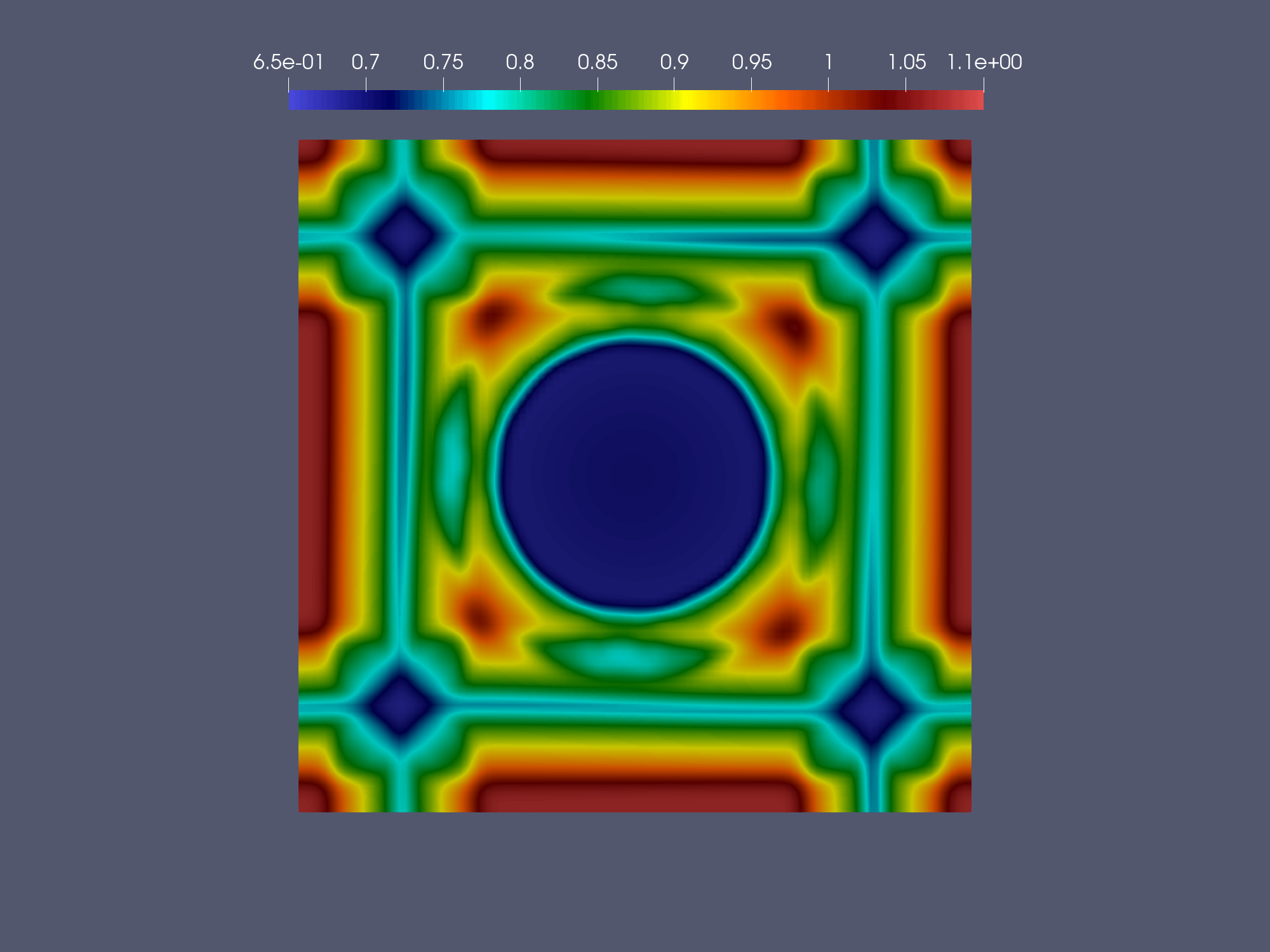} & \includegraphics[valign=c,trim={471bp 221bp 471bp 221bp},clip,width=68.25bp]{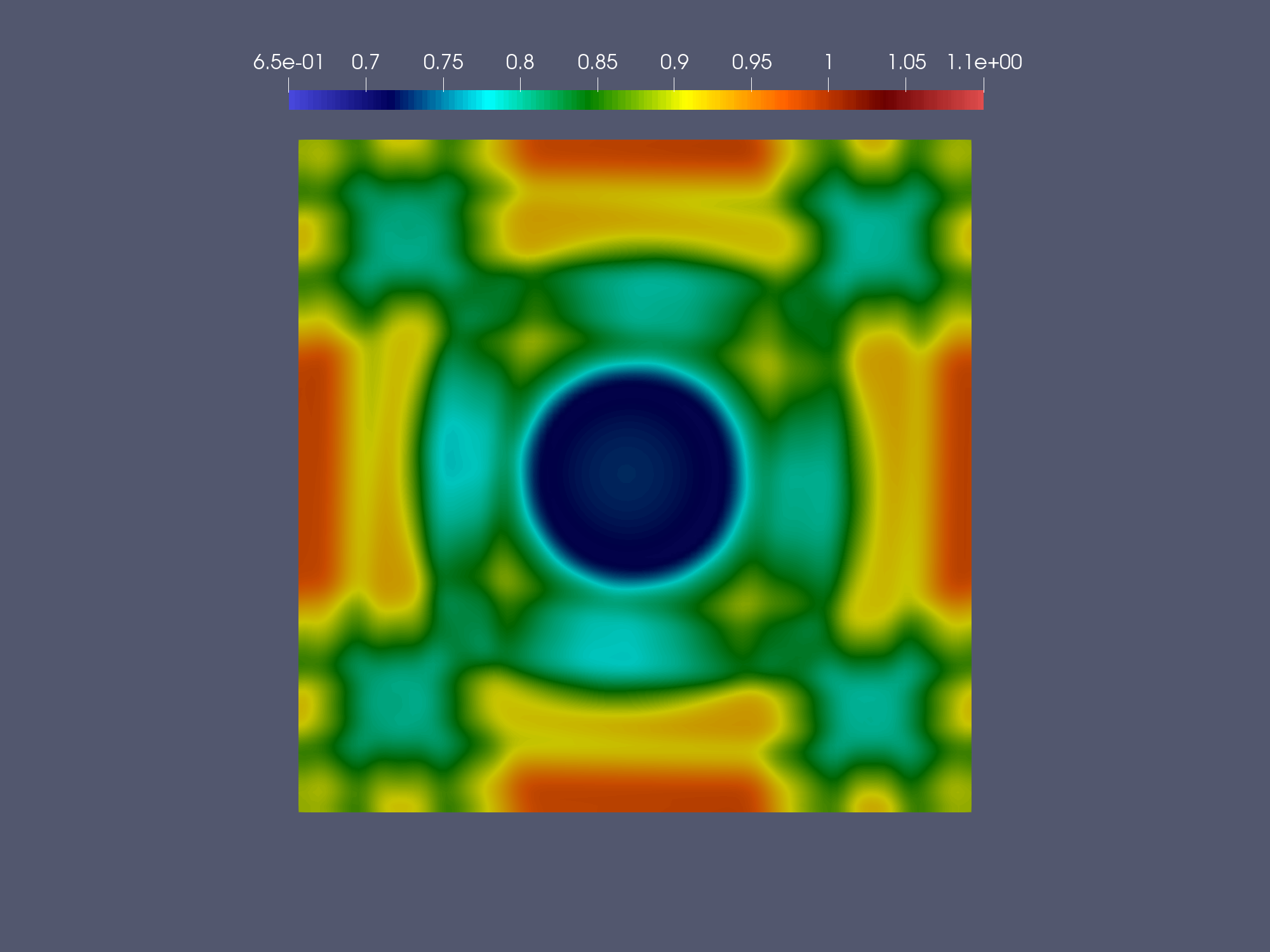} & \includegraphics[valign=c,trim={471bp 221bp 471bp 221bp},clip,width=68.25bp]{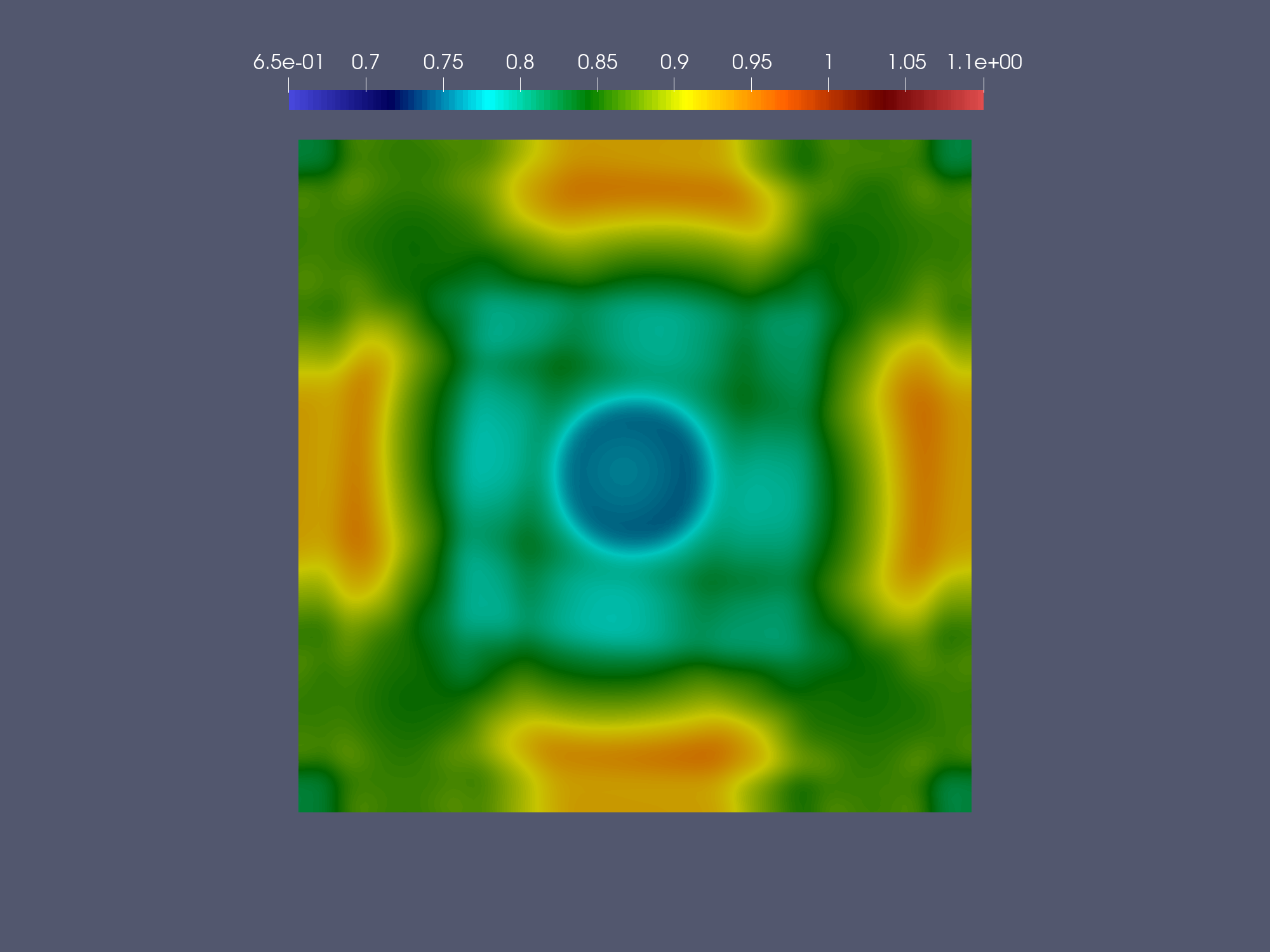} & \includegraphics[valign=c,trim={471bp 221bp 471bp 221bp},clip,width=68.25bp]{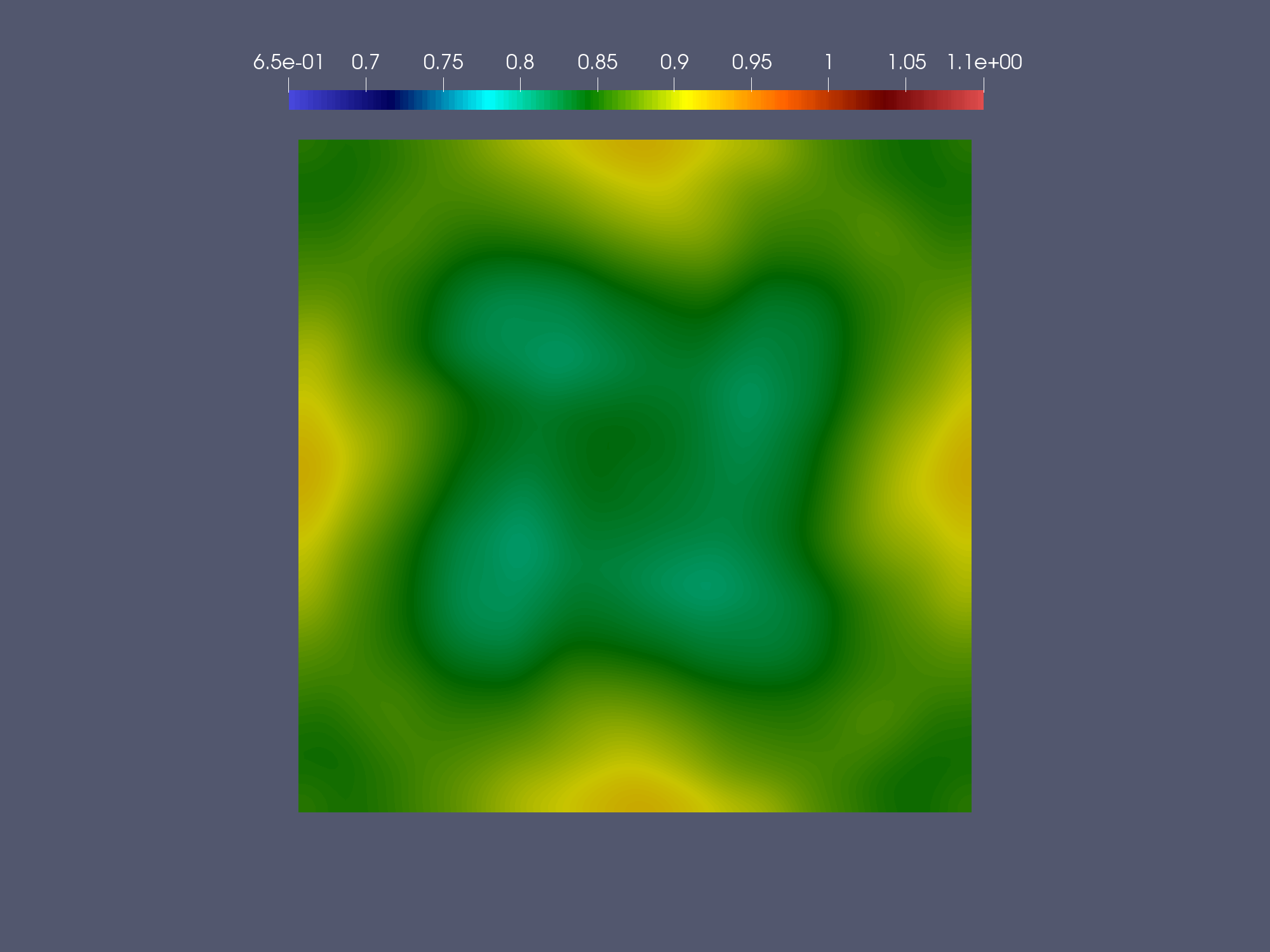} & \raisebox{-35.1bp}[0pt][0pt]{\cbar{0.65}{1.1}{134bp}{0.7,0.8,0.9,1,1.1}} \\ \noalign{\vskip 2pt}
     & \small II & \includegraphics[valign=c,trim={471bp 221bp 471bp 221bp},clip,width=68.25bp]{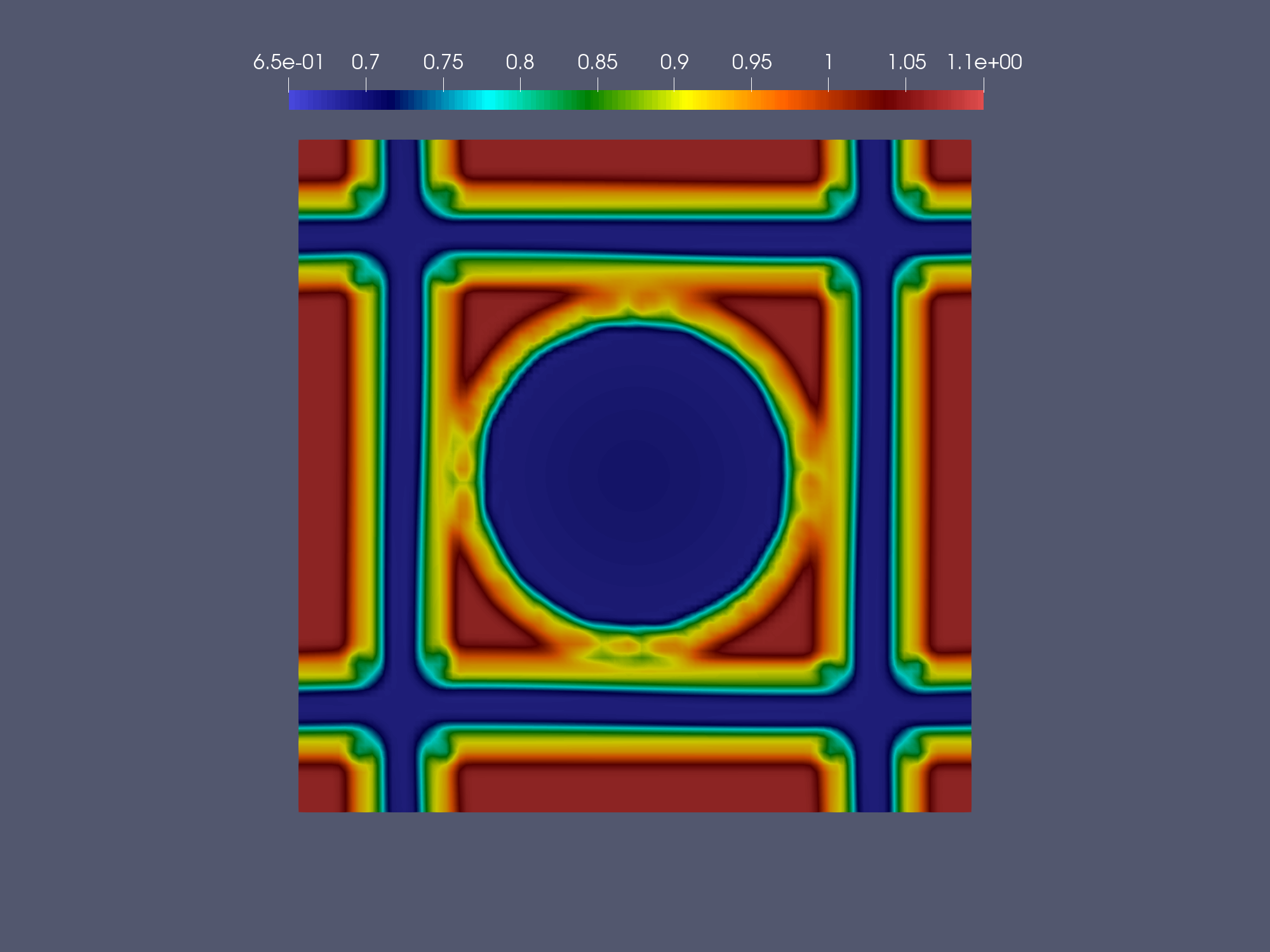} & \includegraphics[valign=c,trim={471bp 221bp 471bp 221bp},clip,width=68.25bp]{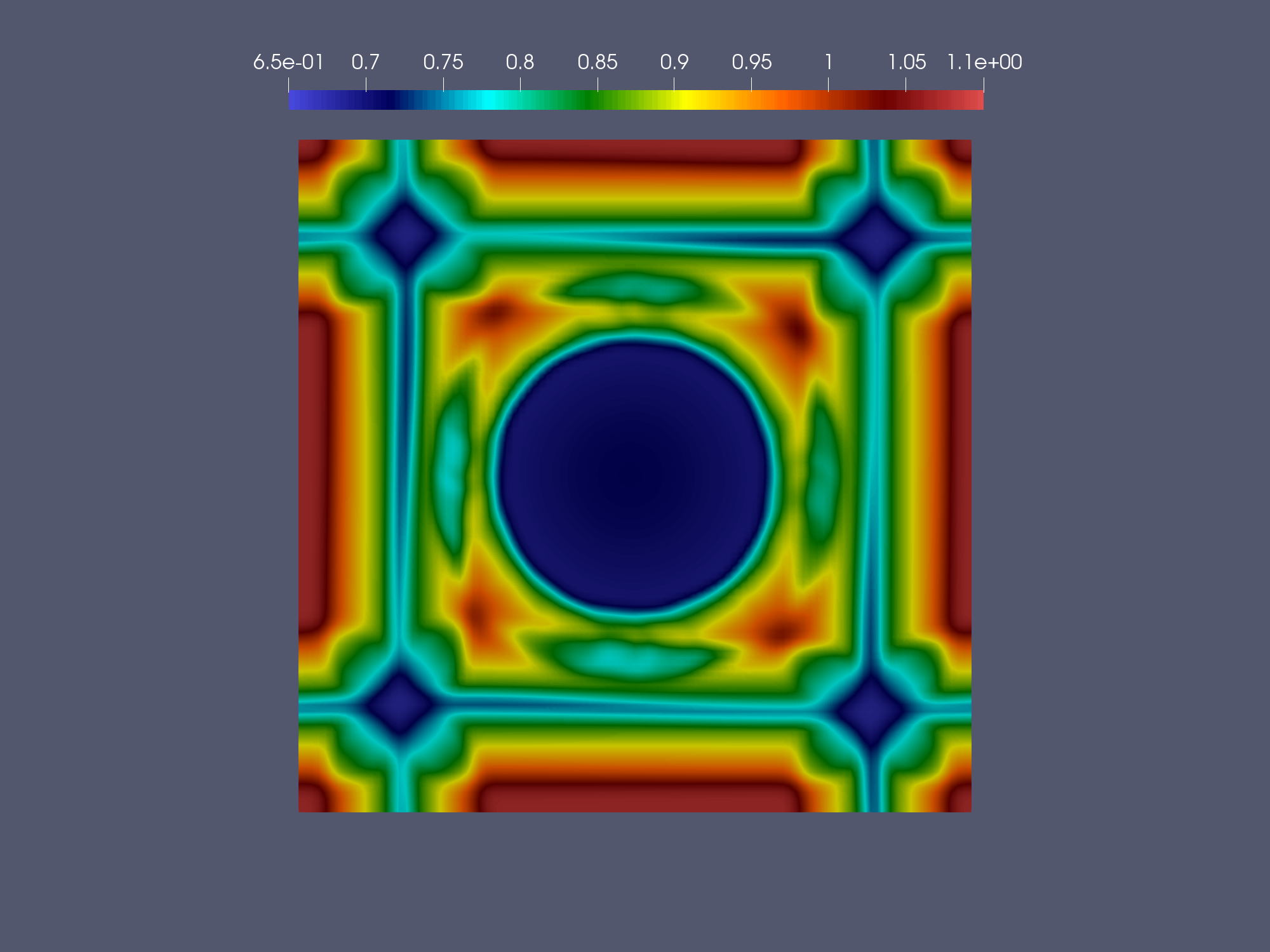} & \includegraphics[valign=c,trim={471bp 221bp 471bp 221bp},clip,width=68.25bp]{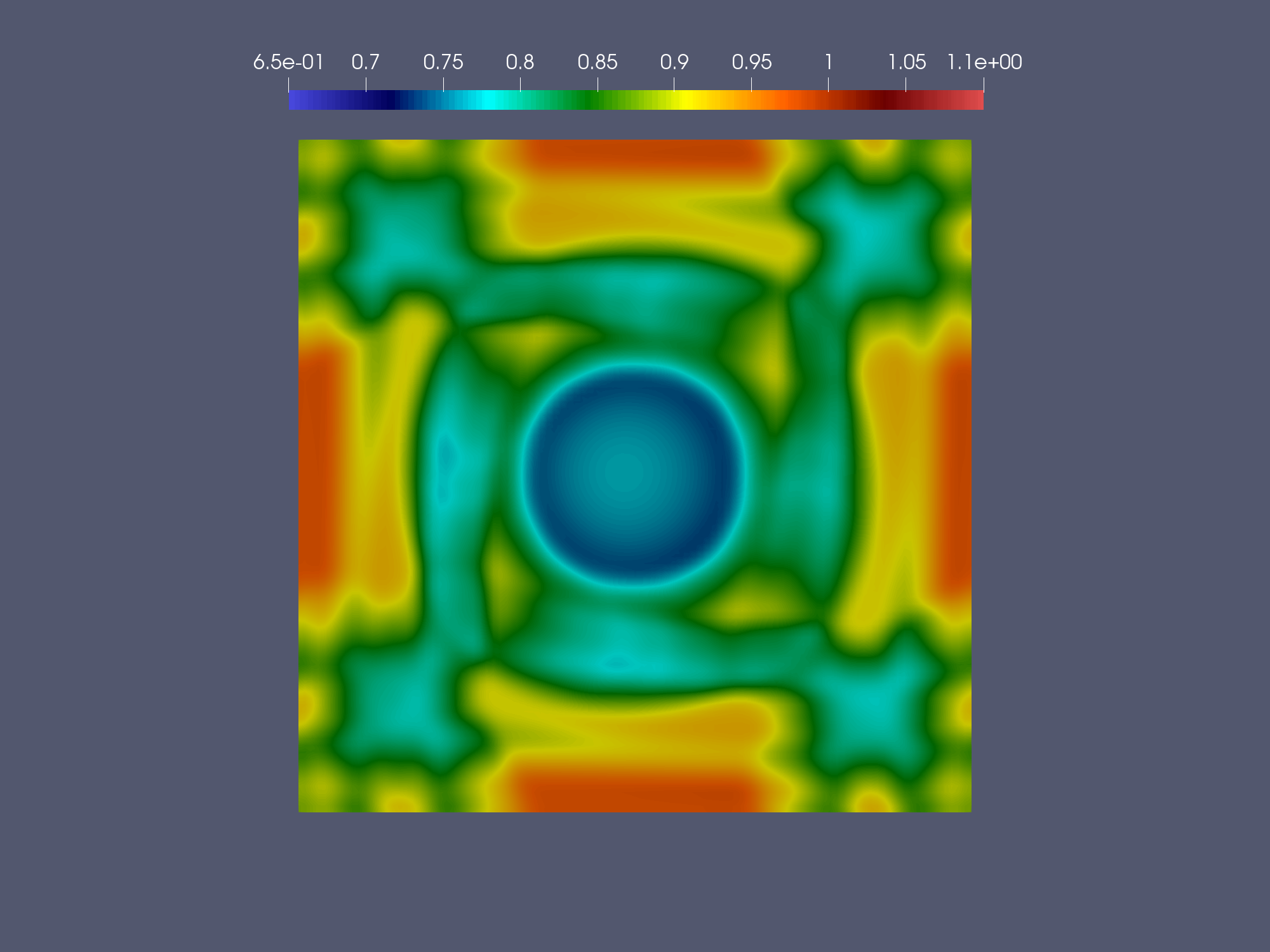} & \includegraphics[valign=c,trim={471bp 221bp 471bp 221bp},clip,width=68.25bp]{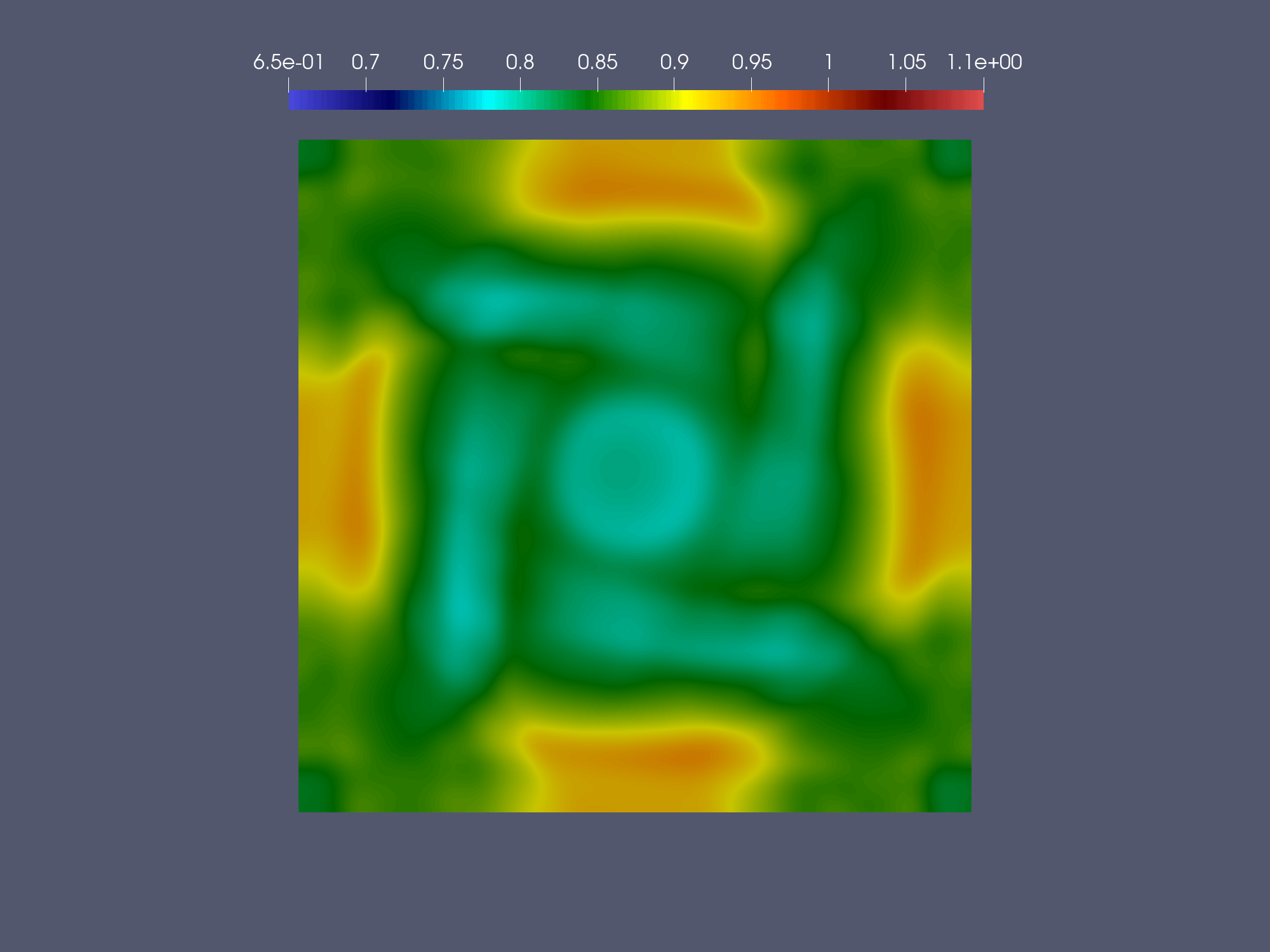} & \includegraphics[valign=c,trim={471bp 221bp 471bp 221bp},clip,width=68.25bp]{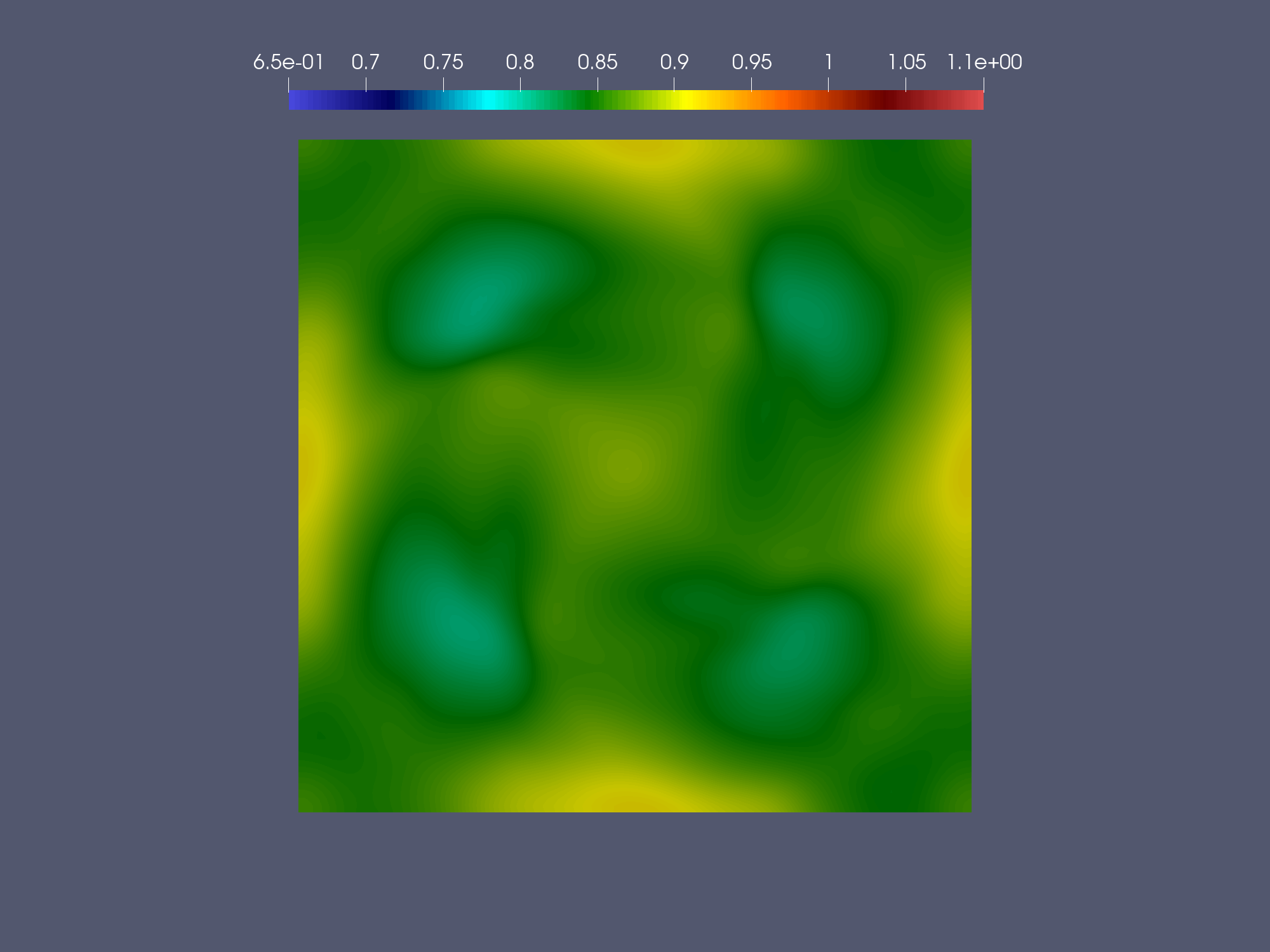} \\ \noalign{\vskip 2pt}
\end{tabular}%
}
\caption{Structure-preserving test: snapshots of the partial densities $\rho_1$, $\rho_2$, $\rho_3$ for Scheme~I (semi-implicit Euler) and Scheme~II (CPG).}
\label{fig:rho}
\end{figure}

\begin{figure}[!htbp]
\centering
\scalebox{0.86}{%
\begin{tabular}{c@{\hspace{0.5em}}c@{\hspace{0.5em}}c@{\hspace{2pt}}c@{\hspace{2pt}}c@{\hspace{2pt}}c@{\hspace{2pt}}c@{\hspace{3pt}}l}
    & & \small $t = 0.01$ & \small $t = 0.02$ & \small $t = 0.04$ & \small $t = 0.06$ & \small $t = 0.1$ \\[3pt]
    \raisebox{-35.1bp}[0pt][0pt]{$\vtheta$} & \small I & \includegraphics[valign=c,trim={471bp 221bp 471bp 221bp},clip,width=68.25bp]{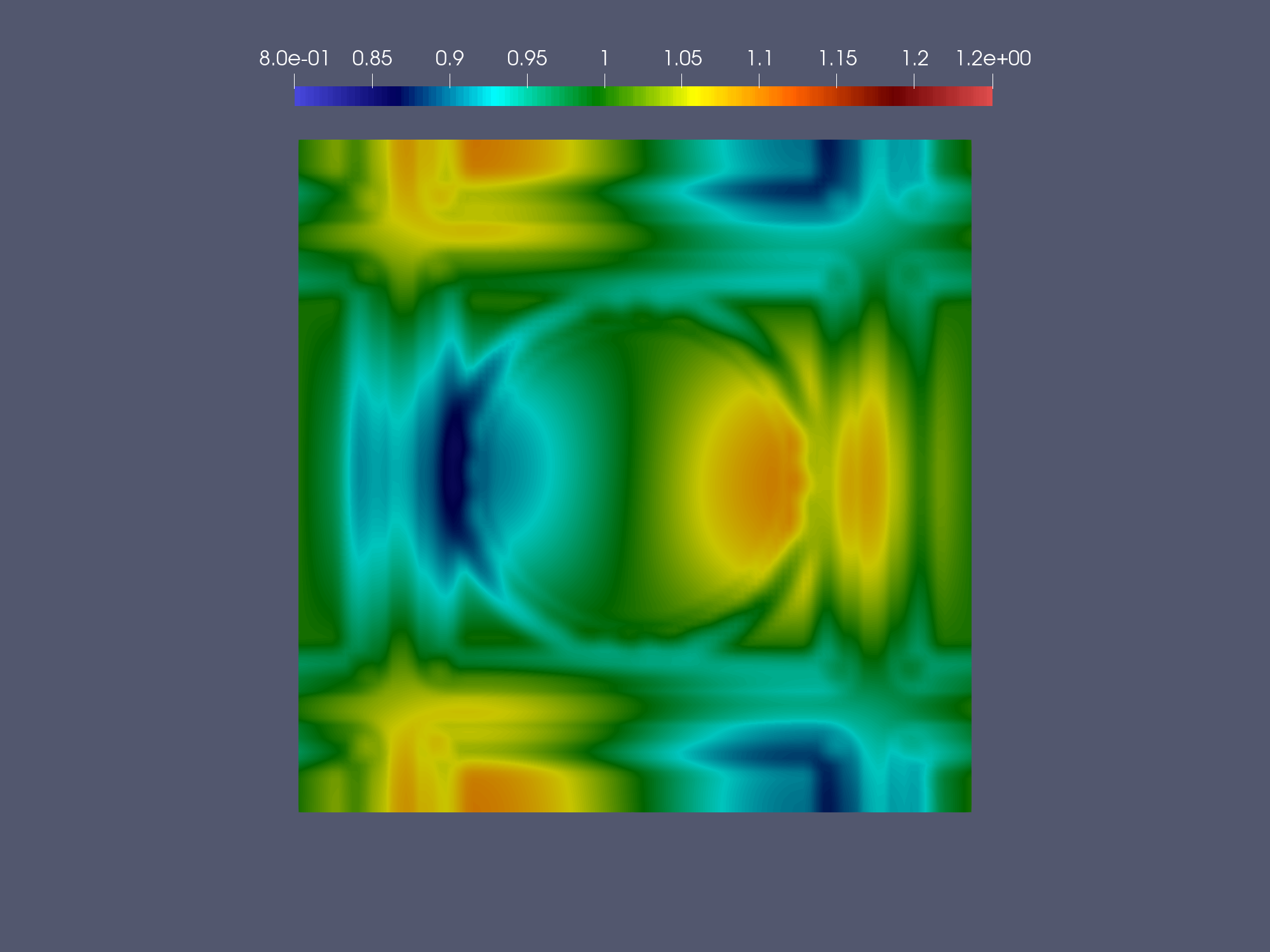} & \includegraphics[valign=c,trim={471bp 221bp 471bp 221bp},clip,width=68.25bp]{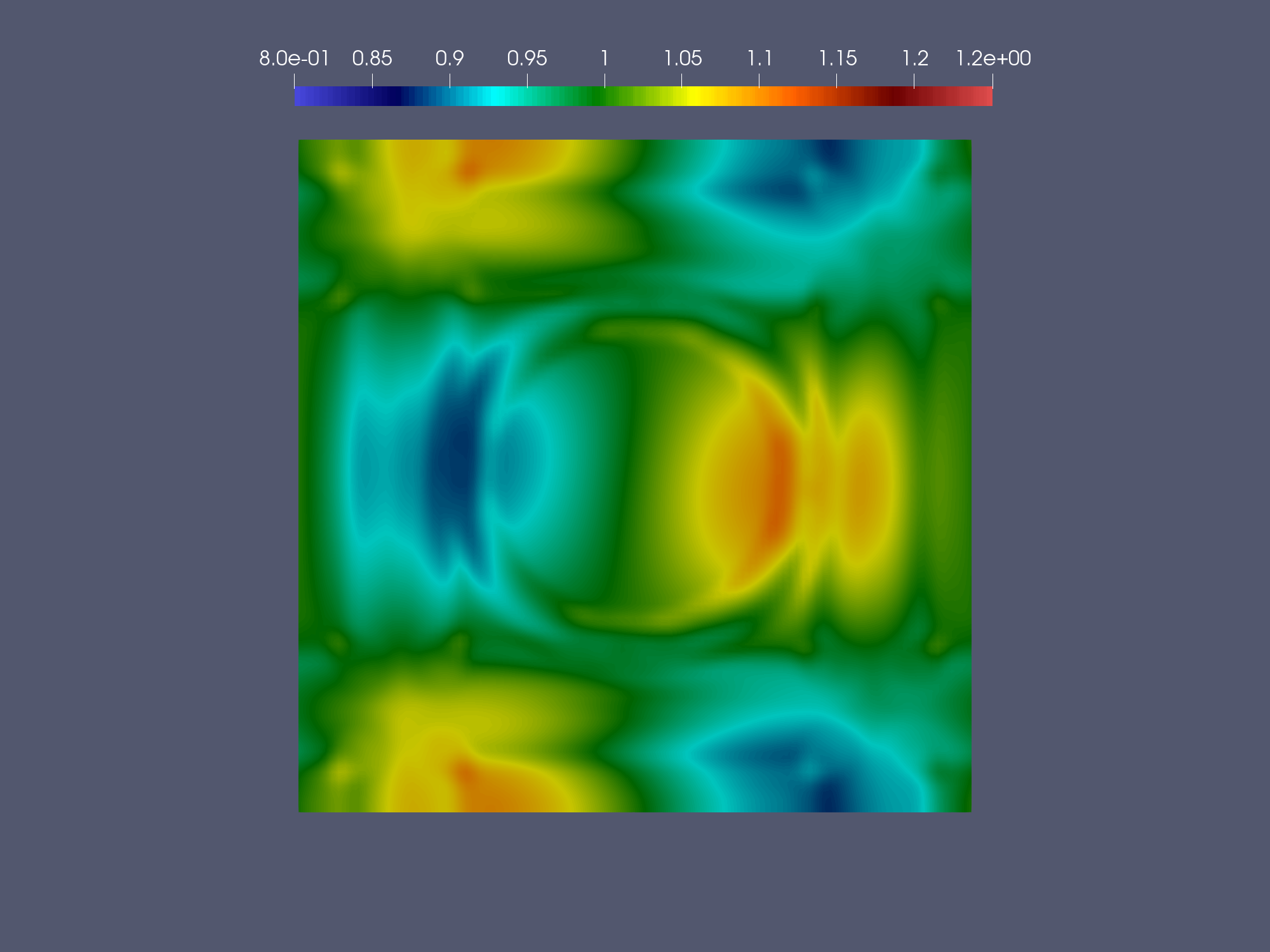} & \includegraphics[valign=c,trim={471bp 221bp 471bp 221bp},clip,width=68.25bp]{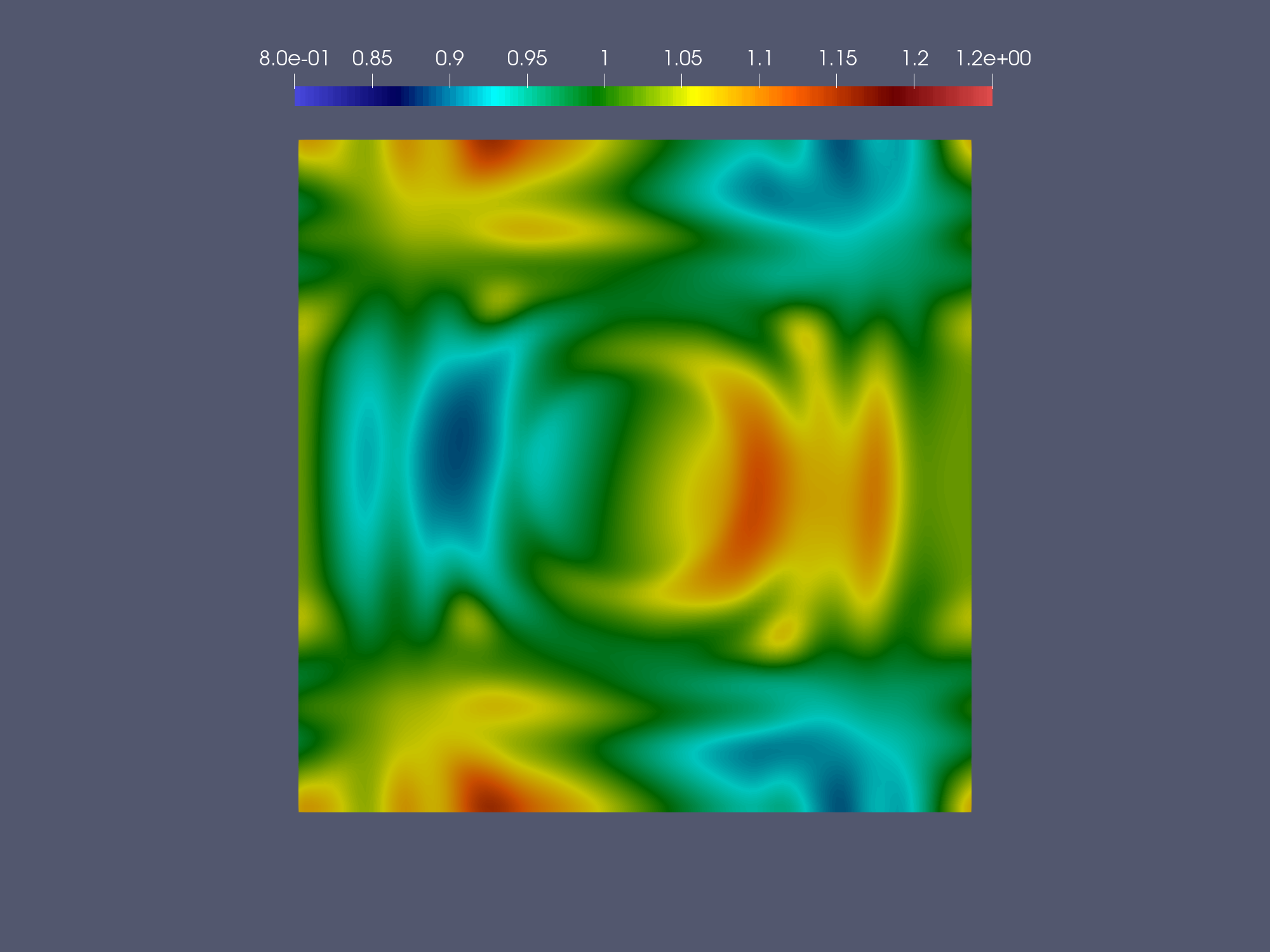} & \includegraphics[valign=c,trim={471bp 221bp 471bp 221bp},clip,width=68.25bp]{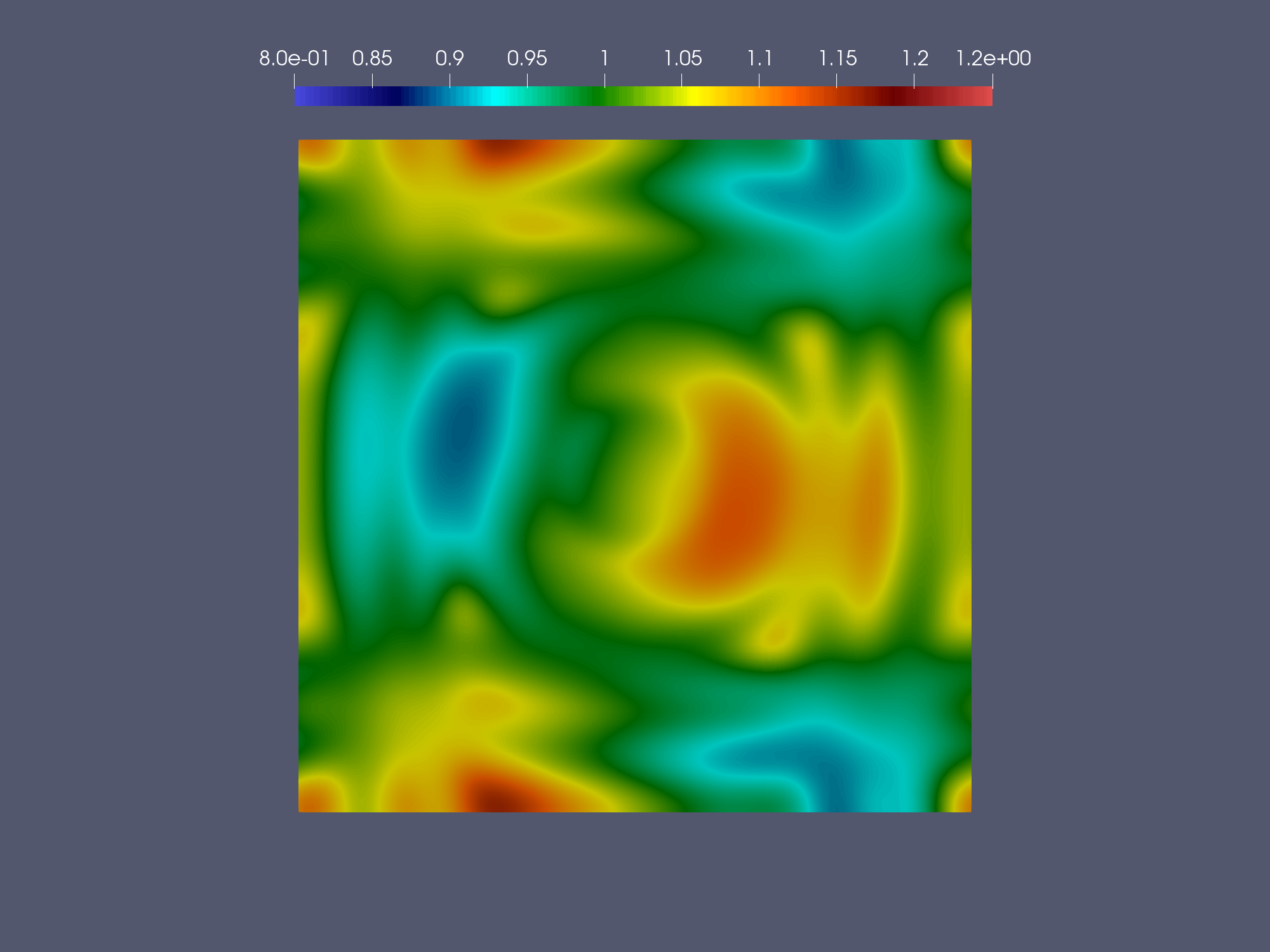} & \includegraphics[valign=c,trim={471bp 221bp 471bp 221bp},clip,width=68.25bp]{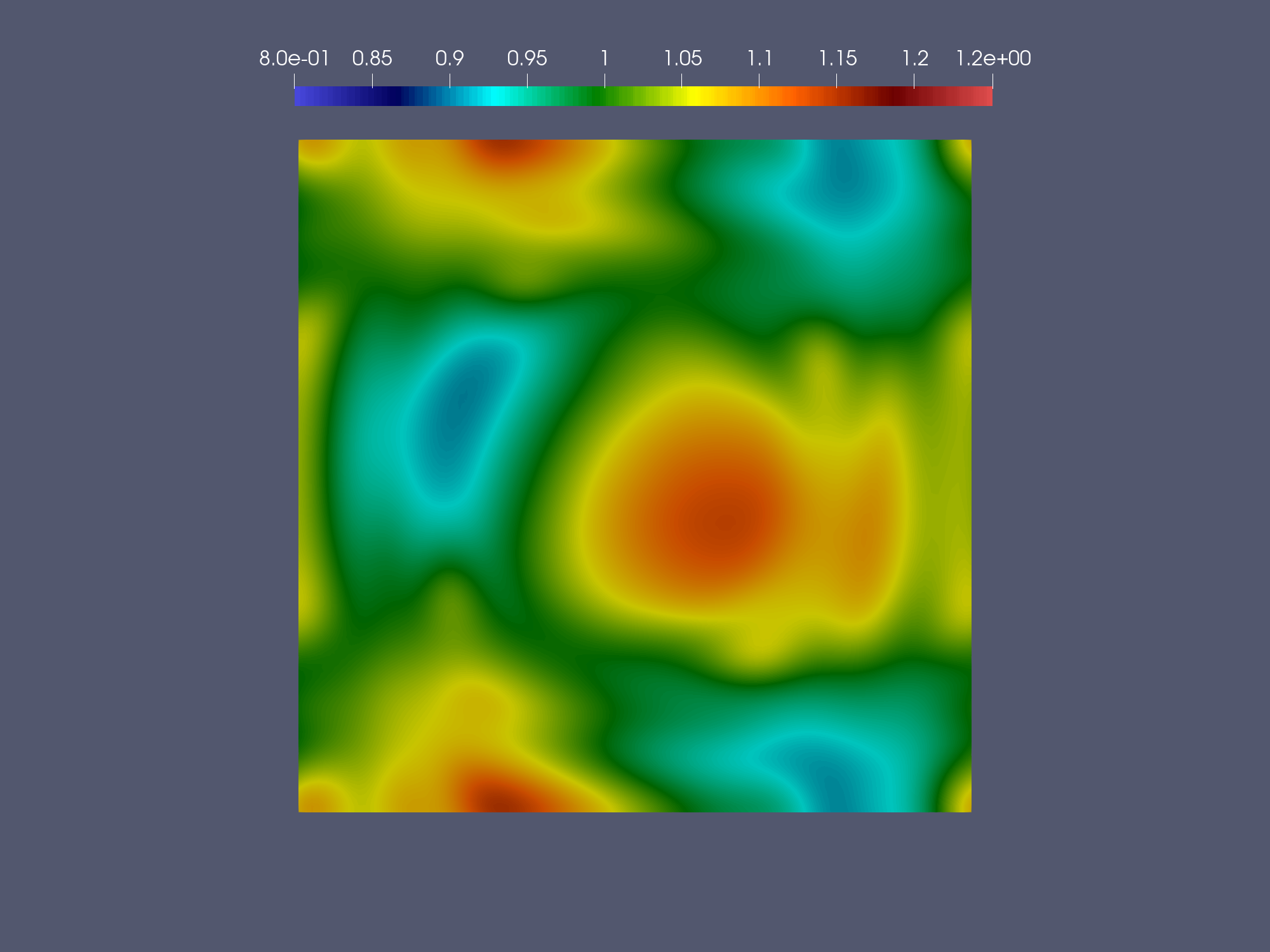} & \raisebox{-35.1bp}[0pt][0pt]{\cbar{0.8}{1.2}{134bp}{0.8,0.9,1,1.1,1.2}} \\ \noalign{\vskip 2pt}
     & \small II & \includegraphics[valign=c,trim={471bp 221bp 471bp 221bp},clip,width=68.25bp]{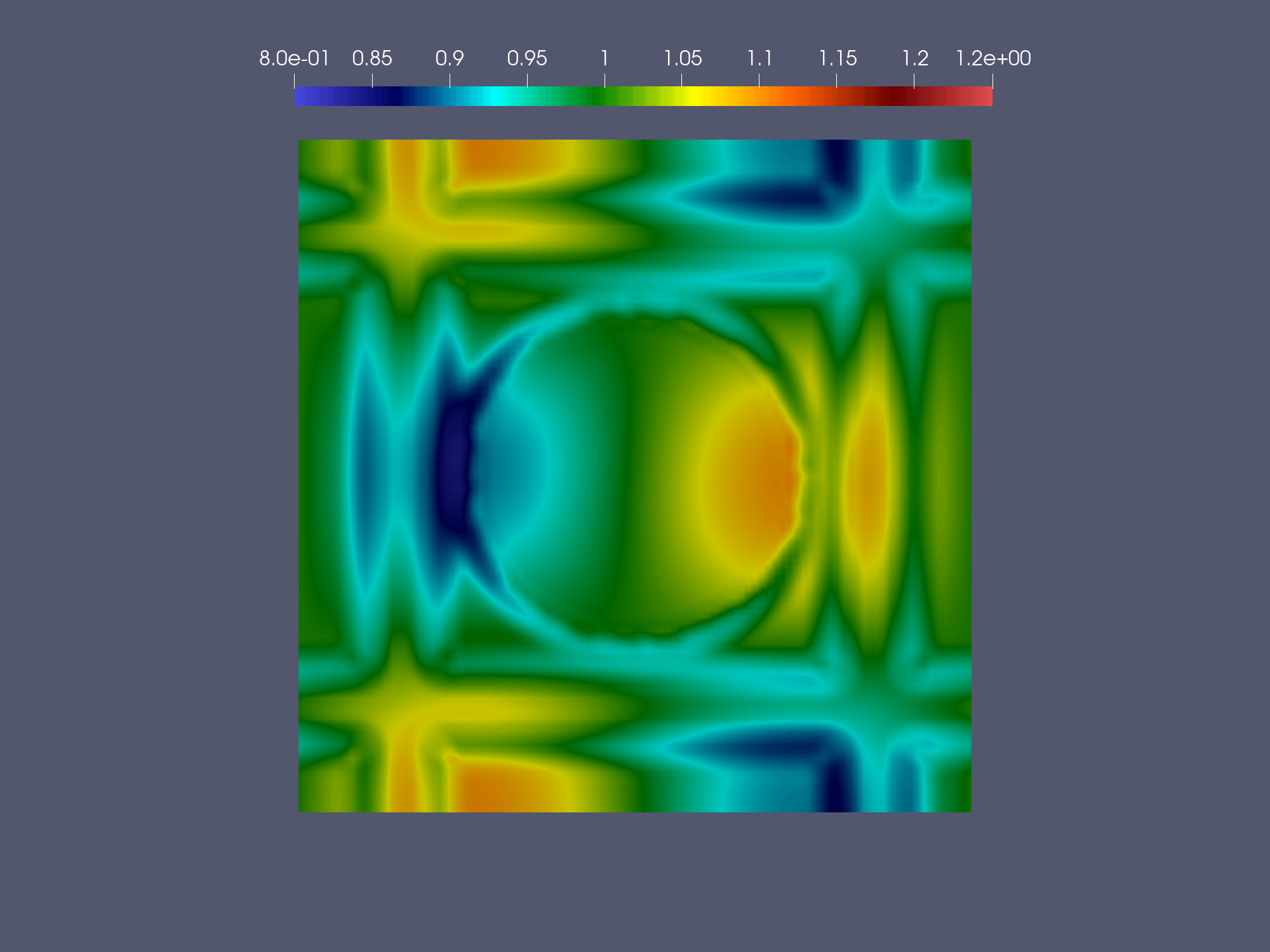} & \includegraphics[valign=c,trim={471bp 221bp 471bp 221bp},clip,width=68.25bp]{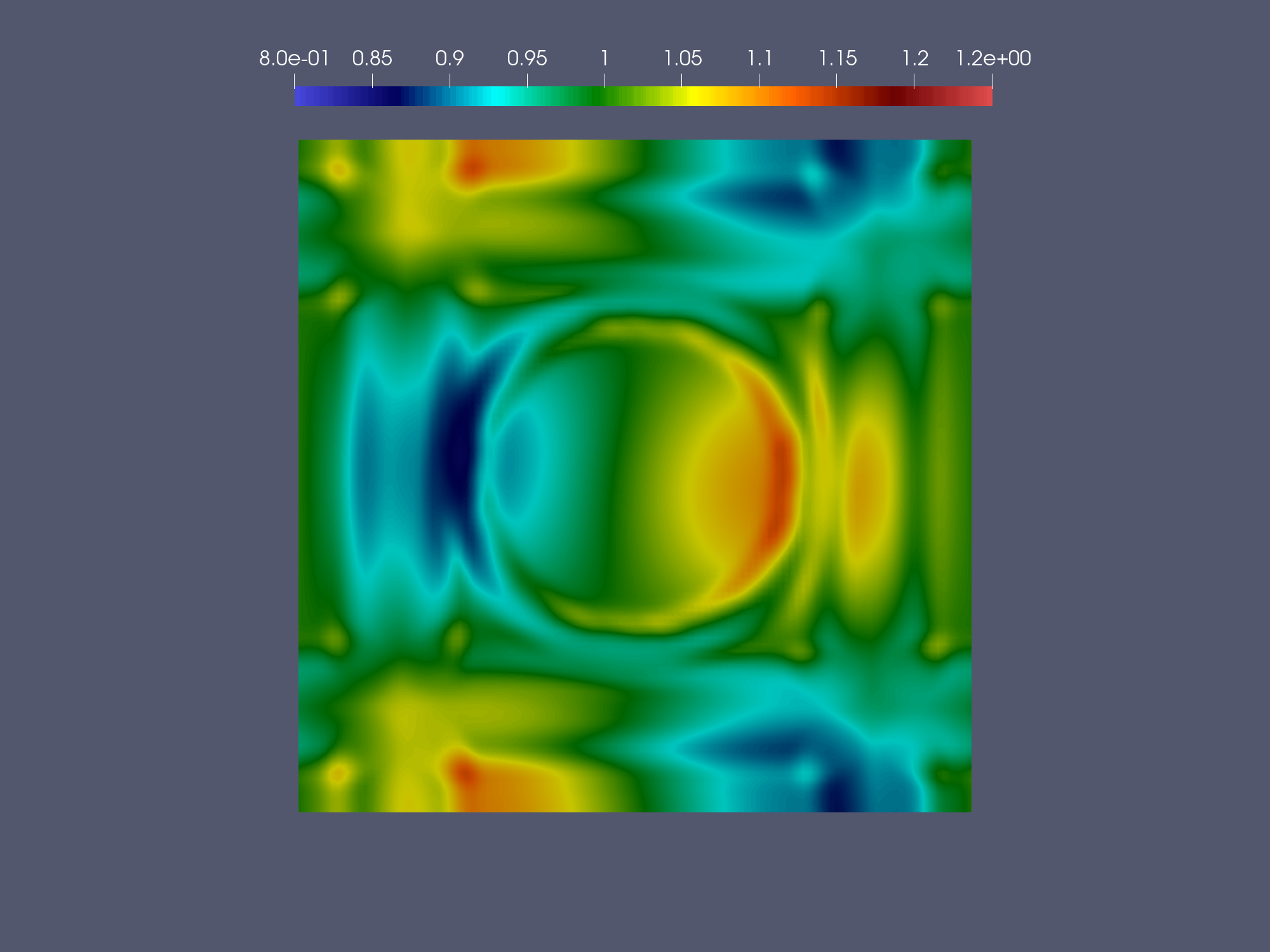} & \includegraphics[valign=c,trim={471bp 221bp 471bp 221bp},clip,width=68.25bp]{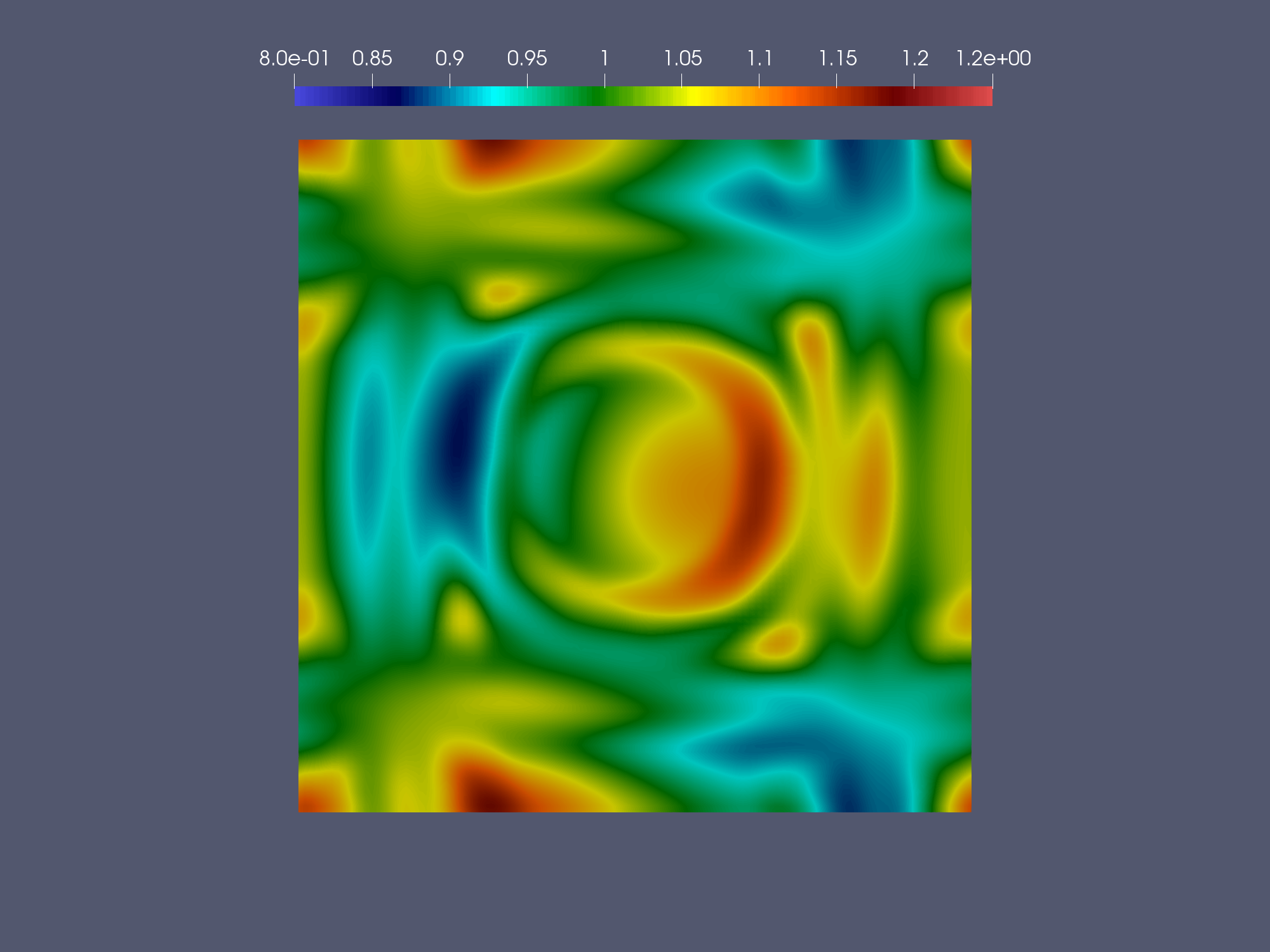} & \includegraphics[valign=c,trim={471bp 221bp 471bp 221bp},clip,width=68.25bp]{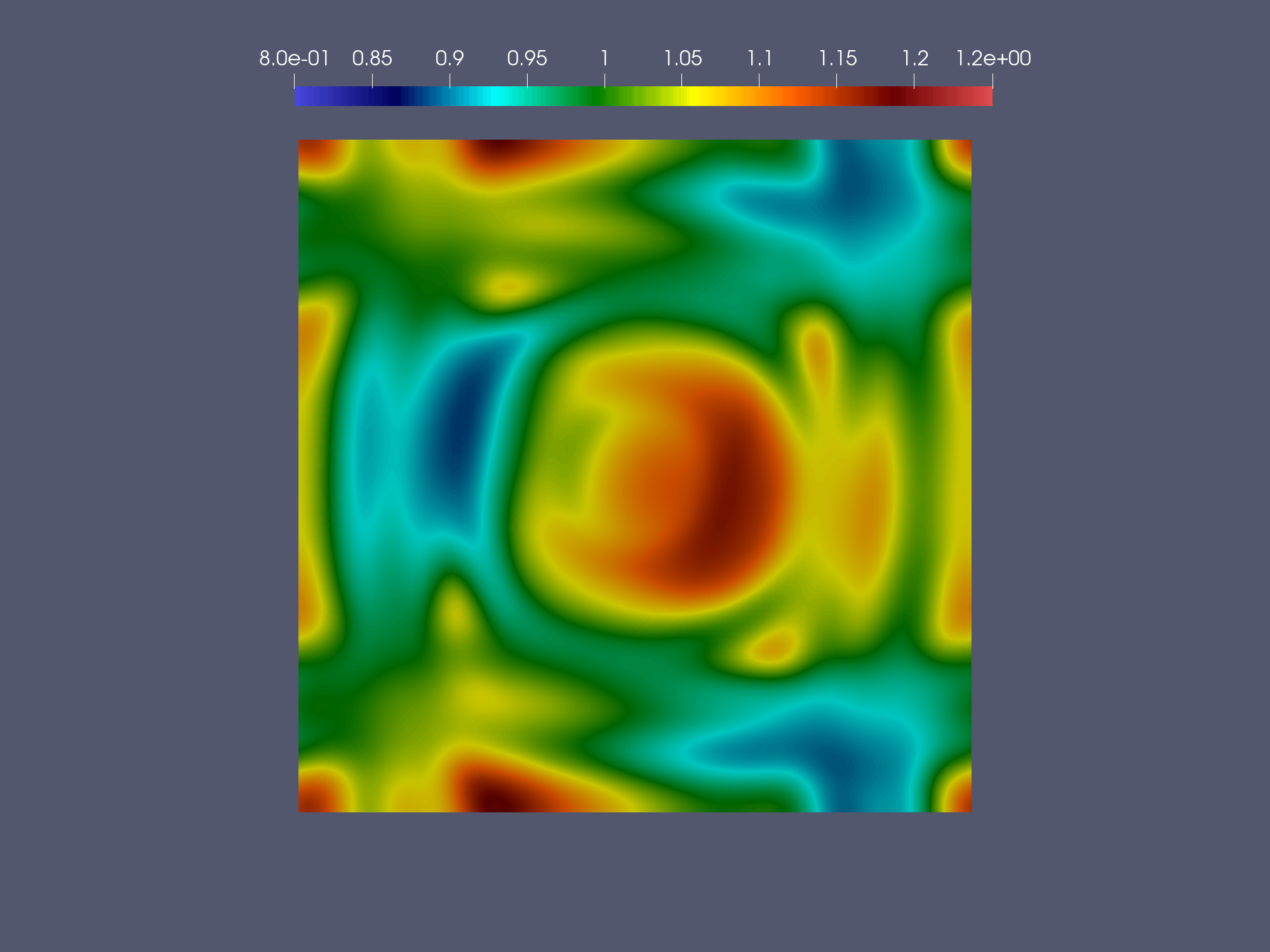} & \includegraphics[valign=c,trim={471bp 221bp 471bp 221bp},clip,width=68.25bp]{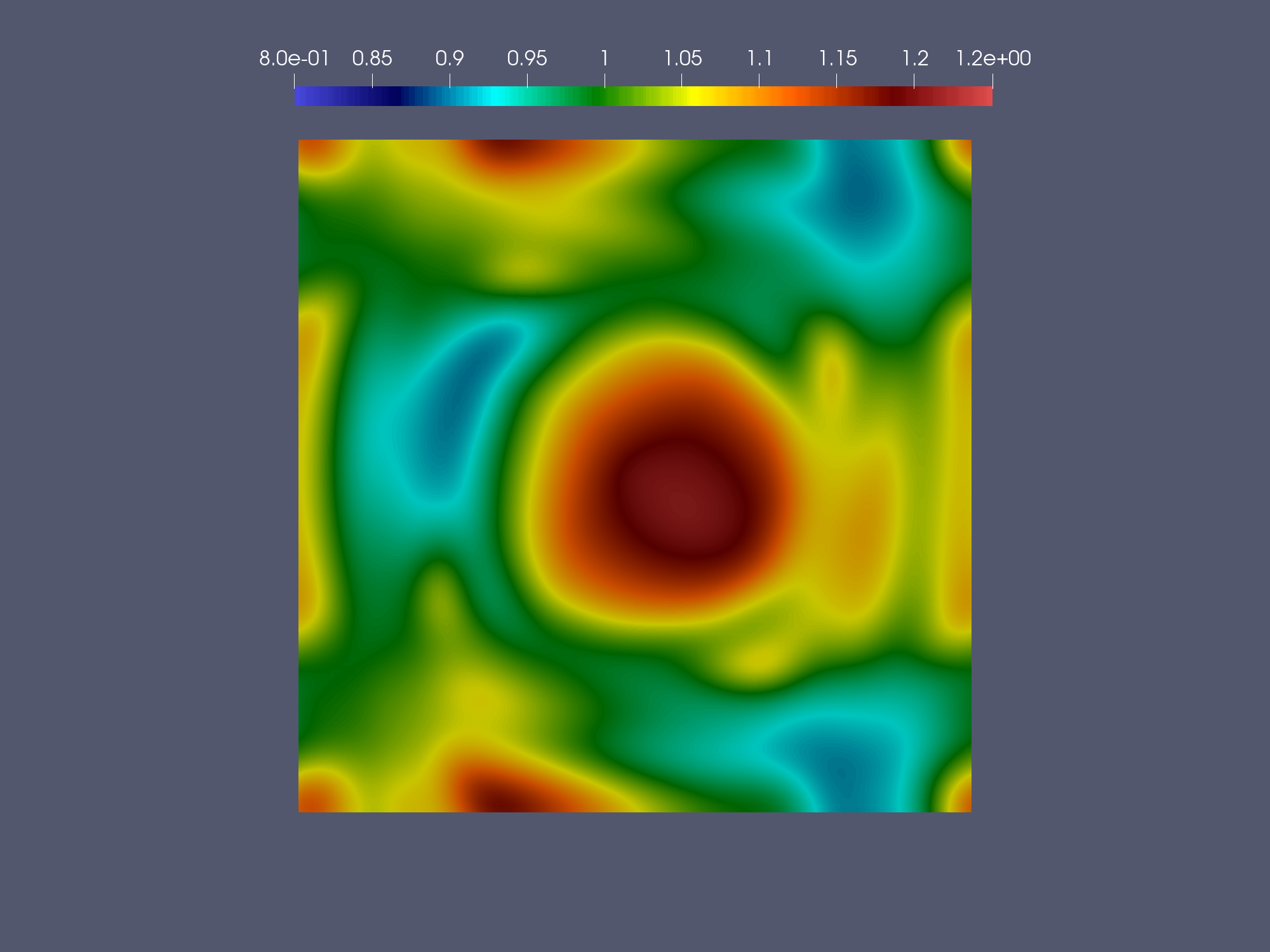} \\ \noalign{\vskip 6pt}
    \raisebox{-35.1bp}[0pt][0pt]{$\u$} & \small I & \includegraphics[valign=c,trim={471bp 221bp 471bp 221bp},clip,width=68.25bp]{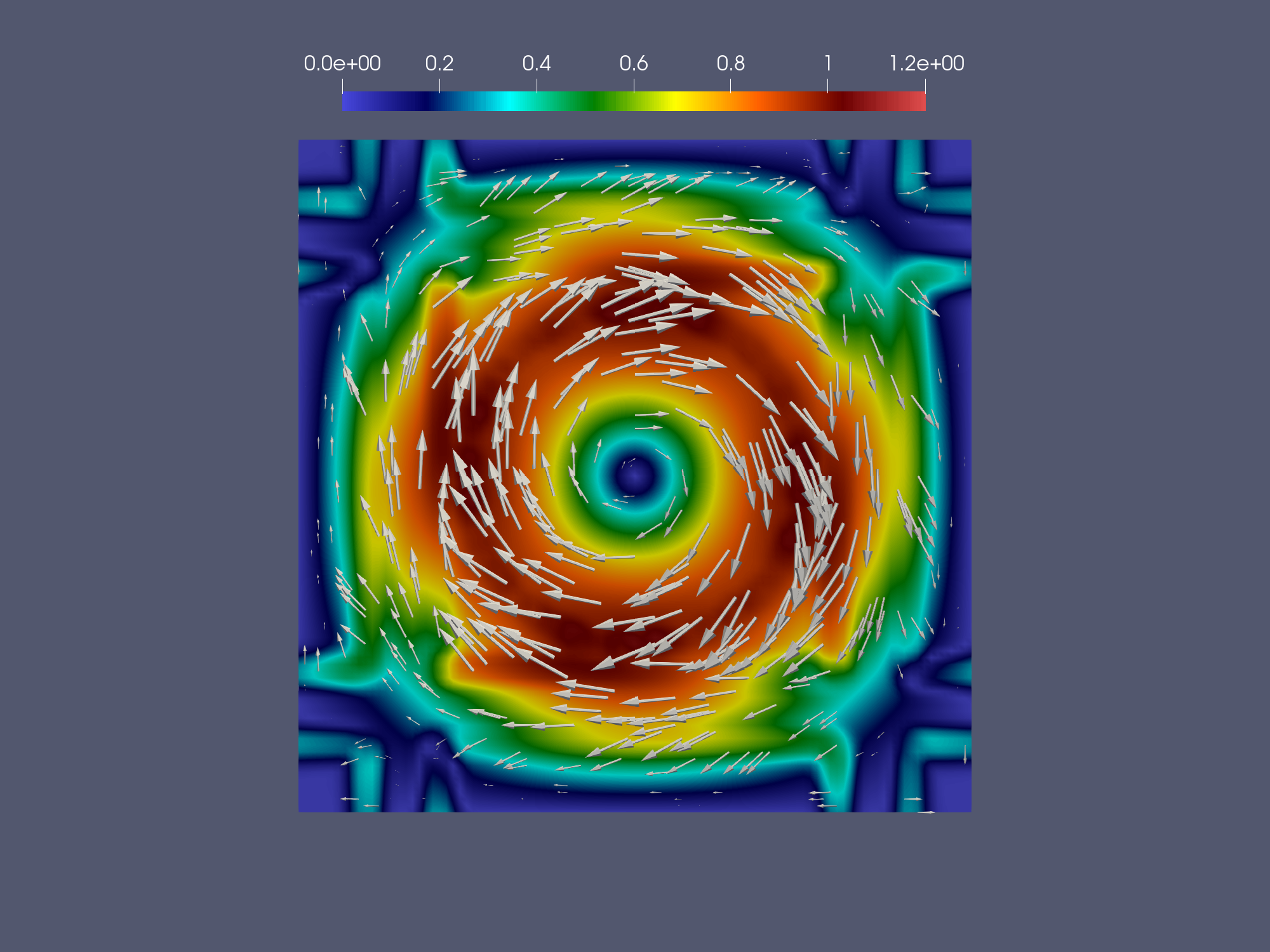} & \includegraphics[valign=c,trim={471bp 221bp 471bp 221bp},clip,width=68.25bp]{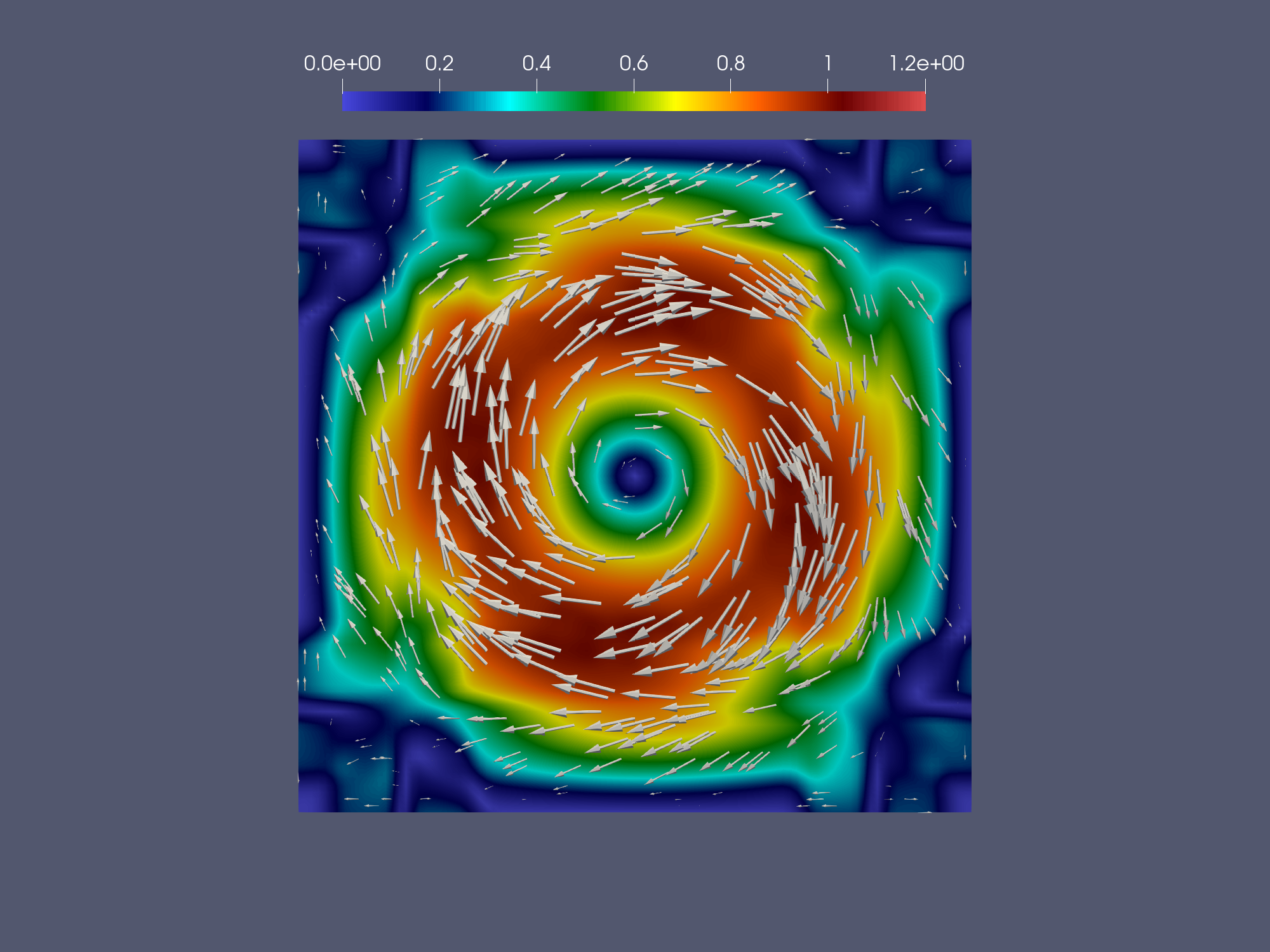} & \includegraphics[valign=c,trim={471bp 221bp 471bp 221bp},clip,width=68.25bp]{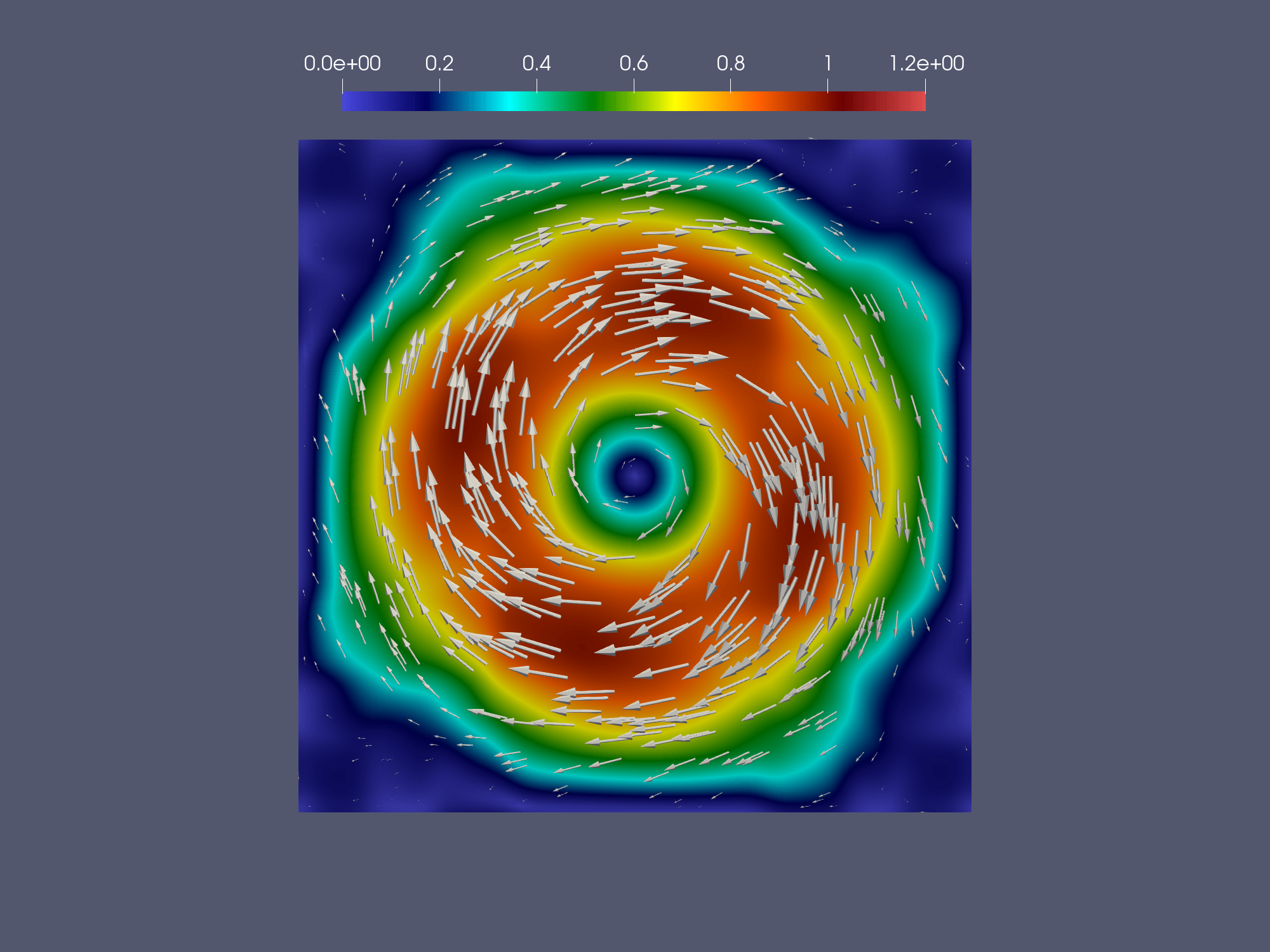} & \includegraphics[valign=c,trim={471bp 221bp 471bp 221bp},clip,width=68.25bp]{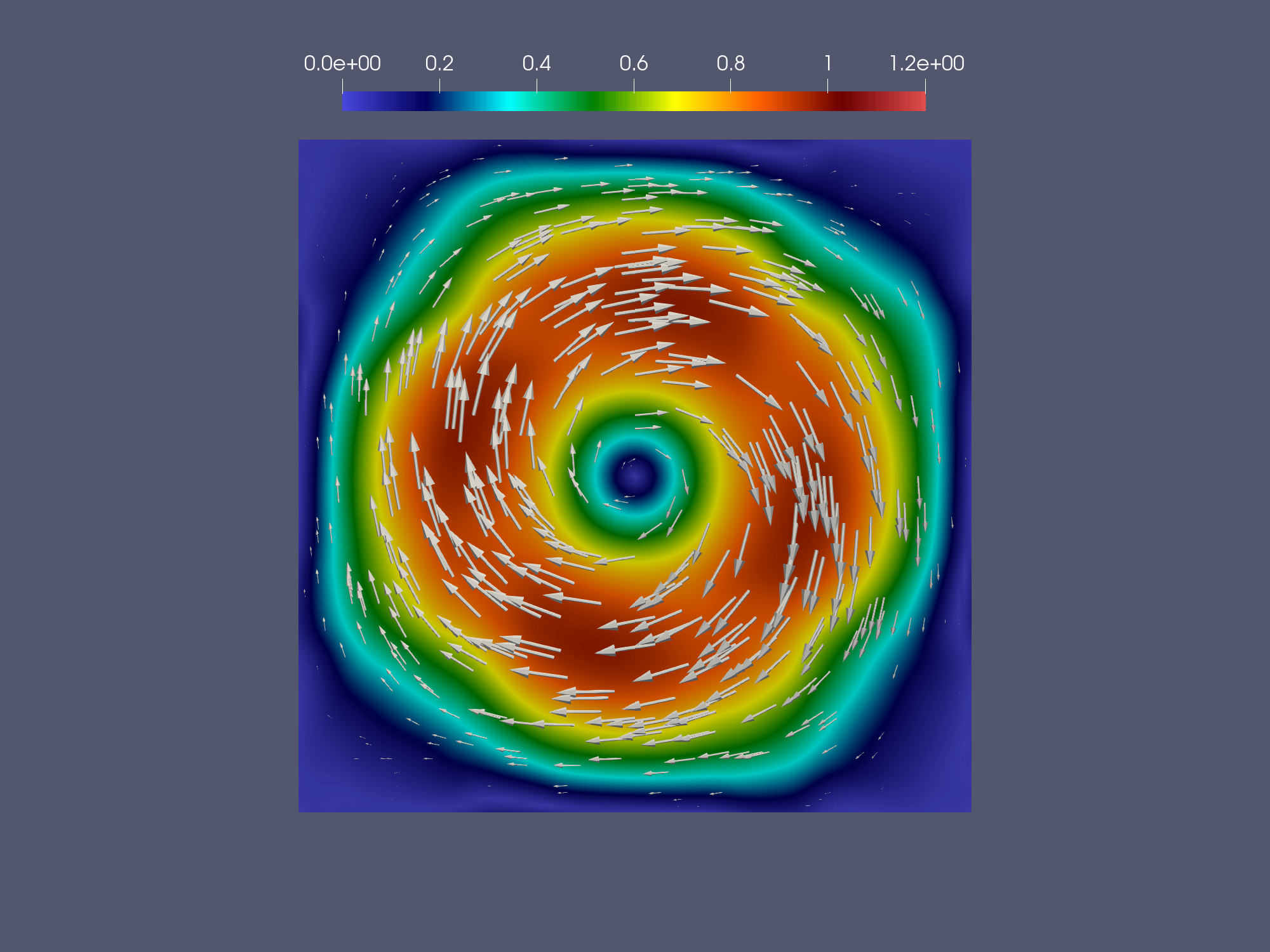} & \includegraphics[valign=c,trim={471bp 221bp 471bp 221bp},clip,width=68.25bp]{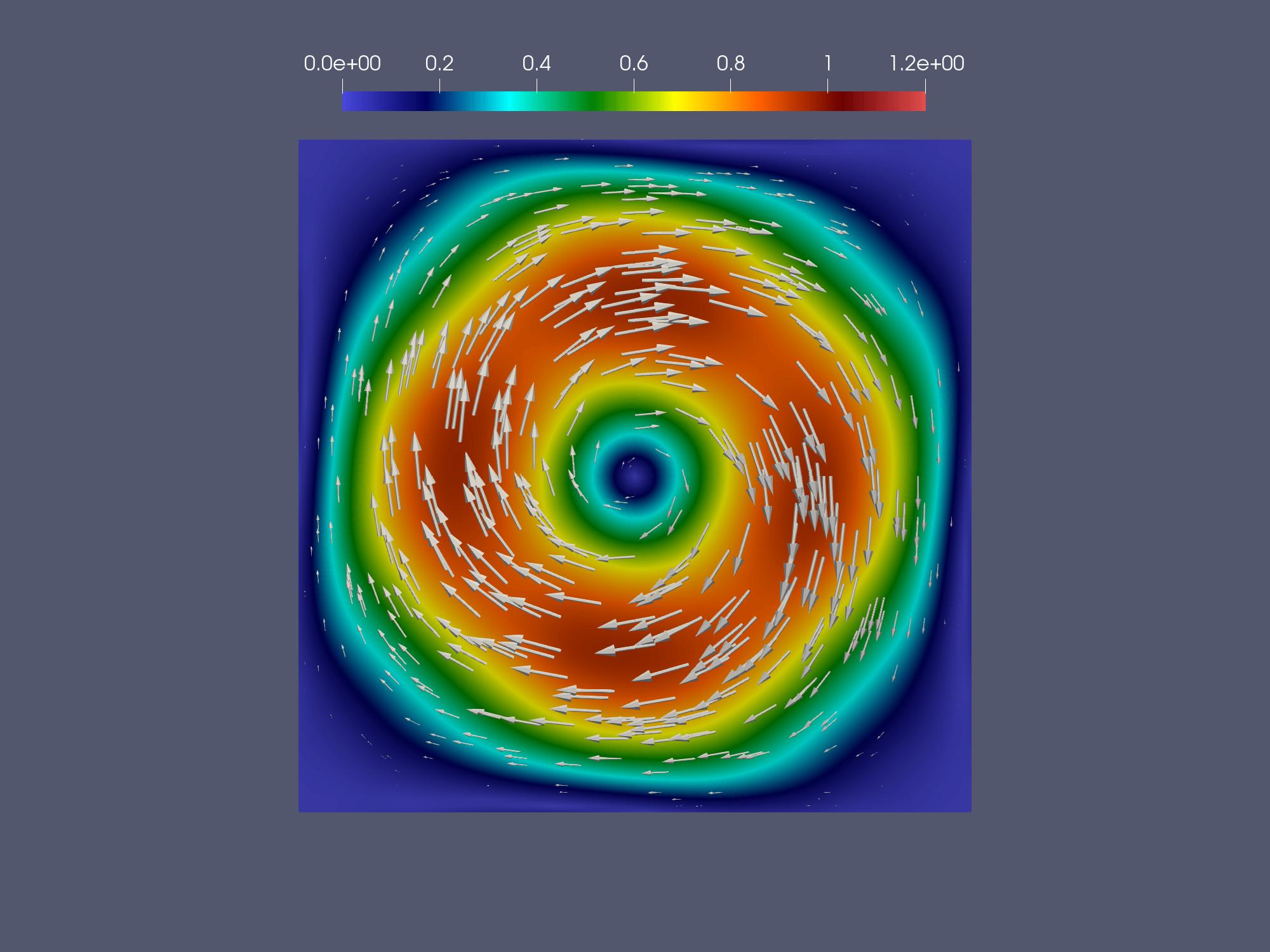} & \raisebox{-35.1bp}[0pt][0pt]{\cbar{0}{1.2}{134bp}{0,0.4,0.8,1.2}} \\ \noalign{\vskip 2pt}
     & \small II & \includegraphics[valign=c,trim={471bp 221bp 471bp 221bp},clip,width=68.25bp]{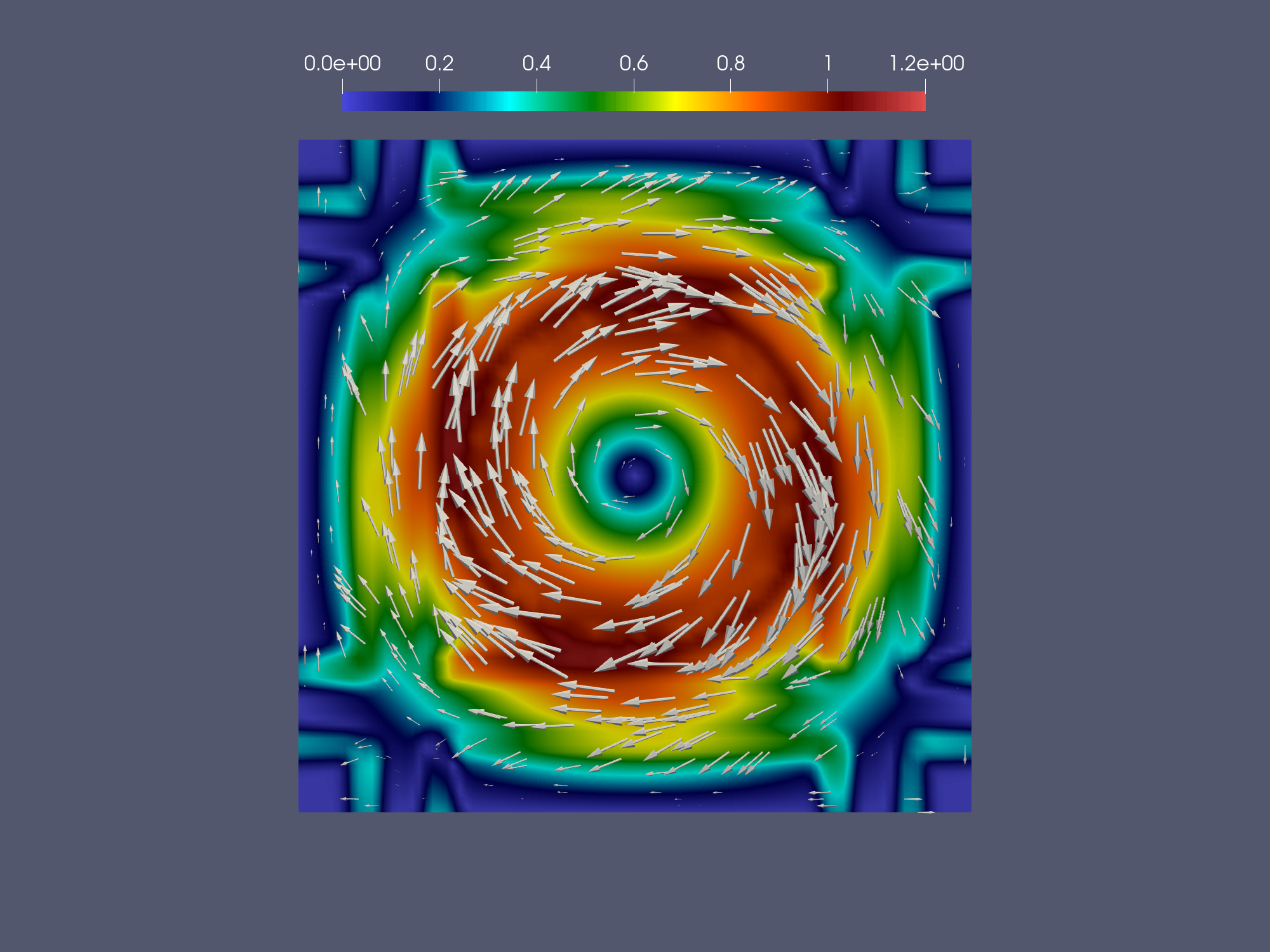} & \includegraphics[valign=c,trim={471bp 221bp 471bp 221bp},clip,width=68.25bp]{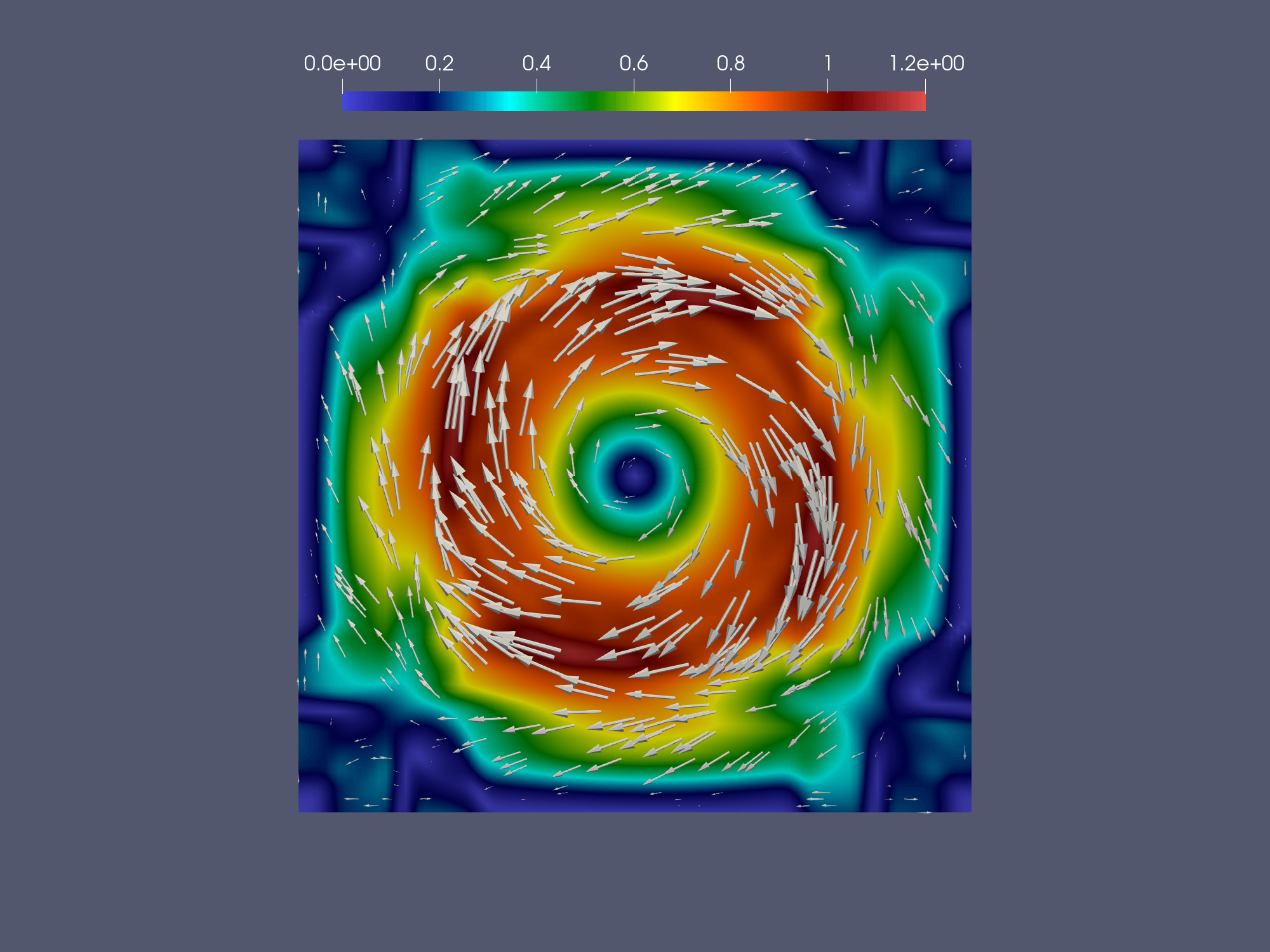} & \includegraphics[valign=c,trim={471bp 221bp 471bp 221bp},clip,width=68.25bp]{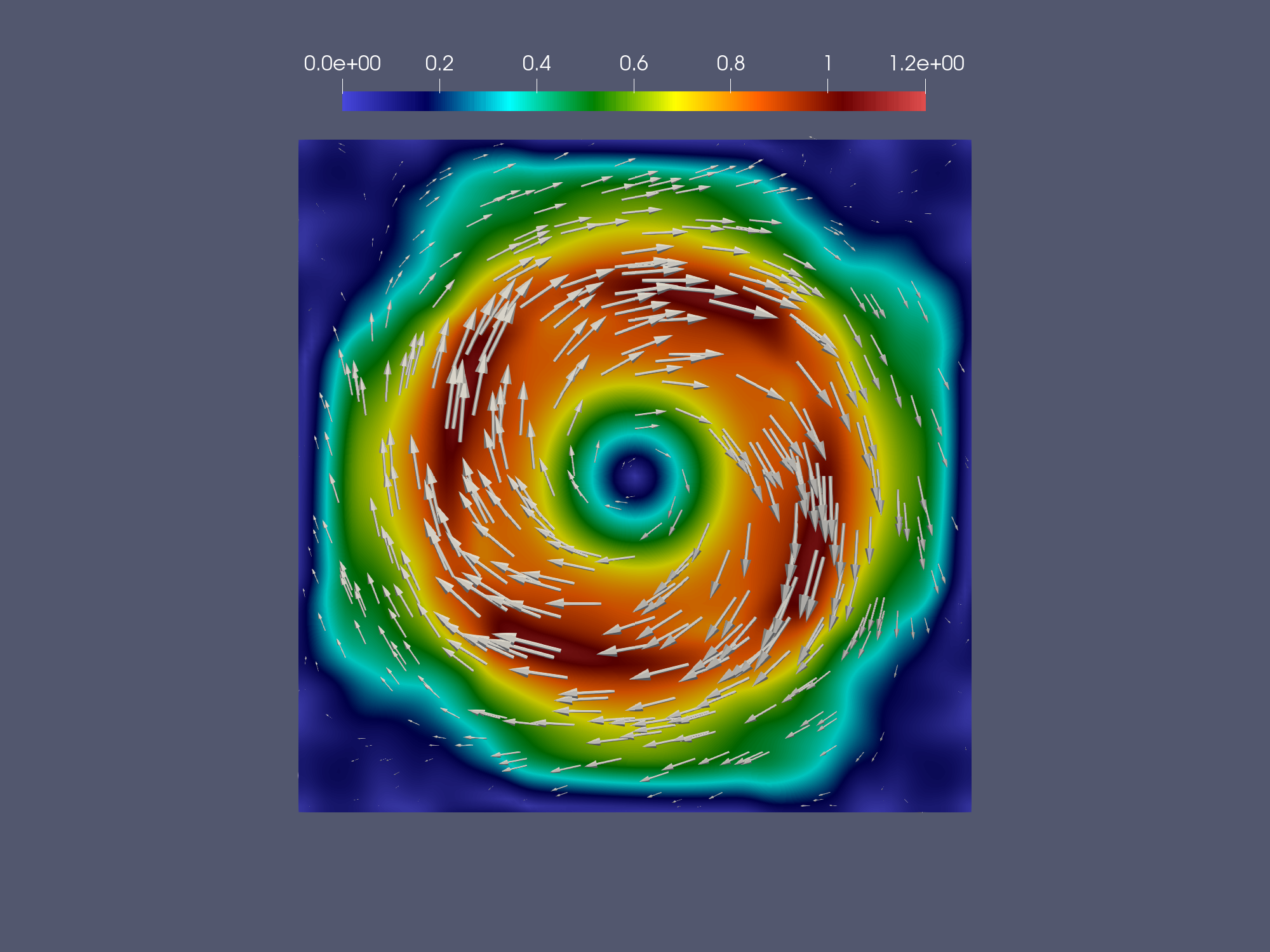} & \includegraphics[valign=c,trim={471bp 221bp 471bp 221bp},clip,width=68.25bp]{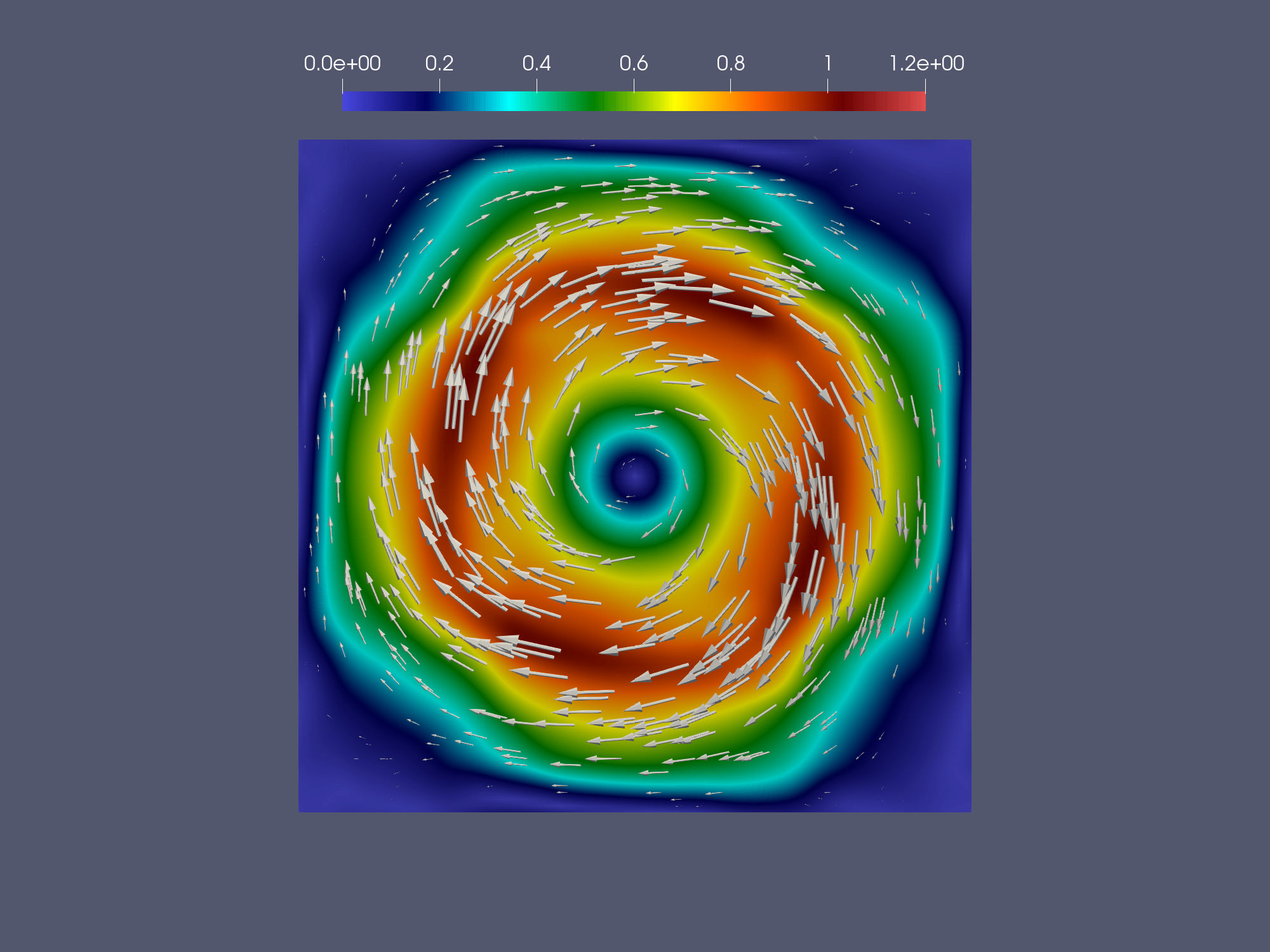} & \includegraphics[valign=c,trim={471bp 221bp 471bp 221bp},clip,width=68.25bp]{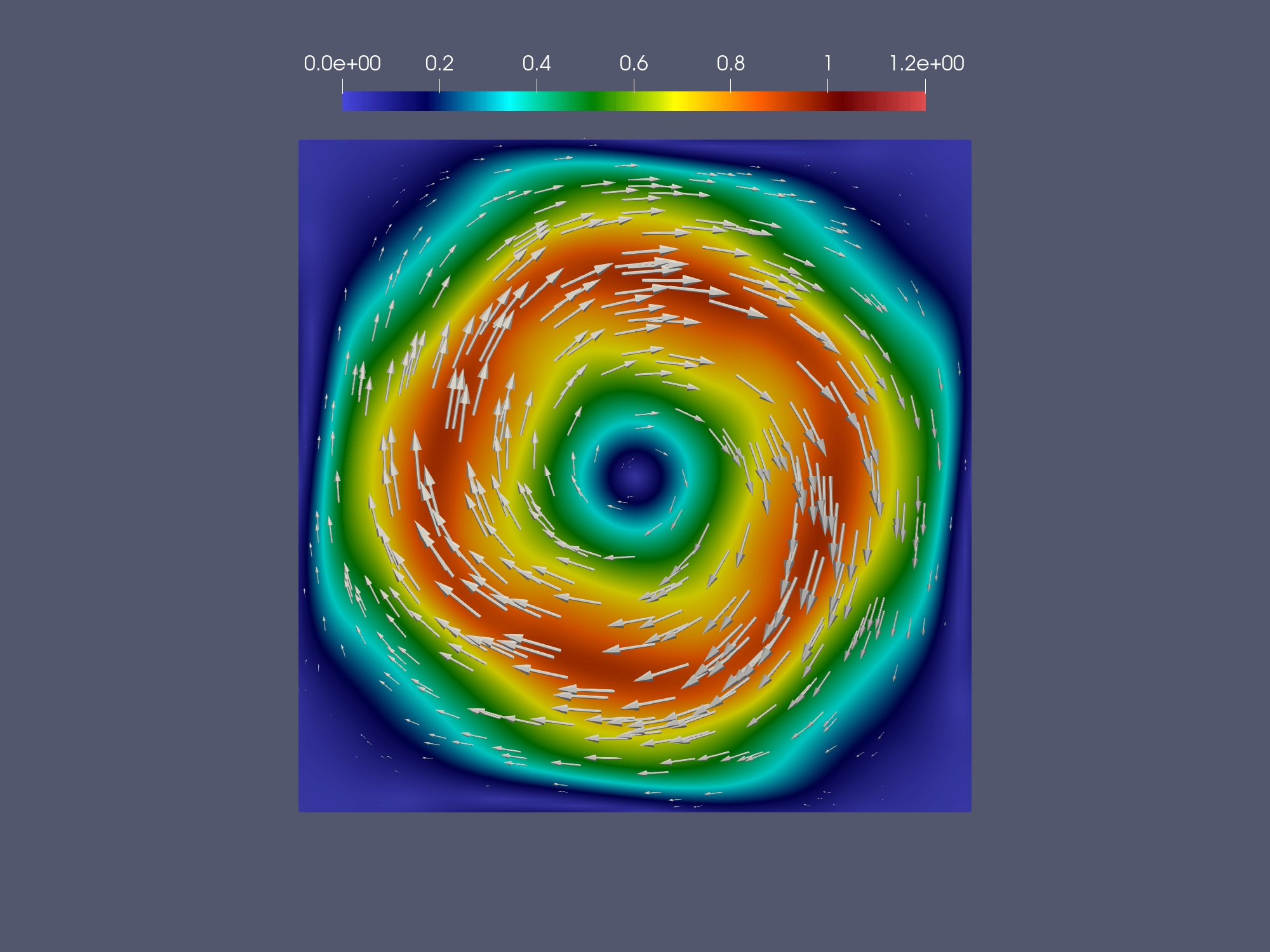} \\ \noalign{\vskip 2pt}
\end{tabular}%
}
\caption{Structure-preserving test: snapshots of the temperature $\vtheta$ and velocity $\u$ for Scheme~I (semi-implicit Euler) and Scheme~II (CPG).}
\label{fig:velotemp}

\end{figure}

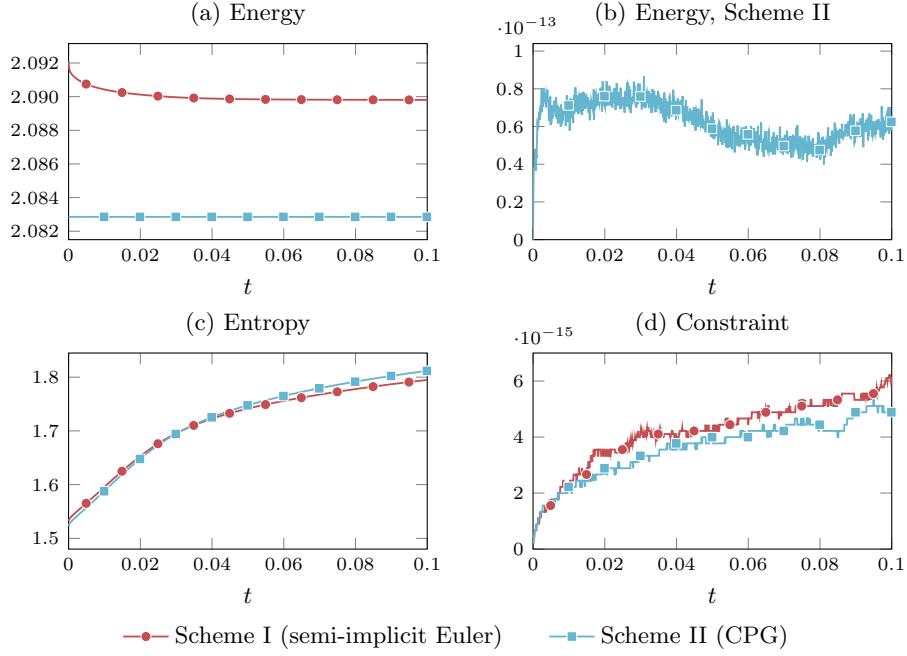
\begin{figure}[!htbp]
\centering
\begin{tikzpicture}
\begin{groupplot}[
    group style={group size=2 by 2, horizontal sep=1.4cm, vertical sep=1.5cm},
    width=0.5\textwidth, height=0.33\textwidth,
    xmin=0, xmax=0.1, xlabel={$t$},
    xtick={0,0.02,0.04,0.06,0.08,0.1},
    scaled x ticks=false,
    xticklabel style={/pgf/number format/fixed, /pgf/number format/precision=2},
    tick label style={font=\scriptsize},
    every y tick scale label/.style={font=\scriptsize, at={(yticklabel cs:1)}, anchor=south west},
    label style={font=\small},
    title style={font=\small, yshift=-3pt},
]
\nextgroupplot[title={(a) Energy}, ymin=2.0815, ytick={2.082,2.084,...,2.092}, yticklabel style={/pgf/number format/fixed, /pgf/number format/fixed zerofill, /pgf/number format/precision=3}, legend to name=diaglegend, legend columns=2, legend style={draw=none, font=\small, /tikz/every even column/.append style={column sep=1.5em}}]
    \addplot[seabornred, line width=0.7pt, mark=*, mark size=1.9pt, mark repeat=100, mark phase=50, mark options={solid, fill=seabornred, draw=white, line width=0.3pt}] table[x=t, y=totalEnergyI] {fig_data/structure_test_diagnostics.dat};
    \addlegendentry{Scheme~I (semi-implicit Euler)}
    \addplot[seaborncyan, line width=0.7pt, mark=square*, mark size=1.8pt, mark repeat=100, mark phase=100, mark options={solid, fill=seaborncyan, draw=white, line width=0.3pt}] table[x=t, y=totalEnergyII] {fig_data/structure_test_diagnostics.dat};
    \addlegendentry{Scheme~II (CPG)}
\nextgroupplot[title={(b) Energy, Scheme~II}, ymin=0, ymax=1.04e-13]
    \addplot[seaborncyan, line width=0.7pt, mark=square*, mark size=1.8pt, mark repeat=100, mark phase=100, mark options={solid, fill=seaborncyan, draw=white, line width=0.3pt}] table[x=t, y=energyII] {fig_data/structure_test_diagnostics.dat};
\nextgroupplot[title={(c) Entropy}, ymin=1.48, yticklabel style={/pgf/number format/fixed, /pgf/number format/precision=2}]
    \addplot[seabornred, line width=0.7pt, mark=*, mark size=1.9pt, mark repeat=100, mark phase=50, mark options={solid, fill=seabornred, draw=white, line width=0.3pt}] table[x=t, y=entropyI] {fig_data/structure_test_diagnostics.dat};
    \addplot[seaborncyan, line width=0.7pt, mark=square*, mark size=1.8pt, mark repeat=100, mark phase=100, mark options={solid, fill=seaborncyan, draw=white, line width=0.3pt}] table[x=t, y=entropyII] {fig_data/structure_test_diagnostics.dat};
\nextgroupplot[title={(d) Constraint}, ymin=0, ymax=7e-15]
    \addplot[seabornred, line width=0.7pt, mark=*, mark size=1.9pt, mark repeat=100, mark phase=50, mark options={solid, fill=seabornred, draw=white, line width=0.3pt}] table[x=t, y=constraintI] {fig_data/structure_test_diagnostics.dat};
    \addplot[seaborncyan, line width=0.7pt, mark=square*, mark size=1.8pt, mark repeat=100, mark phase=100, mark options={solid, fill=seaborncyan, draw=white, line width=0.3pt}] table[x=t, y=constraintII] {fig_data/structure_test_diagnostics.dat};
\end{groupplot}
\end{tikzpicture}

\pgfplotslegendfromname{diaglegend}
\caption{Structure-preserving test: (a)~total energy, $\mathcal{E}(t)$; (b)~relative change $(\mathcal{E}(t) - \mathcal{E}(0))/\mathcal{E}(0)$ in total energy for Scheme~II only; (c)~total entropy, $\int_\Omega \varrho s$; (d)~error in the quasi-incompressibility constraint, $\|\sum_{i} V_i\rho_i - 1\|_{L^\infty(\Omega)}$.}
\label{fig:meta}
\end{figure}

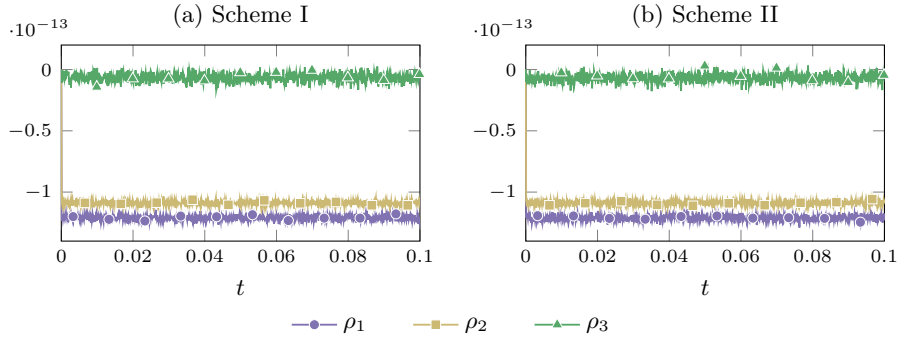
\begin{figure}[!htbp]
\centering
\begin{tikzpicture}
\begin{groupplot}[
    group style={group size=2 by 1, horizontal sep=1.4cm, vertical sep=1.5cm},
    width=0.5\textwidth, height=0.33\textwidth,
    xmin=0, xmax=0.1, xlabel={$t$},
    xtick={0,0.02,0.04,0.06,0.08,0.1},
    scaled x ticks=false,
    xticklabel style={/pgf/number format/fixed, /pgf/number format/precision=2},
    tick label style={font=\scriptsize},
    every y tick scale label/.style={font=\scriptsize, at={(yticklabel cs:1)}, anchor=south west},
    label style={font=\small},
    title style={font=\small, yshift=-3pt},
    ymin=-1.4e-13, ymax=0.2e-13,
]
\nextgroupplot[title={(a) Scheme~I}, legend to name=masslegend, legend columns=3, legend style={draw=none, font=\small, /tikz/every even column/.append style={column sep=1.5em}}]
    \addplot[seabornpurple, line width=0.7pt, mark=*, mark size=1.9pt, mark repeat=100, mark phase=34, mark options={solid, fill=seabornpurple, draw=white, line width=0.3pt}] table[x=t, y=massI1] {fig_data/structure_test_diagnostics.dat};
    \addlegendentry{$\rho_1$}
    \addplot[seabornsand, line width=0.7pt, mark=square*, mark size=1.8pt, mark repeat=100, mark phase=67, mark options={solid, fill=seabornsand, draw=white, line width=0.3pt}] table[x=t, y=massI2] {fig_data/structure_test_diagnostics.dat};
    \addlegendentry{$\rho_2$}
    \addplot[seaborngreen, line width=0.7pt, mark=triangle*, mark size=2.4pt, mark repeat=100, mark phase=100, mark options={solid, fill=seaborngreen, draw=white, line width=0.3pt}] table[x=t, y=massI3] {fig_data/structure_test_diagnostics.dat};
    \addlegendentry{$\rho_3$}
\nextgroupplot[title={(b) Scheme~II}]
    \addplot[seabornpurple, line width=0.7pt, mark=*, mark size=1.9pt, mark repeat=100, mark phase=34, mark options={solid, fill=seabornpurple, draw=white, line width=0.3pt}] table[x=t, y=massII1] {fig_data/structure_test_diagnostics.dat};
    \addplot[seabornsand, line width=0.7pt, mark=square*, mark size=1.8pt, mark repeat=100, mark phase=67, mark options={solid, fill=seabornsand, draw=white, line width=0.3pt}] table[x=t, y=massII2] {fig_data/structure_test_diagnostics.dat};
    \addplot[seaborngreen, line width=0.7pt, mark=triangle*, mark size=2.4pt, mark repeat=100, mark phase=100, mark options={solid, fill=seaborngreen, draw=white, line width=0.3pt}] table[x=t, y=massII3] {fig_data/structure_test_diagnostics.dat};
\end{groupplot}
\end{tikzpicture}

\pgfplotslegendfromname{masslegend}
\caption{Structure-preserving test: change in each partial mass, $\int_\Omega \rho_i(t) - \int_\Omega \rho_i(0)$, for (a)~Scheme~I (semi-implicit Euler) and (b)~Scheme~II (CPG).}
\label{fig:masses}
\end{figure}

\subsection{Application example} \label{sec:app_ex}

We consider a mixture of $N=2$ species, namely benzene ($\rho_1$) and cyclohexane ($\rho_2$).
The domain $\Omega$ represents a microfluidic container.
This example is based on the benzene--cyclohexane simulation in \cite{AFMV25}, which we extend here to
incorporate time dependence and thermal effects.

To describe $\Omega$, we introduce dimensionless Cartesian spatial coordinates
$(\hat{x}, \hat{y}) \in \RR^2$, so that the physical spatial coordinates are given by
$(x, y) = L_{\textrm{ref}} (\hat{x}, \hat{y})$ where $L_{\textrm{ref}} = \qty{2}{\milli\metre}$ is
a reference length scale.
The domain is depicted in \Cref{fig:app_ex}.
The benzene inlet $\partial\Omega_1$ is located at the top left
($0 \leq \hat{x} \leq 1$ and $\hat{y} = 1$) and
the cyclohexane inlet $\partial\Omega_2$ is at the bottom left
($0 \leq \hat{x} \leq 1$ and $\hat{y} = -1$).
At the right is the outlet $\partial\Omega_3$
($\hat{x}=4$ and $-0.5 \leq \hat{y} \leq 0.5$).
The remaining parts of the boundary
$\partial\Omega_4 = \partial \Omega \setminus (\partial\Omega_1 \cup \partial\Omega_2 \cup \partial\Omega_3)$
are walls.

On the benzene inlet $\partial\Omega_1$ we prescribe a non-zero benzene mass flux and set the cyclohexane mass flux to zero;
the roles are reversed on the cyclohexane inlet $\partial\Omega_2$.
Specifically, given inflow fluxes $\bm{G}_{i,k}$, $i,k=1,2$, we impose on the total mass flux $\rho_i\u + \bm{J}_i$ of each species \eqref{eq:1.flux}
\begin{equation}\label{eq:bc_rhoi}
    \big(\rho_i\u + \bm{J}_i\big)\cdot\bm{n} = \bm{G}_{i,k}\cdot\bm{n}
    \quad \text{on } \partial\Omega_k, \qquad i,k=1,2,
\end{equation}
together with the consistent condition on the barycentric velocity
\begin{equation}\label{eq:bc_u}
    \u = \frac{1}{\varrho}\sum_{i=1}^2 \bm{G}_{i,k}
    \quad \text{on } \partial\Omega_k, \qquad k=1,2.
\end{equation}
The normal and tangential components of $\bm{G}_{i,k}$ are prescribed independently: only the normal component appears in \eqref{eq:bc_rhoi}, whereas both appear in \eqref{eq:bc_u}.
We take the parabolic inflow profiles
\begin{equation}\label{eq:inflow_profiles}
    \bm{G}_{1,1} = \frac{2v_{1,\textrm{ref}}}{V_1}\,\hat{x}(1 - \hat{x})\begin{bmatrix} 2 \\ -1 \end{bmatrix},
    \qquad
    \bm{G}_{2,2} = \frac{2v_{2,\textrm{ref}}}{V_2}\,\hat{x}(1 - \hat{x})\begin{bmatrix} 2 \\ 1 \end{bmatrix},
\end{equation}
and $\bm{G}_{1,2} = \bm{G}_{2,1} = \bm{0}$, where $v_{i,\textrm{ref}}$ is a characteristic inflow velocity of species~$i$, specified below.
On the outlet $\partial\Omega_3$ we impose a \emph{do-nothing} condition: zero normal diffusive fluxes $\bm{J}_i\cdot\bm{n} = 0$, $i=1,2$, and zero traction $\bsig\cdot\bm{n} = \bm{0}$.
On the walls $\partial\Omega_4$ we impose $\bm{J}_i\cdot\bm{n} = 0$, $i=1,2$, and no slip, $\u = \bm{0}$, so that the total mass fluxes $(\rho_i\u + \bm{J}_i)\cdot\bm{n}$ vanish.

For heat transport, we impose on the inlets a Robin-type condition on the total entropy flux $\varrho s\u + \bm{J}_s$ (\ref{eq:1.flux},~\ref{eq:entropy_flux_def}),
\begin{equation}\label{eq:bc_T}
    \vtheta \big(\varrho s\u + \bm{J}_s \big)\cdot\bm{n}
    = - \, k_{\text{inlet}} \frac{\vtheta_{\text{inlet}} - \vtheta}{l_{\text{inlet}}},
\end{equation}
where $k_{\text{inlet}}$ is an effective thermal conductivity of the inlet, $l_{\text{inlet}}$ an effective thickness, and $\vtheta_{\text{inlet}}$ the ambient temperature directly outside the inlet.
With the parameters chosen below, \eqref{eq:bc_T} effectively enforces $\vtheta \approx \vtheta_{\text{inlet}}$ on the inlets.
On the walls and the outlet we impose $\nabla\vtheta\cdot\bm{n} = 0$.

\subsubsection*{Thermodynamic description of benzene and cyclohexane.}
Let $m_i$ denote the molar mass of species $i$.
To express the free energy of mixing, we introduce the mole fraction of species $i$,
\begin{equation}\label{eq:density_mole_convert}
    x_i = \frac{\rho_i / m_i}{\sum_{j=1}^N \rho_j / m_j}, \qquad i = 1, \ldots, N,
\end{equation}
and its partial molar heat capacity $c_{p,i}$, which we assume to be constant.
To capture the non-ideal contribution to the free energy of mixing, we employ a Margules model \cite{PerryBook}, with two dimensionless constants $A_{12}, A_{21} \in \RR$.
Accordingly, we take
\begin{multline}\label{eq:app_free_energy}
    \varrho F = R\vtheta\frac{\rho_1}{m_1}\big(\log(x_1) + x_2^2(A_{12} + 2(A_{21}-A_{12})x_1)\big)  \\
    + R\vtheta\frac{\rho_2}{m_2}\big(\log(x_2) + x_1^2(A_{21} + 2(A_{12}-A_{21})x_2)\big)  \\
    + \left(c_{p,1}\frac{\rho_1}{m_1}+c_{p,2}\frac{\rho_2}{m_2}\right)(\vtheta-\vtheta\log(\vtheta)),
\end{multline}
which yields the chemical potentials, entropy and internal energy
\begin{subequations}\label{eq:app_thermo}
\begin{align}
    \mu_1 &= \frac{R\vtheta}{m_1}\big(\log(x_1) + x_2^2(A_{12} + 2(A_{21}-A_{12})x_1)\big) \label{eq:app_mu1}  \\
    &\qquad\qquad\qquad\qquad\qquad\qquad\qquad\qquad\quad + \frac{c_{p,1}}{m_1}(\vtheta-\vtheta\log(\vtheta)) + V_1p, \notag  \\
    \mu_2 &= \frac{R\vtheta}{m_2}\big(\log(x_2) + x_1^2(A_{21} + 2(A_{12}-A_{21})x_2)\big) \label{eq:app_mu2}  \\
    &\qquad\qquad\qquad\qquad\qquad\qquad\qquad\qquad\quad + \frac{c_{p,2}}{m_2}(\vtheta-\vtheta\log(\vtheta)) + V_2p, \notag  \\
    \varrho s &= -R\frac{\rho_1}{m_1}\big(\log(x_1) + x_2^2(A_{12} + 2(A_{21}-A_{12})x_1)\big) \label{eq:app_entropy}  \\
    &\qquad\qquad\qquad\quad - R\frac{\rho_2}{m_2}\big(\log(x_2) + x_1^2(A_{21} + 2(A_{12}-A_{21})x_2)\big)  \notag  \\
    &\qquad\qquad\qquad\qquad\qquad\qquad\qquad\qquad\quad + \left(c_{p,1}\frac{\rho_1}{m_1}+c_{p,2}\frac{\rho_2}{m_2}\right)\log(\vtheta), \notag  \\
    \varrho e &= \left(c_{p,1}\frac{\rho_1}{m_1}+c_{p,2}\frac{\rho_2}{m_2}\right)\vtheta. \label{eq:app_energy}
\end{align}
\end{subequations}
For the mobility $\M = (M_{ij})$ and heat conductivity $K$ we take
\begin{equation}\label{eq:mobility_heat_cond}
    \M = \frac{\mathscr{D}_{12} \rho_1 \rho_2 m_1 m_2}{R \varrho (m_1 x_1 + m_2 x_2)}
    \begin{pmatrix}
        1 & -1 \\ -1 & 1
    \end{pmatrix},
    \qquad K = (k_1 x_1 + k_2 x_2) \vtheta^2.
\end{equation}
Here $\mathscr{D}_{12}$ is the binary Maxwell--Stefan diffusivity of the mixture, and the formula for $\M$ follows from inverting the Maxwell--Stefan relations; $k_i$ denotes the thermal conductivity of pure species~$i$.
Note that $K$, as it appears in \eqref{eq:1.flux}, differs from the usual thermal conductivity by a factor of $\vtheta^2$, which accounts for the use of Fourier's law.
The material parameters are listed in \Cref{tab:app_params}.

\begin{table}[htbp!]
\centering
\small
\caption{Application example: material parameters of benzene and cyclohexane.}
\label{tab:app_params}
\begin{tabular}{|l|c|c|c|}
    \hline
    Quantity & Benzene ($i=1$) & Cyclohexane ($i=2$) & Unit \\
    \hline
    Molar mass $m_i$ & $0.078$ & $0.084$ & \unit[per-mode=symbol]{\kg\per\mol} \\
    Mass density $1/V_i$ & $876$ & $773$ & \unit[per-mode=symbol]{\kg\per\metre\cubed} \\
    Molar heat capacity $c_{p,i}$ & $135.69$ & $156.00$ & \unit[per-mode=symbol]{\joule\per\kelvin\per\mol} \\
    Thermal conductivity $k_i$ & $0.14$ & $0.11$ & \unit[per-mode=symbol]{\watt\per\kelvin\per\metre} \\
    \hline
    Margules constants $A_{12}$, $A_{21}$ & \multicolumn{2}{c|}{$0.4498$, $0.4952$} & -- \\
    Maxwell--Stefan diffusivity $\mathscr{D}_{12}$ & \multicolumn{2}{c|}{\num{2.1e-9}} & \unit[per-mode=symbol]{\metre\squared\per\second} \\
    Viscosities $\eta$, $\lambda$ & \multicolumn{2}{c|}{\num{6e-4}, \num{1e-7}} & \unit[per-mode=symbol]{\Pa\second} \\
    Gas constant $R$ & \multicolumn{2}{c|}{$8.314$} & \unit[per-mode=symbol]{\joule\per\kelvin\per\mol} \\
    Inlet thermal conductivity $k_{\text{inlet}}$ & \multicolumn{2}{c|}{$0.0125$} & \unit[per-mode=symbol]{\watt\per\kelvin\per\metre} \\
    Inlet thickness $l_{\text{inlet}}$ & \multicolumn{2}{c|}{$0.02$} & \unit[per-mode=symbol]{\milli\metre} \\
    \hline
\end{tabular}
\end{table}

The heat capacities of benzene and cyclohexane are taken from \cite{grolier1993heat} and \cite{trejo1991excess} respectively, and the thermal conductivities from \cite{tanaka1988thermal}.
The values of $k_{\text{inlet}}$ and $l_{\text{inlet}}$ are hypothetical, corresponding roughly to inlet boundaries that are ten times more thermally insulating than the mixture and one hundred times thinner than the reference length $L_{\textrm{ref}}$.
All other material parameters are taken from \cite{AFMV25}.
We take the ambient inlet temperature to be
\begin{equation}\label{eq:temp_bcs}
    \vtheta_{\text{inlet}} = \begin{cases*}
        1.1 \cdot \qty{298.15}{\kelvin}, &on $\partial\Omega_1$ (benzene inlet), \\
        1.0 \cdot \qty{298.15}{\kelvin}, &on $\partial\Omega_2$ (cyclohexane inlet),
    \end{cases*}
\end{equation}
and the inlet reference velocities to be
\begin{equation}\label{eq:ref_velocities}
    v_{1,\textrm{ref}} = \qty{0.4}{\micro\metre\per\second}, \qquad
    \frac{v_{1,\textrm{ref}}}{m_1 V_1} = \frac{v_{2,\textrm{ref}}}{m_2 V_2},
\end{equation}
where the choice of $v_{2,\textrm{ref}}$ symmetrises the molar fluxes at the inlets \cite{AFMV25}.
The initial partial densities correspond to an equimolar mixture, $x_1 = x_2 = 0.5$ in \eqref{eq:density_mole_convert}, and the initial temperature is $\vtheta^0 = \qty{298.15}{\kelvin}$.

\subsubsection*{Non-dimensionalisation.}
The equations are solved in non-dimensionalised form.
For this mixture, the Lagrange number, as discussed in \Cref{rem:Me}, is
\begin{equation}\label{eq:pressure_diffusion_number}
    \textrm{La} \coloneqq \frac{2R\vtheta_{\mathrm{ref}}L_{\mathrm{ref}}}{\eta\, v_{1,\textrm{ref}}}
    \left(\frac{1}{V_1 m_1}+\frac{1}{V_2 m_2}\right),
\end{equation}
with reference temperature $\vtheta_{\mathrm{ref}} = \qty{298.15}{\kelvin}$, so that $\textrm{La} \approx 10^{15}$.
Since the free energy \eqref{eq:app_free_energy} is homogeneous of degree one in $\vec\rho$, we remove the term \eqref{eq:thermo_forcing} from the momentum equation in both schemes.
Together with the open inlets and outlet (through which mass, energy and entropy enter and leave the domain), this means that the global mass, energy and entropy structures of \Cref{thm:schemeA_global,thm:schemeB_global} do not apply to this example as stated;
it serves instead to demonstrate the schemes in a physically realistic setting.

\subsubsection*{Simulation results.}
We simulate up to the final time $\qty{6.25}{\hour}$, although the solution is approximately stationary after roughly $\qty{3}{\hour}$, using both schemes with spatial polynomial degree $\ell=1$ in \eqref{eq:fe_spaces} and, for Scheme~II, $k=1$ in \eqref{eq:scheme_ii_pg}.
The two schemes give results in good agreement.
\Cref{fig:app_ex} shows the solution computed by Scheme~I at the final time.
The temperature field is distributed as expected from the inlet temperatures \eqref{eq:temp_bcs}, and the streamlines of the benzene velocity $\u + \bm{J}_1/\rho_1$ agree with the steady-state benzene streamlines reported in \cite{AFMV25}.

\begin{figure}[!htbp]
    \centering
    \includegraphics[width=\linewidth]{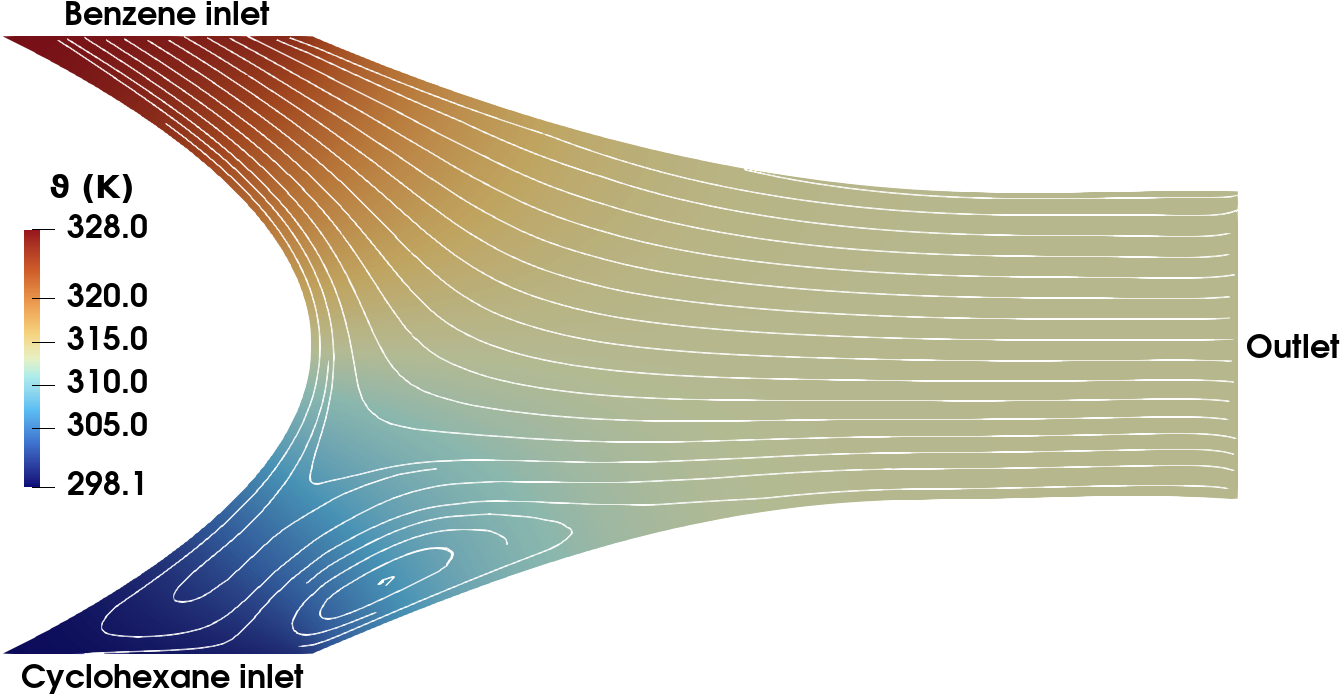}
    \caption{The microfluidic container $\Omega$ of \Cref{sec:app_ex}, with the solution computed by Scheme~I at the final time $\qty{6.25}{\hour}$. The domain is coloured by the temperature $\vtheta$, and streamlines of the benzene velocity $\u + \bm{J}_1/\rho_1$ are shown in white.}
    \label{fig:app_ex}
\end{figure}

\bibliographystyle{amsplain}
\bibliography{lit}

\end{document}